\documentclass[fleqn,10pt]{article}
\usepackage{latexsym, graphicx, epsfig, amsmath, amssymb,amsfonts}
\usepackage{natbib,amsthm,version}
\usepackage{amsbsy,bm,multirow}
\usepackage[titletoc,page]{appendix}
\usepackage[mathscr]{eucal}
\usepackage{mathtools}
\usepackage{color}
\usepackage[utf8]{inputenc}
\usepackage[english]{babel}
\usepackage{amsthm}
\usepackage[hidelinks]{hyperref}
\usepackage{url}
\usepackage[lined,boxed,linesnumbered,ruled]{algorithm2e}
\usepackage{float}
\usepackage{arydshln}
\usepackage{dsfont,bbm}
\usepackage{courier}
\usepackage{caption,subcaption}
\usepackage{epstopdf}
\usepackage{mathabx}

\usepackage[shortlabels]{enumitem}
\setlist{nolistsep}

\usepackage{natbib}
\usepackage{pgfplots}
\pgfplotsset{compat=1.17}
\usetikzlibrary{arrows.meta}

\newcommand{\mbs}[1]{\mathbf{#1}}
\newcommand{\mbb}[1]{\mathbb{#1}}

\newtheorem{theorem}{Theorem}[section]

\newtheorem{remark}{Remark}[section]

\theoremstyle{definition}

\title{
  A Novel Method for Enforcing Exactly Dirichlet, Neumann and Robin Conditions
  on Curved Domain Boundaries for Physics Informed Machine Learning
} 
\author{
  Suchuan Dong\thanks{Author of correspondence. Email: sdong@purdue.edu},
  \  Yuchuan Zhang
  \\
  Center for Computational and Applied Mathematics \\
  Department of Mathematics \\
  Purdue University, USA 
 } 

\date{(\today)}

\begin{document}
\maketitle


\begin{abstract}

  We present a systematic method for exactly enforcing Dirichlet, Neumann,
  and Robin type conditions on general quadrilateral domains with arbitrary
  curved boundaries. Our method is built upon exact mappings between general
  quadrilateral domains and the standard domain, and employs a combination
  of TFC (theory of functional connections) constrained expressions and
  transfinite interpolations. 
  When Neumann or Robin boundaries are present,
  especially when two Neumann (or Robin) boundaries meet at a vertex,
  it is critical to enforce exactly the induced compatibility constraints at
  the intersection, in order to enforce exactly the imposed conditions on the
  joining boundaries. We analyze in detail and present constructions
  for handling the imposed boundary conditions and the induced compatibility
  constraints for two types of situations: (i) when Neumann (or Robin) boundary
  only intersects with Dirichlet boundaries, and (ii) when two Neumann (or Robin)
  boundaries intersect with each other. 
  We describe a four-step procedure to systematically formulate the general
  form of functions that exactly satisfy the imposed Dirichlet, Neumann, or Robin
  conditions on general quadrilateral domains.
  The method developed herein has been implemented together with the extreme learning
  machine (ELM) technique we have developed recently for scientific
  machine learning.
  Ample numerical experiments are presented with several linear/nonlinear
  stationary/dynamic problems on a variety of two-dimensional domains with complex boundary
  geometries. Simulation results demonstrate that the proposed method has
  enforced the Dirichlet, Neumann, and Robin conditions on curved domain boundaries
  exactly, with the numerical boundary-condition errors
  at the machine accuracy.

\end{abstract}


\vspace{0.05cm}



\section{Introduction}
\label{sec_intro}



Artificial neural networks (NN) have garnered remarkable success in diverse fields
of science and engineering~\cite{LeCun2015DP,GoodfellowBC2016}.
These advances have catalyzed the development and adoption of neural network-based techniques
for scientific computing. 
As universal function approximators, neural networks are natural for the
ansatz space for solving ordinary or partial differential equations (ODE/PDE).
This underpins their use in mathematical modeling and scientific computing,
and has fueled the advancement of scientific machine learning~\cite{Karniadakisetal2021}.


The use of NNs for computing ODEs/PDEs
dates back to the 1990s
(see~\cite{LeeK1990,MeadeF1994,MeadeF1994b,DissanayakeP1994,YentisZ1996,LagarisLF1998}).
Leveraging the universal approximation
property~\cite{Cybenko1989,HornikSW1989,HornikSW1990},
NN-based methods typically transform the PDE solution problem into
an optimization problem, thanks to the residual minimization theorem as
elaborated in~\cite{Jiang1998}.
The PDE and its boundary and initial conditions (BC/IC) are encoded
into a cost/loss function by penalizing their residual norms
on a set of sampling points~\cite{Eason1976}.
This general residual minimization technique~\cite{Jiang1998} is presently
often known as the physics-informed approach for solving PDEs.
The differential operators involved therein
are often computed via
automatic differentiation.
The optimization,
usually through gradient descent-type algorithms,
constitutes the core computations in NN-based PDE solvers,
commonly known as the network training.
After training, the NN  parameters effectively
encode the PDE solution.

Prominent advancements of this area in recent years include the development of
physics-informed neural network (PINN) method~\cite{RaissiPK2019}
and sister approaches such as the
deep Galerkin method (DGM)~\cite{SirignanoS2018} and deep Ritz method~\cite{EY2018},
as well as related methods such as the weak adversarial network~\cite{zang2020weak},
Galerkin neural network~\cite{AinsworthD2021},
deep Nitsche method~\cite{LiaoW2021}, deep mixed residual method~\cite{LyuZCC2022}, 
along with many variant techniques (see e.g.~\cite{LiTWL2020,JagtapKK2020,CyrGPPT2020,WangL2020,WangYP2020,lu2021deepxde,KrishnapriyanGZKM2021,DuZ2021,TangWL2021,GaoZW2022,Penwardenetal2023,McClennyB2023,ZhangZZZ2023,BrunaPV2024,WangL2024,AldiranyCLP2024,WangSP2024}, among others).
Another solution approach for PDEs, usually in high dimensions,
involves reformulating them using stochastic differential equations,
exemplified by the deep backward stochastic differential
equation (Deep BSDE)~\cite{EHJ2017,han2018solving}, the forward-backward
stochastic neural network method~\cite{Raissi2018},
and related techniques~\cite{ZengCZ2022,lu2021priori,weinan2021algorithms}.
The above methods have been applied across a wide range of fields.
Comprehensive reviews of these developments can be found
in~\cite{Karniadakisetal2021,Cuomoetal2022,BeckHJK2022}.


Enforcing boundary conditions is one of the central technical issues in
physics informed machine learning.
Unlike classical numerical methods, standard NN ansatzes are non-interpolatory
and do not automatically satisfy prescribed traces or fluxes on the boundary.
As a result, a substantial literature has developed on how to enforce
Dirichlet, Neumann, Robin, and periodic boundary conditions in PINN and
related neural PDE solvers. Existing approaches can be classified into four
broad categories: (i) soft or penalty-based enforcement, (ii) exact or hard
enforcement through trial function design, (iii) weak enforcement in variational
formulations, and (iv) multiplier-based constrained formulations.

The original PINN adopts soft enforcement, in which the PDE residual is minimized
together with BC residual terms sampled on the boundary \cite{RaissiPK2019}.
Dirichlet data are imposed through value penalties, while Neumann and Robin data
are enforced through penalties on normal derivatives or mixed value--flux expressions.
Periodic BCs are commonly handled in the same spirit, by penalizing  the
solution difference and, when required, selected derivatives on paired periodic boundaries.
This strategy is simple and broadly applicable, but BC satisfaction is only approximate
and training can be highly sensitive to the relative weights of the PDE and BC loss terms.
This difficulty has been analyzed from the perspective of gradient-flow
pathologies in PINNs \cite{Wang2021SISC};
see also \cite{RowanHMD2025} for a recent comparative study of BC-enforcement strategies
on  three-dimensional geometries.

A second major line of work seeks exact or hard enforcement, in which the neural
approximation is constructed to satisfy the BCs identically. This idea goes back
to the early works of Lagaris, McFall and collaborators,
who introduced neural trial functions of the
form $u_\theta=g+\phi N_\theta$, with $g$ satisfying the boundary data and
$\phi$ vanishing on the constrained boundary \cite{LagarisLF1998,LagarisLP2000,McFallM2009}.
In more recent literature, 
it has been shown that boundary and initial conditions
can be embedded directly into the ansatz, thereby eliminating BC penalty terms
from the loss in forward and inverse
settings \cite{ShengY2021,Lyu2021CSIAM,Lu2021SISC,Liuetal2022,RoyC2024,LaiSYYZ2025}.

Among exact-enforcement methods, the geometry-aware technique of
Sukumar and Srivastava is particularly influential \cite{Sukumar2022CMAME}.
Using approximate distance functions (ADF) and the theory of R-functions,
they constructed admissible neural trial spaces that can enforce Dirichlet,
Neumann, and Robin conditions exactly. 
A comparative study of penalty, output-modification, distance-function, and Nitsche-type approaches
has been conducted in \cite{Berrone2023Heliyon}, which concludes 
that exact output modification is generally superior to penalty-only
training for Dirichlet conditions. A number of follow-on works extend or specialize
this perspective; see e.g.~\cite{WangMIK2023} for exact Dirichlet
enforcement in solid mechanics and \cite{Li2024CAMWA} for hybrid hard/soft Fourier-based
treatment in advection--diffusion problems.

Although the distance-function framework of~\cite{Sukumar2022CMAME} is elegant, 
this comes with certain limitations.
This method is especially effective for Dirichlet conditions.
Its extension to Neumann and Robin conditions, however, is more delicate because
derivative boundary operators require higher regularity of the trial space and of
the underlying geometric representation. In particular, follow-on work has noted that
exact enforcement using approximate distance functions becomes more challenging for
higher-order PDEs, and recent studies have further pointed out that strong Neumann/Robin
constructions may become unstable when boundary segments are only piecewise $C^1$
rather than globally $C^1$ \cite{GladstoneNSS2025,GoschelGR2025,SukumarR2026}.
A further subtlety concerns vertices or corner points where two boundary segments meet.
The true Euclidean distance function is generally not differentiable at such points,
which is precisely why the method relies on approximate distance functions constructed
through R-functions and related smooth implicit representations.
Accordingly, differentiability of the resulting trial function at a vertex is not
automatic in a classical geometric sense; rather, it depends on the regularity built into
the approximate distance construction. For Dirichlet conditions this is often sufficient in
practice, since exact trace satisfaction is the primary requirement. For Neumann and Robin
conditions, however, the issue is more fundamental: at a corner where two Neumann
or Robin boundaries meet the outward normal is typically
not uniquely defined. So the boundary operator itself becomes ambiguous unless additional
compatibility constraints 
are introduced. Thus, while the method is a powerful hard-constraint strategy
its application to Neumann and Robin conditions
on nonsmooth geometries requires additional geometric and analytical care, and corner
singularities remain a genuine limitation rather than a purely technical
detail \cite{GoschelGR2025}.

Periodic BCs have motivated a distinct line of research for exact-enforcement techniques.
A key work here is~\cite{DongN2021}, which introduced periodic layers that can be embedded
into feed-forward networks to impose exactly $C^{\infty}$-periodic or $C^k$-periodic boundary
conditions. This construction plays  a role for periodic BCs analogous to that played
by distance functions for Dirichlet-type constraints: periodicity becomes an architectural
property of the ansatz rather than a penalty term in the loss. See also~\cite{Li2025CPC} for a
recent structure-preserving extension using embedded periodic boundary layers in
geometric-flow problems.

A related strategy for exact enforcement is based on the Theory of Functional
Connections (TFC)~\cite{Mortari2017,MortariL2019}.
In TFC, the boundary or initial
conditions are embedded analytically into a constrained expression, leaving the neural
network to represent only the free function \cite{Leake2020MAKE,Schiassietal2021}.
%
While TFC provides an effective mechanism for
BC enforcement by analytically embedding linear equality constraints,
its main limitation is geometric flexibility. In its 
multivariate form, TFC is most natural on tensor-product domains such as rectangles
and hyperrectangles, where the constrained expressions can be built in separable
coordinates \cite{MortariL2019,Leake2020MAKE}.
Non-rectangular domains generally require additional bijective
mappings to a rectangular domain, together with either an  inverse
map or an approximation thereof (see \cite{MortariA2020}).
In addition, the constrained expressions can become increasingly
cumbersome in higher dimensions, with the number of TFC terms growing
exponentially~\cite{WangD2024}.
Recent reduced-TFC work~\cite{ThiruthummalSK2024} explicitly motivates itself by
improved efficiency and the ability for more complex boundary
geometries.

A third category is weak enforcement in variational neural PDE techniques.
In the Deep Ritz method and in variational PINNs, the PDE is enforced through an
energy or weak residual rather than through strong-form
collocation \cite{E2018CMS,Kharazmi2021CMAME}. In this setting, Neumann conditions
often arise naturally through integration by parts, whereas Dirichlet conditions remain
essential constraints that must be imposed separately. The Nitsche's
method was adapted to this setting and  a Deep Nitsche method was developed for essential
BCs in \cite{LiaoW2021}.

A fourth  category consists of multiplier-based constrained
formulations. Rather than enforcing BCs by fixed penalties, these methods introduce
auxiliary Lagrange multipliers or saddle-point formulations. Makridakis et al.\ recently
proposed a Deep Uzawa approach for BC enforcement in PINNs and Deep Ritz
methods \cite{MakridakisPP2024}. For  Neumann conditions,
specialized architectural variants are also beginning to emerge;
see e.g.~\cite{StraubBMR2025} for a recent hard-constraint
treatment based on embedded Fourier features.



In the current paper we present a systematic method for enforcing exactly Dirichlet, Neumann
and Robin type conditions on general quadrilateral domains with arbitrary curved boundaries.
This method is based on exact mappings between general quadrilateral domains and
the standard domain, and leverages
a combination of TFC constrained expressions~\cite{Mortari2017}
and the transfinite interpolations developed by Gordon
and collaborators~\cite{Gordon1971,GordonH1973} for both the domain mapping and
the trial-function formulation.
The resultant trial ansatzs are in parametric forms, formulated in terms of
the standard domain.
%
%
The formulation for exactly enforcing Dirichlet BCs on general quadrilateral
domains is conceptually straightforward, once the exact domain mapping is achieved.
As a matter of fact, the mapping problem between a general quadrilateral domain and the standard
domain itself is treated as a problem involving solely Dirichlet boundary conditions,
and is formulated by a combination of TFC constrained expression
and transfinite interpolation.
%
When Neumann or Robin boundaries are present over the domain,
especially when two Neumann (or Robin) boundaries intersect with each other,
the formulation becomes considerably more challenging.
In this case, the TFC constrained expression and the transfinite interpolation
need to be modified in order to exactly enforce not only these conditions, but also
the compatibility constraints at the intersecting vertex induced by
these conditions. Enforcing exactly the induced compatibility constraints
at the intersection
is critical, because otherwise the Dirichlet, Neumann or Robin conditions on
the adjacent boundaries fail to be exactly satisfied. 
We analyze in detail and present formulations to handle the induced compatibility constraints
for two types of situations: (i) when Neumann (or Robin) boundaries only intersect
with Dirichlet boundaries, and (ii) when two Neumann (or Robin) boundaries intersect
with each other.
We present a four-step procedure for systematically formulating
the general form of trial functions to satisfy these conditions and their compatibility constraints.
When a combination of Dirichlet, Neumann, and Robin boundaries are
present on the general quadrilateral domain, the induced compatibility constraints
can be decomposed into those of the aforementioned cases
and the trial function can be formulated
analogously based on the four-step procedure.


The method proposed herein for exact BC enforcement
has been implemented with the extreme learning machine (ELM) technique
we have developed recently~\cite{DongL2021,DongL2021bip,DongY2022rm,NiD2023,DongW2023,WangD2024}.
ELM is a scientific machine learning approach based on randomized feedforward neural networks,
in which the hidden-layer coefficients are randomly assigned and fixed (non-trainable)
and only the output-layer coefficients are trained.
The ELM network is trained by the linear least squares method for linear
problems or by the nonlinear least squares method (Gauss-Newton method) for nonlinear
problems. There exists a sizeable volume of literature on ELM and variant techniques
(with different aliases). We refer the reader
to e.g.~\cite{PanghalK2020,DwivediS2020,CalabroFS2021,Schiassietal2021,FabianiCRS2021,ChenCEY2022,QuanH2023,FabianiGRS2023,SunDF2024,ZhangBJZ2024,FabianiKSY2025,FalcoSC2026} (among others), and the references therein,
for contributions from other researchers to this area.


Extensive numerical experiments have been conducted
using several linear/nonlinear stationary/dynamic PDEs on a variety of domains with
complex boundary geometries. Simulations demonstrate that
the ELM network together with the current method for BC enforcement has produced highly
accurate results. In particular, numerical results show that the current method has enforced
the Dirichlet, Neumann, and Robin boundary conditions on curved domain
boundaries to the machine accuracy.


The fundamental contribution of this work lies in the systematic method 
for formulating trial functions that exactly satisfy the Dirichlet, Neumann, and Robin type
conditions on general quadrilateral domains with arbitrary curved boundaries. 
The analyses and formulations for enforcing the induced compatibility constraints and the BCs,
especially when two Neumann (or Robin) boundaries intersect with each other, are particularly important.
Another contribution is the numerical demonstration of the effectiveness of
the proposed method for BC enforcement, with the numerical errors for Dirichlet, Neumann,
and Robin BCs on curved domain boundaries achieving machine accuracy.
We would like to emphasize that the method for BC enforcement presented here
can be used with other NN architectures and training algorithms (e.g.~PINNs), not limited to ELM or
randomized neural networks. Because the boundary conditions are enforced exactly with
the solution ansatz, it is agnostic to the neural network used for learning
the arbitrary free function therein.


The rest of this paper is organized as follows. In Section \ref{sec_method}
we first discuss how to map a general quadrilateral domain with arbitrary curved boundaries
to the standard domain, and then present a four-step procedure to systematically formulate
trial functions on general quadrilateral domains that exactly satisfy
the imposed Dirichlet, Neumann, and Robin boundary conditions.
The implementation of this method using the ELM technique for solving linear and
nonlinear PDEs has also been presented.
In Section~\ref{sec_tests} we present a set of numerical examples for  linear and nonlinear
boundary/initial value problems to demonstrate the effectiveness and the
performance of the proposed method for BC enforcement on a variety of domains involving
complex boundary geometries. Section~\ref{sec_summary} concludes
the presentation with a summary of key points and some further comments.
The Appendix (Section~\ref{sec_geom}) provides details about the geometric parameters
for all the computational domains used in the numerical experiments in Section~\ref{sec_tests}.


\section{Exact Enforcement of Boundary Conditions on General Quadrilateral
  Domains with Curved Boundaries}
\label{sec_method}

\subsection{Mapping General Quadrilateral Domains with Curved Boundaries}
\label{sec_21}

We consider a general quadrilateral (Quad) domain $\Omega=\overline{ABCD}$
as sketched in Figure~\ref{fg_1}(a),
whose boundaries can each be an arbitrary curve.
To represent a field function $u(\mbs x)$, $\mbs x=(x,y)\in\Omega$,
defined on this domain that exactly satisfy
the prescribed boundary conditions, it is necessary to
first consider the mapping of this domain to the standard
quadrilateral domain,
\begin{equation}\label{eq_1}
  (x,y) = \mbs x(\bm\xi) = \mbs{x}(\xi,\eta) = (x(\xi,\eta), y(\xi,\eta)),
\end{equation}
where $\bm\xi=(\xi,\eta)\in\Omega_{st}$, and
$\Omega_{st}=\overline{A'B'C'D'}=[-1,1]\times[-1,1]$ denotes the standard
 domain (see Figure~\ref{fg_1}(b)).

\begin{figure}
  \centerline{
    \begin{tabular}{ccc}
      \includegraphics[width=2.3in]{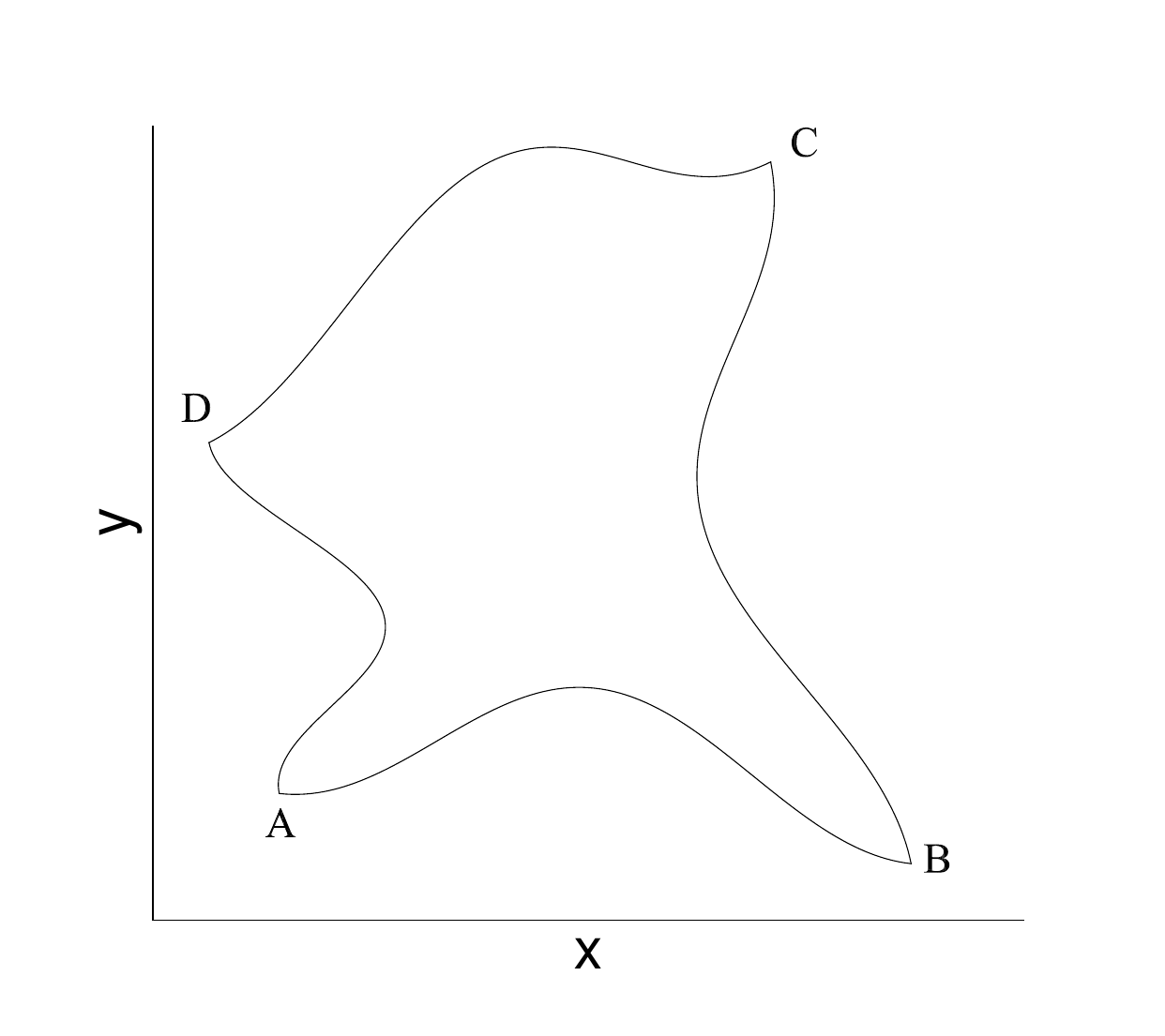}(a)
      &
      \raisebox{1in}{\Large $\xLeftrightarrow{\  \text{mapping}\ }$}
      &
      \includegraphics[width=2.3in]{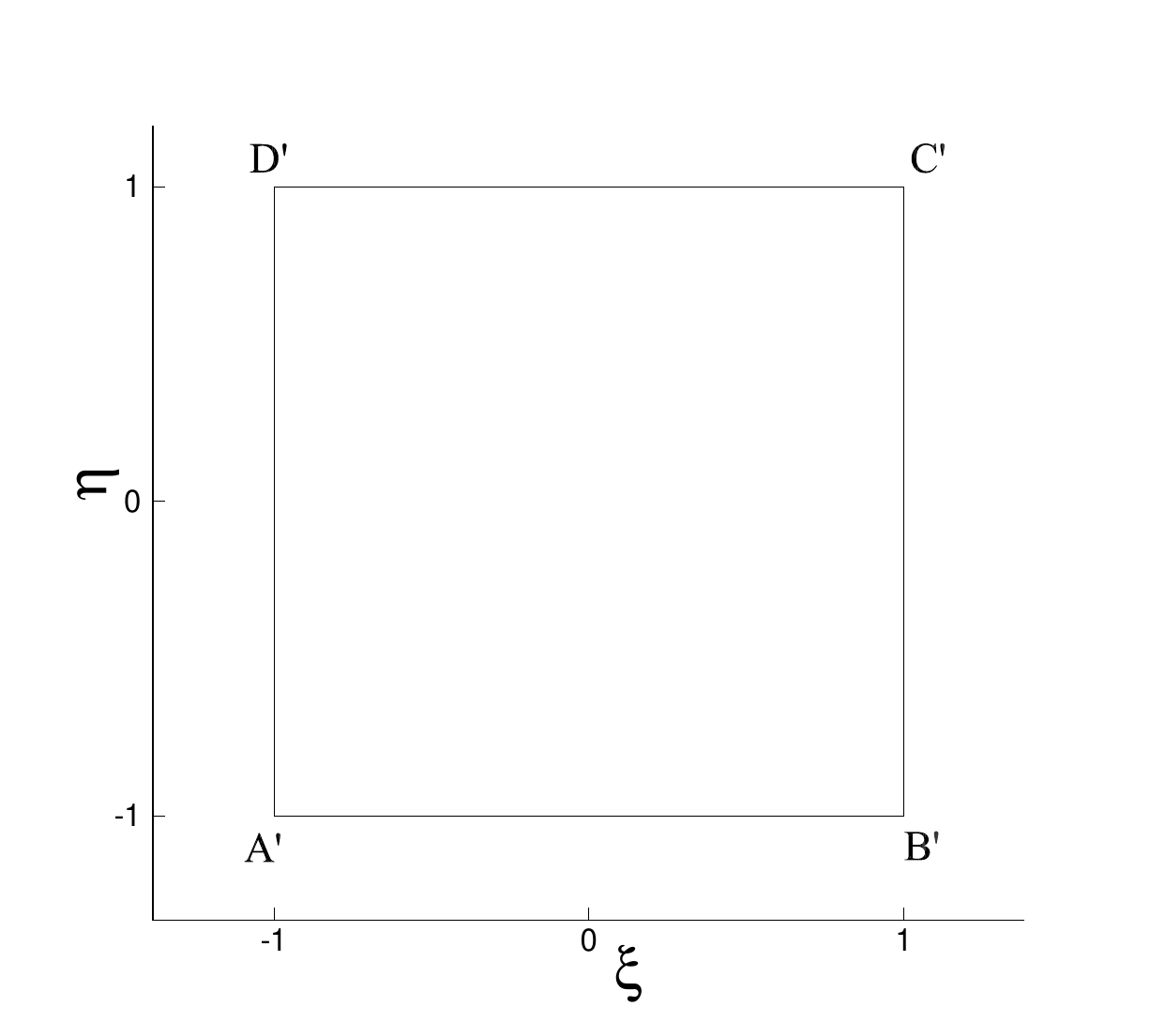}(b)
    \end{tabular}
  }
  \caption{Mapping between a general quadrilateral domain $\Omega$  and the standard
    quadrilateral domain $\Omega_{st}=[-1,1]^2$.
  }
  \label{fg_1}
\end{figure}

We assume that the map
$\mbs x(\xi,\eta)=(x(\xi,\eta),y(\xi,\eta))$ represents a regular
transformation~\cite{Fleming1977} (i.e.~of class at least $C^1$, univalent,
with non-singular Jacobian matrix).
It needs to satisfy
the following Dirichlet boundary conditions,
\begin{align}\label{eq_2}
  \left\{
  \begin{array}{ll}
  \mbs{x}(-1,\eta) = \mbs{x}_{AD}(\eta), & \quad \eta\in[-1,1], \\
  \mbs{x}(1,\eta) = \mbs{x}_{BC}(\eta), & \quad \eta\in[-1,1], \\
  \mbs{x}(\xi,-1) = \mbs{x}_{AB}(\xi), & \quad \xi\in[-1,1], \\
  \mbs{x}(\xi,1) = \mbs{x}_{CD}(\xi), & \quad \xi\in[-1,1],
  \end{array}
  \right.
\end{align}
where $\mbs{x}_{AB}(\xi)$, $\mbs{x}_{BC}(\eta)$, $\mbs{x}_{CD}(\xi)$,
and $\mbs{x}_{AD}(\eta)$ are 
the prescribed boundary curves for $\overline{AB}$, $\overline{BC}$, $\overline{CD}$
and $\overline{AD}$ in  parametric forms. The prescribed forms
should be compatible  at the vertices,
\begin{equation}
  \left\{
  \begin{array}{ll}
  \mbs x_{AB}(-1)=\mbs x_{AD}(-1)=\mbs x_A, &
  \mbs x_{AB}(1) = \mbs x_{BC}(-1) = \mbs x_B, \\
  \mbs x_{BC}(1) = \mbs x_{CD}(1) = \mbs x_C, &
  \mbs x_{CD}(-1) = \mbs x_{AD}(1) = \mbs x_{D},
  \end{array}
  \right.
\end{equation}
where $\mbs x_{A}$, $\mbs x_{B}$, $\mbs x_C$ and $\mbs x_D$ denote
the coordinates of the four vertices $A$, $B$, $C$ and $D$, respectively.

The general form of a vector-valued function that satisfies the
conditions in~\eqref{eq_2} is given by the TFC constrained expression,
\begin{equation}\label{eq_4}
  \mbs x(\xi,\eta) = \mbs g(\xi,\eta) - P\mbs g(\xi,\eta) + P\mbs X(\xi,\eta).
\end{equation}
Here $\mbs g(\xi,\eta)\in\mbb R^2$ is an arbitrary or free function (with sufficient
regularity), $P$ is
a 2D projection operator and
$P\mbs X(\xi,\eta)$ denotes the 2D transfinite interpolation~\cite{Gordon1971}
of $\mbs X(\xi,\eta)$,
where $\mbs X(\xi,\eta)$ is defined on the boundaries of $\Omega_{st}$ by
\begin{equation}
  \mbs X(-1,\eta) = \mbs x_{AD}(\eta), \
  \mbs X(1,\eta) = \mbs x_{BC}(\eta), \
  \mbs X(\xi,-1) = \mbs x_{AB}(\xi), \
  \mbs X(\xi,1) = \mbs x_{CD}(\xi).
\end{equation}
$P$ is given by the boolean sum of two 1D transfinite interpolation operators
$P_1$ and $P_2$ (see~\cite{Gordon1971} for details), for $\xi$ and $\eta$ directions
respectively, defined by 
\begin{equation}
  \left\{
  \begin{split}
    &
  P_1\mbs X(\xi,\eta) = \mbs X(-1,\eta)\phi_0(\xi) + \mbs X(1,\eta)\phi_1(\xi), \\
  &
  P_2\mbs X(\xi,\eta) = \mbs X(\xi,-1)\phi_0(\eta) + \mbs X(\xi,1)\phi_1(\eta),
  \end{split}
  \right.
\end{equation}
where $\phi_0(\xi)$ and $\phi_1(\xi)$ are the so-called
blending functions~\cite{Gordon1971}, also termed switching
functions in TFC, on $\xi\in[-1,1]$ as given by
\begin{equation}\label{eq_7}
  \phi_0(\xi) = \frac12(1-\xi), \quad
  \phi_1(\xi) = \frac12(1+\xi).
\end{equation}
$P\mbs X(\xi,\eta)$ is specifically given by,
\begin{equation}\label{eq_8}
  \begin{split}
  P\mbs X(\xi,\eta) =& \ (P_1\oplus P_2)\mbs X(\xi,\eta)
  = P_1\mbs X(\xi,\eta) + P_2\mbs X(\xi,\eta) - P_1P_2X(\xi,\eta) \\
  =& \ \mbs X(-1,\eta)\phi_0(\xi) + \mbs X(1,\eta)\phi_1(\xi)
  + \mbs X(\xi,-1)\phi_0(\eta) + \mbs X(\xi,1)\phi_1(\eta) \\
  & -\left[\mbs X(-1,-1)\phi_0(\eta) + \mbs X(-1,1)\phi_1(\eta) \right]\phi_0(\xi)
  - \left[\mbs X(1,-1)\phi_0(\eta) + \mbs X(1,1)\phi_1(\eta) \right]\phi_1(\xi),
  \end{split}
\end{equation}
where $\oplus$ denotes the boolean sum.
$P\mbs g(\xi,\eta)$ is defined in a fashion analogous to equation~\eqref{eq_8}.

The map given by~\eqref{eq_4}, for arbitrary $\mbs g(\xi,\eta)$,
satisfies the boundary conditions in~\eqref{eq_2} exactly.
For most numerical simulations in this paper, we employ simply
$\mbs g=0$ in~\eqref{eq_4}, leading to the map
\begin{equation}\label{eq_9}
  \mbs x(\xi,\eta) = P\mbs X(\xi,\eta).
\end{equation}
When the boundary curves are more complicated or highly
distorted, other choices for $\mbs g(\xi,\eta)$ in the mapping
can be more favorable. These choices will be specified
in the numerical simulations in later sections.


In practice, it is extremely difficult to ensure analytically the univalency
of the map in~\eqref{eq_4} or~\eqref{eq_9}.
In this work we employ the strategy as advocated in~\cite{GordonH1973}
for constructing univalent mappings, by combining
visualization with numerical simulation.
By visualizing representative grid lines in the domain,
one can effectively detect abnormalities in the domain mapping,
such as the ``overspill''
or the intersection of generalized grid lines
corresponding to different values of the same variable.
These anomalies can then be remedied via measures
such as boundary curve re-parameterization or incorporating
auxiliary constraints into the mapping. 

\begin{remark}\label{rem_1}
  The requirement for non-singularity in the Jacobian matrix for
  the mapping function~\eqref{eq_1} can be relaxed on those vertices
  where two intersecting boundaries involve either (i) both Dirichlet BCs,
  or (ii) one Dirichlet BC and one Neumann BC, or (iii) one
  Dirichlet BC and one Robin BC.
  In this work we allow these types of boundary pairs to intersect at a vertex
  smoothly, i.e. with the same tangent line, leading to
  a singular Jacobian matrix at that vertex.

\end{remark}


\subsection{Enforcing Dirichlet Boundary Conditions}
\label{sec_22}

We next systematically develop formulations for field functions
defined on the general quadrilateral domain as in Figure~\ref{fg_1}(a) that
exactly satisfy the  prescribed conditions on
its boundaries. We first consider Dirichlet boundary conditions (DBCs).

Specifically, we seek a scalar field function $u(\mbs x)$
($\mbs x\in\Omega=\overline{ABCD}$)
satisfying the following conditions,
  \begin{align}\label{eq_10}
    & \left.u(\mbs x)\right|_{\mbs x\in \overline{AB}} = u_{AB}(\mbs x), \quad
     \left.u(\mbs x)\right|_{\mbs x\in \overline{BC}} = u_{BC}(\mbs x), \quad
     \left.u(\mbs x)\right|_{\mbs x\in \overline{CD}} = u_{CD}(\mbs x), \quad
     \left.u(\mbs x)\right|_{\mbs x\in \overline{BC}} = u_{AD}(\mbs x),
  \end{align}
where $u_{AB}(\mbs x)$, $u_{BC}(\mbs x)$, $u_{CD}(\mbs x)$, and $u_{AD}(\mbs x)$
are prescribed Dirichlet data on the boundaries. These boundary
distributions should be compatible on the vertices,
\begin{subequations}\label{eq_11}
\begin{align}
  & u_{AB}(\mbs x_A)=u_{AD}(\mbs x_A)=u_A, \label{eq_11a} \\
  & u_{AB}(\mbs x_B)=u_{BC}(\mbs x_B)=u_B, \label{eq_11b} \\
  & u_{BC}(\mbs x_C)=u_{CD}(\mbs x_C)=u_C, \label{eq_11c} \\
  & u_{CD}(\mbs x_D)=u_{AD}(\mbs x_D)=u_D, \label{eq_11d}
\end{align}
\end{subequations}
where $\mbs x_A$, $\mbs x_B$, $\mbs x_C$ and $\mbs x_D$ are the vertex
coordinates, and $u_A$, $u_B$, $u_C$ and $u_D$ are their function values.
We aim to formulate the general form of $u(\mbs x)$ that
satisfies the DBCs in~\eqref{eq_10} exactly.

Employing the mapping $\mbs x(\xi,\eta)$ from Section~\ref{sec_21},
we  transform the field function into,
\begin{equation}\label{eq_12}
  u(\mbs x)=u(\mbs x(\xi,\eta))=V(\xi,\eta),
\end{equation}
and the boundary distributions in~\eqref{eq_10} into,
\begin{subequations}\label{eq_13}
  \begin{align}
    &
    u_{AB}(\mbs x) = u_{AB}(\mbs x(\xi,-1)) = F(\xi,-1),
    \quad \xi\in[-1,1]; \label{eq_13a} \\
    &
    u_{BC}(\mbs x) = u_{BC}(\mbs x(1,\eta)) = F(1,\eta),
    \quad \eta\in[-1,1]; \label{eq_13b} \\
    &
    u_{CD}(\mbs x) = u_{CD}(\mbs x(\xi,1)) = F(\xi,1),
    \quad \xi\in[-1,1]; \label{eq_13c}  \\
    &
    u_{AD}(\mbs x) = u_{AD}(\mbs x(-1,\eta)) = F(-1,\eta),
    \quad \eta\in[-1,1]. \label{eq_13d}
  \end{align}
\end{subequations}

The problem of seeking $u(\mbs x)$ is then transformed into the following.
Find $V(\xi,\eta)$, for $(\xi,\eta)\in\Omega_{st}$, such that
\begin{subequations}\label{eq_14}
  \begin{align}
    & V(\xi,-1) = F(\xi,-1), \quad \xi\in[-1,1]; \label{eq_14a} \\
    & V(1,\eta) = F(1,\eta), \quad \eta\in[-1,1]; \label{eq_14b} \\
    & V(\xi,1) = F(\xi,1), \quad \xi\in[-1,1]; \label{eq_14c} \\
    & V(-1,\eta) = F(-1,\eta), \quad \eta\in[-1,1], \label{eq_14d}
  \end{align}
\end{subequations}
where $F$ denotes the boundary distributions given in~\eqref{eq_13}.

The general form of $V(\xi,\eta)$ that satisfies the conditions in~\eqref{eq_14}
is given by the following TFC constrained expression,
\begin{equation}\label{eq_15}
  V(\xi,\eta) = g(\xi,\eta) - Pg(\xi,\eta) + PF(\xi,\eta), \quad
  (\xi,\eta)\in\Omega_{st},
\end{equation}
where $g(\xi,\eta)$ is a free (arbitrary) function, and
$P$ is the transfinite interpolation operator defined in~\eqref{eq_8}.
More specifically,
\begin{subequations} \label{eq_16}
\begin{align}
    P g(\xi,\eta) 
  =& \  g(-1,\eta)\phi_0(\xi) +  g(1,\eta)\phi_1(\xi)
  +  g(\xi,-1)\phi_0(\eta) +  g(\xi,1)\phi_1(\eta) \notag \\
  & -\left[ g(-1,-1)\phi_0(\eta) +  g(-1,1)\phi_1(\eta) \right]\phi_0(\xi)
  - \left[ g(1,-1)\phi_0(\eta) +  g(1,1)\phi_1(\eta) \right]\phi_1(\xi); \label{eq_16a} \\
    P F(\xi,\eta) 
  =& \  F(-1,\eta)\phi_0(\xi) +  F(1,\eta)\phi_1(\xi)
  +  F(\xi,-1)\phi_0(\eta) +  F(\xi,1)\phi_1(\eta) \notag \\
  & -\left[ F(-1,-1)\phi_0(\eta) +  F(-1,1)\phi_1(\eta) \right]\phi_0(\xi)
  - \left[ F(1,-1)\phi_0(\eta) +  F(1,1)\phi_1(\eta) \right]\phi_1(\xi). \label{eq_16b}
\end{align}
\end{subequations}
It is straightforward to verify that, for arbitrary $g(\xi,\eta)$,
the function $V(\xi,\eta)$ given by~\eqref{eq_15} satisfies
the boundary conditions in~\eqref{eq_14} exactly.

Therefore, employing the parametric form~\eqref{eq_15} for $u(\mbs x)$,
one can satisfy the DBCs in~\eqref{eq_10} exactly on the general
quadrilateral domain $\Omega$.
If $u(\mbs x)$ represents the unknown solution field to a given PDE,
one can restrict the free function $g(\xi,\eta)$ in~\eqref{eq_15}
to an appropriate function space or represent it by an artificial
neural network (NN), and then determine the expansion coefficients or the
NN trainable parameters based on the given PDE.
We will discuss how to combine the extreme learning machine (ELM) method
and the formulations developed
here for solving PDEs on general quadrilateral domains
later in Section~\ref{sec_25}.


\subsection{Enforcing Neumann Boundary Conditions}
\label{sec_23}

We next consider how to enforce Neumann boundary conditions (NBCs) exactly on
the general quadrilateral domain $\Omega$.
Since the boundary condition involves function derivatives,
the formulation for its exact enforcement
becomes much more intricate.

We consider a combination of Dirichlet and Neumann BCs
for the domain boundaries, and
assume that the domain involves at least one Neumann boundary
and one Dirichlet boundary, with the rest being either Dirichlet
or Neumann types.
We distinguish two
cases: (i) when the Neumann boundary only intersects with
Dirichlet boundaries (i.e. no two Neumann boundaries intersect), and
(ii) when two Neumann boundaries intersect with each other.
The formulations for the exact enforcement of NBCs/DBCs
of these cases are developed below separately.

\subsubsection{When Neumann Boundary Only Intersects with Dirichlet Boundary}
\label{sec_231}

This case occurs when the domain involves one Neumann boundary and three
Dirichlet boundaries, or when two Neumann conditions are imposed
on opposite sides of the domain.
In the discussions below we assume that the domain
involves a single Neumann boundary, and consider
the exact enforcement of NBC/DBCs. The formulation presented below,
with some modification that involves no essential difficulties,
can be used to enforce two Neumann conditions imposed on opposite
boundaries of the quadrilateral domain.

Let us assume, without loss of generality, that the single Neumann condition is
imposed on the boundary $\overline{BC}$.
Specifically, we seek a scalar field function $u(\mbs x)$, for $\mbs x\in\Omega$,
which satisfies the Neumann condition on $\overline{BC}$ and Dirichlet
conditions on the other boundaries, 
\begin{subequations}\label{eq_17}
  \begin{align}
    & \left.u\right|_{\mbs x\in \overline{AB}} = u_{AB}(\mbs x), \label{eq_17a} \\
    & \left.\mbs n\cdot \nabla u\right|_{\mbs x\in \overline{BC}} = u_{nBC}(\mbs x),
    \label{eq_17b}\\
    & \left.u\right|_{\mbs x\in \overline{CD}} = u_{CD}(\mbs x), \label{eq_17c} \\
    & \left.u\right|_{\mbs x\in \overline{AD}} = u_{AD}(\mbs x), \label{eq_17d}
  \end{align}
\end{subequations}
where $\mbs n$ denotes the outward-pointing unit normal vector, and
$u_{nBC}(\mbs x)$ is the prescribed Neumann boundary distribution on $\overline{BC}$.
The prescribed Dirichlet boundary functions $u_{AB}(\mbs x)$, $u_{AD}(\mbs x)$
and $u_{CD}(\mbs x)$ must be compatible at the vertices $A$ and $D$
(see~\eqref{eq_11a} and~\eqref{eq_11d}). They must also be compatible
with the Neumann boundary function $u_{nBC}(\mbs x)$ at the vertices
$B$ and $C$, which will be elaborated below.
%

In the following development we assume that the
boundaries $\overline{BC}$ and $\overline{AB}$ intersect
at an angle at vertex $B$ (i.e. no common tangent at $B$),
and that at vertex $C$ the boundaries
$\overline{BC}$ and $\overline{CD}$ also intersect at an angle,
thus leading to a nonsingular Jacobian matrix of the map
$\mbs x(\xi,\eta)$ at both vertices.
We will discuss how to handle a smooth domain boundary
at vertices $B$ or $C$, i.e.~with a common tangent at those locations
(singular Jacobian matrix for the mapping), 
in a remark at the end of this section.

Employing the mapping function $\mbs x(\xi,\eta)$, the function $u(\mbs x)$ and
the Dirichlet boundary functions in~\eqref{eq_17a} and~\eqref{eq_17c}--\eqref{eq_17d}
are transformed into~\eqref{eq_12},~\eqref{eq_13a} and~\eqref{eq_13c}--\eqref{eq_13d},
respectively.
The Neumann condition~\eqref{eq_17b} is transformed into,
\begin{subequations} \label{eq_18}
\begin{align}
  & V_{\xi}(1,\eta) + S_{BC}(\eta)V_{\eta}(1,\eta) = T_{BC}(\eta), \quad \eta\in[-1,1],
  \label{eq_18a} \\
    \text{or}\ &
    V_{\xi}(1,\eta) = T_{BC}(\eta) - S_{BC}(\eta)V_{\eta}(1,\eta) = F_{\xi}(1,\eta),
    \label{eq_18b}
\end{align}
\end{subequations}
where 
\begin{equation}\label{eq_19}
  \left\{
  \begin{split}
    &
    S_{BC}(\eta) = \frac{K_{yBC}(\eta)}{K_{xBC}(\eta)}, \quad
    T_{BC}(\eta) = \frac{F_{nBC}(\eta)}{K_{xBC}(\eta)}
    = \frac{u_{nBC}(\mbs x(1,\eta))}{K_{xBC}(\eta)}, \\
    &
    \mbs K_{BC}(\eta)=\begin{bmatrix}K_{xBC}(\eta) \\ K_{yBC}(\eta) \end{bmatrix}
  = \mbs J^{-1}(1,\eta)\begin{bmatrix}n_{xBC}(\eta) \\ n_{yBC}(\eta) \end{bmatrix}
  = \frac{1}{ \det{\mbs J(1,\eta)}}
  \begin{bmatrix}\|\mbs x_{\eta}(1,\eta) \| \\
    -\frac{\mbs x_{\xi}(1,\eta)\cdot\mbs x_{\eta}(1,\eta)}
    {\|\mbs x_{\eta}(1,\eta) \|} \end{bmatrix}, \\
  & \mbs J(\xi,\eta) = \begin{bmatrix} x_{\xi}(\xi,\eta) & x_{\eta}(\xi,\eta)\\
  y_{\xi}(\xi,\eta) & y_{\eta}(\xi,\eta)\end{bmatrix}.
  \end{split}
  \right.
\end{equation}
In the above equations $\mbs J(\xi,\eta)$ is the Jacobian matrix of the map
$\mbs x(\xi,\eta)$, and
we have used the relations
\begin{equation}\label{eq_20}
  \left\{
  \begin{split}
    &
    \mbs n\cdot\nabla u = \begin{bmatrix}u_x & u_y\end{bmatrix} \begin{bmatrix} n_x \\ n_y \end{bmatrix}
  = \begin{bmatrix}V_{\xi} & V_{\eta}\end{bmatrix}\mbs J^{-1}(\xi,\eta)\begin{bmatrix} n_x \\ n_y \end{bmatrix}
  = \begin{bmatrix}V_{\xi} & V_{\eta}\end{bmatrix} \begin{bmatrix} K_x \\ K_y \end{bmatrix}; \\
  &
  \begin{bmatrix} n_x \\ n_y \end{bmatrix} = \mbs n = \bm \sigma\bm\tau = \bm \sigma
  \begin{bmatrix} \tau_x \\ \tau_y \end{bmatrix}; \quad
  \bm \sigma = \begin{bmatrix} 0 & 1\\ -1 & 0 \end{bmatrix} \
  \text{on}\ \overline{AB}\ \text{or}\ \overline{BC}\ \text{and}\
  \begin{bmatrix} 0 & -1\\ 1 & 0 \end{bmatrix} \
  \text{on}\ \overline{CD}\ \text{or}\ \overline{AD},
  \end{split}
  \right.
\end{equation}
in which $\bm\tau=(\tau_x,\tau_y)$ is the unit tangent vector of the domain
boundary.

The requirement for continuity of $V(\xi,\eta)$ and $V_{\xi}(\xi,\eta)$
at the vertices $B$ and $C$ leads to,
\begin{subequations}\label{eq_21}
  \begin{align}
    &
    V(1,-1) = F(1,-1) = \lim_{\xi\rightarrow 1}F(\xi,-1), \label{eq_21a} \\
    &
    V(1,1) = F(1,1) = \lim_{\xi\rightarrow 1}F(\xi,1), \label{eq_21b} \\
    &
    V_{\xi}(1,-1) = F_{\xi}(1,-1) = \lim_{\xi\rightarrow 1}F_{\xi}(\xi,-1), \label{eq_21c} \\
    &
    V_{\xi}(1,1) = F_{\xi}(1,1) = \lim_{\xi\rightarrow 1} F_{\xi}(\xi,1), \label{eq_21d}
  \end{align}
\end{subequations}
where $F_{\xi}(\xi,-1)$ and $F_{\xi}(\xi,1)$ can be computed from the
Dirichlet boundary functions on $\overline{AB}$ and $\overline{CD}$
(see~\eqref{eq_13a} and~\eqref{eq_13c}).
Therefore, the  compatibility
between~\eqref{eq_18} and~\eqref{eq_14a} at vertex $B$
leads to, 
\begin{subequations}\label{eq_22}
  \begin{align}
    &V_{\eta}(1,-1) =  \frac{1}{S_{BC}(-1)}\left[ T_{BC}(-1) - F_{\xi}(1,-1) \right]
    =F_{\eta}^a(1,-1),
    &
     \text{if}\ \overline{AB}\notperp\overline{BC}\ \text{at}\ B, \label{eq_22a}  \\
  &F_{\xi}(1,-1) - T_{BC}(-1) = 0, & 
  \text{if}\ \overline{AB}\perp\overline{BC}\ \text{at}\ B, \label{eq_22b}
  \end{align}
\end{subequations}
where we have used~\eqref{eq_21c}.
The compatibility between~\eqref{eq_18} and~\eqref{eq_14c} at vertex $C$ leads to,
\begin{subequations}\label{eq_23}
  \begin{align}
    &V_{\eta}(1,1) =  \frac{1}{S_{BC}(1)}\left[ T_{BC}(1) - F_{\xi}(1,1) \right]
    =F_{\eta}^a(1,1),
    &
     \text{if}\ \overline{CD}\notperp\overline{BC}\ \text{at}\ C, \label{eq_23a} \\
  &F_{\xi}(1,1) - T_{BC}(1) = 0, & 
  \text{if}\ \overline{CD}\perp\overline{BC}\ \text{at}\ C. \label{eq_23b}
  \end{align}
\end{subequations}
The conditions~\eqref{eq_22b} and~\eqref{eq_23b} are constraints
on the prescribed Dirichlet and Neumann boundary functions when the boundaries
are orthogonal at these vertices, which we assume will always be satisfied
by the prescribed data.
The equations~\eqref{eq_22a} and~\eqref{eq_23a} are constraints
on $V_{\eta}(1,-1)$ and $V_{\eta}(1,1)$ for the unknown field function $V(\xi,\eta)$,
 imposed only when the boundary curves are not
orthogonal at vertices $B$ or $C$.

Our objective is to develop general forms of $V(\xi,\eta)$
that exactly
satisfy the conditions~\eqref{eq_18},~\eqref{eq_21a}--\eqref{eq_21b},
\eqref{eq_22a},~\eqref{eq_23a},
as well as~\eqref{eq_14a} and~\eqref{eq_14c}--\eqref{eq_14d}.
The primary challenge here is caused by~\eqref{eq_18},
in which the unknowns $V_{\xi}(1,\eta)$ and $V_{\eta}(1,\eta)$ are coupled together.
We will use~\eqref{eq_18b} for enforcing this condition,
by treating it as a constraint on $V_{\xi}(1,\eta)$,
where the imposed data $F_{\xi}(1,\eta)$ contains the unknown $V_{\eta}(1,\eta)$.


We will follow a four-step procedure to
formulate the general form of $V(\xi,\eta)$ that satisfies the above conditions:
\begin{itemize}
\item (step \#1) Identify the set of variables on which the boundary conditions, and
  the compatibility constraints 
  induced by these boundary conditions, are imposed.

\item (step \#2)
  Construct a transfinite interpolation for the types of identified variables,
  including the induced forms of constraints.

\item (step \#3)
  Formulate a preliminary TFC constrained expression based on this transfinite interpolation,
  thus giving rise to a preliminary form for $V(\xi,\eta)$ with a free function.

\item (step \#4)
  Update the terms in the transfinite interpolant that involve the unknown function
  $V(\xi,\eta)$, by replacing those $V$ terms with corresponding terms that
  result from the preliminary TFC form of the previous step.
  The updated TFC constrained expression
  provides the final form for $V(\xi,\eta)$.
  
\end{itemize}

\begin{table}
  \centering
  \begin{tabular}{c|cc|cc|cc|cc}
    \hline
    $\xi$ & $\varphi_0(\xi)$ & $\varphi_0'(\xi)$
    & $\varphi_1(\xi)$ & $\varphi_1'(\xi)$
    & $\psi_0(\xi)$ & $\psi_0'(\xi)$
    & $\psi_1(\xi)$ & $\psi_1'(\xi)$ \\ \hline
    $-1$ & 1 & 0 & 0 & 0 & 0 & 1 & 0 & 0 \\ \hline
    1 & 0 & 0 & 1 & 0 & 0 & 0 & 0 & 1 \\
    \hline
  \end{tabular}
  \caption{Interpolation properties of $C^1$ Hermite interpolation
    polynomials defined on $\xi\in[-1,1]$.
  }
  \label{tab_1}
\end{table}

To facilitate the subsequent discussions, let us recall the $C^1$ Hermite
interpolation polynomials $\varphi_0(\xi)$, $\varphi_1(\xi)$, $\psi_0(\xi)$
and $\psi_1(\xi)$ defined on $\xi\in[-1,1]$, which satisfy
the interpolation properties listed in Table~\ref{tab_1}.
These polynomials are given by
\begin{align}
  \begin{array}{ll}
    \varphi_0(\xi) = \phi_0^2(\xi)\left[1+2\phi_1(\xi) \right], &
    \varphi_1(\xi) = \phi_1^2(\xi)\left[1+2\phi_0(\xi) \right], \\
    \psi_0(\xi) = 2\phi_0^2(\xi)\phi_1(\xi), &
    \psi_1(\xi) = -2\phi_0(\xi)\phi_1^2(\xi),
    \end{array}
  \end{align}
where $\phi_0(\xi)$ and $\phi_1(\xi)$ are defined in~\eqref{eq_7}.
We define two constants $\lambda_B$ and $\lambda_C$
as flags on whether the boundary curves are
orthogonal at the vertices $B$ and $C$,
\begin{equation}\label{eq_25}
  \lambda_B = \left\{
  \begin{array}{ll}
    0, & \text{if}\ \overline{AB}\perp\overline{BC}\ \text{at}\ B, \\
    1, & \text{if}\ \overline{AB}\notperp\overline{BC}\ \text{at}\ B;
  \end{array}
  \right.
  \quad
  \lambda_C = \left\{
  \begin{array}{ll}
    0, & \text{if}\ \overline{CD}\perp\overline{BC}\ \text{at}\ C, \\
    1, & \text{if}\ \overline{CD}\notperp\overline{BC}\ \text{at}\ C.
  \end{array}
  \right.
\end{equation}

Following the aforementioned procedure, we first identify the set of variables on which
the conditions (or induced constraints) are imposed (step \#1).
The conditions for~\eqref{eq_18b},~\eqref{eq_14a}, and~\eqref{eq_14c}--\eqref{eq_14d}
are very clear. While the conditions~\eqref{eq_22a},~\eqref{eq_23a}, and
also~\eqref{eq_21a}--\eqref{eq_21b} are apparently about
the function values or derivatives on the
vertices $B$ and $C$, they are actually constraints on the
boundary distribution $V(1,\eta)$. $V(1,\eta)$ is not involved in
the original boundary conditions, but should ensure that these conditions
for the vertices be satisfied. Employing the $C^1$ Hermite interpolation
polynomials, the following distribution for $\overline{BC}$
satisfies the conditions~\eqref{eq_22a},~\eqref{eq_23a},
and~\eqref{eq_21a}--\eqref{eq_21b},
\begin{equation}\label{eq_26}
  F(1,\eta) = F(1,-1)\varphi_0(\eta) + F(1,1)\varphi_1(\eta)
   + \lambda_B F_{\eta}(1,-1)\psi_0(\eta) + \lambda_C F_{\eta}(1,1)\psi_1(\eta),
\end{equation}
Here,
\begin{align}\label{eq_n27}
  F_{\eta}(1,-1)=F_{\eta}^a(1,-1), \quad F_{\eta}(1,1)=F_{\eta}^a(1,1),
\end{align}
with $F_{\eta}^a(1,-1)$ and $F_{\eta}^a(1,1)$ defined in~\eqref{eq_22a}
and~\eqref{eq_23a}. The constants $\lambda_B$ and $\lambda_C$
ensure that the conditions~\eqref{eq_22a}
and~\eqref{eq_23a} are imposed only when the boundary curves are not orthogonal
at $B$ or $C$. The distribution~\eqref{eq_26} is compatible with
the Dirichlet boundary functions for $\overline{AB}$ and $\overline{CD}$ at
these vertices.
Our task is then reduced to the following: find $V(\xi,\eta)$ such that
\begin{subequations}\label{eq_27}
  \begin{align}
    &
    V(-1,\eta) = F(-1,\eta), \label{eq_27a} \\
    &
    V(1,\eta) = F(1,\eta)
    = F(1,-1)\varphi_0(\eta) + F(1,1)\varphi_1(\eta)
   + \lambda_B F_{\eta}(1,-1)\psi_0(\eta) + \lambda_C F_{\eta}(1,1)\psi_1(\eta), \label{eq_27b} \\
    &
    V_{\xi}(1,\eta) = F_{\xi}(1,\eta), \label{eq_27c} \\
    &
    V(\xi,-1) = F(\xi,-1), \label{eq_27d} \\
    &
    V(\xi,1) = F(\xi,1), \label{eq_27e}
  \end{align}
\end{subequations}
where equation~\eqref{eq_26} has been used,
the functions $F(-1,\eta)$, $F(\xi,-1)$ and $F(\xi,1)$ are known and
given in~\eqref{eq_13a},~\eqref{eq_13c}--\eqref{eq_13d},
and $F_{\xi}(1,\eta)$ is defined
in~\eqref{eq_18b}, which contains the unknown $V_{\eta}(1,\eta)$.

Next, we construct a transfinite interpolation for the
conditions in~\eqref{eq_27} (step \#2).
Let $P_1F$ and $P_2F$ denote two 1D transfinite interpolations along the
$\xi$ and $\eta$ directions, respectively defined by,
\begin{subequations}
  \begin{align}
    &
    P_1F(\xi,\eta) = F(-1,\eta)\varphi_0(\xi) + F(1,\eta)\varphi_1(\xi)
    + F_{\xi}(1,\eta)\psi_1(\xi), \\
    &
    P_2F(\xi,\eta) = F(\xi,-1)\varphi_0(\eta) + F(\xi,1)\varphi_1(\xi).
  \end{align}
\end{subequations}
Let $PF$ denote the  boolean sum
of $P_1F$ and $P_2F$, given by,
\begin{equation}\label{eq_30}
  \begin{split}
    PF(\xi,\eta) =&\ (P_1 \oplus P_2)F(\xi,\eta)
    = P_1F(\xi,\eta) + P_2F(\xi,\eta) - P_1P_2F(\xi,\eta) \\
    =&\  F(-1,\eta)\varphi_0(\xi) + F_{\xi}(1,\eta)\psi_1(\xi)
    + F(\xi,-1)\varphi_0(\eta) + F(\xi,1)\varphi_1(\eta) \\
    & -\left[F(-1,-1)\varphi_0(\eta) + F(-1,1)\varphi_1(\eta) \right]\varphi_0(\xi)
    - \left[F_{\xi}(1,-1)\varphi_0(\eta) + F_{\xi}(1,1)\varphi_1(\eta) \right]\psi_1(\xi) \\
    & + \left[\lambda_B F_{\eta}(1,-1)\psi_0(\eta)
      + \lambda_C F_{\eta}(1,1)\psi_1(\eta) \right]\varphi_1(\xi)
  \end{split}
\end{equation}
where we have used~\eqref{eq_26}.
It is straightforward to verify that $V=PF(\xi,\eta)$ satisfies
the conditions in~\eqref{eq_27a}, \eqref{eq_27c}--\eqref{eq_27e},
\eqref{eq_22a}, \eqref{eq_23a} and~\eqref{eq_21a}--\eqref{eq_21b},
by noting~\eqref{eq_n27}.

We can now formulate a preliminary general form of $V(\xi,\eta)$ for
the conditions~\eqref{eq_27} using the TFC constrained expression (step \#3),
\begin{align}\label{eq_31}
  & V(\xi,\eta) = g(\xi,\eta) - Pg(\xi,\eta) + PF(\xi,\eta),
\end{align}
where $g(\xi,\eta)$ is a free (arbitrary) function,
$PF$ is defined by~\eqref{eq_30}, and $Pg(\xi,\eta)$ is defined analogously
and specifically given by
\begin{equation}\label{eq_32}
  \begin{split}
    Pg(\xi,\eta) =&\  g(-1,\eta)\varphi_0(\xi) + g_{\xi}(1,\eta)\psi_1(\xi)
    + g(\xi,-1)\varphi_0(\eta) + g(\xi,1)\varphi_1(\eta) \\
    & -\left[g(-1,-1)\varphi_0(\eta) + g(-1,1)\varphi_1(\eta) \right]\varphi_0(\xi)
    - \left[g_{\xi}(1,-1)\varphi_0(\eta) + g_{\xi}(1,1)\varphi_1(\eta) \right]\psi_1(\xi) \\
    & + \left[\lambda_B g_{\eta}(1,-1)\psi_0(\eta)
      + \lambda_C g_{\eta}(1,1)\psi_1(\eta) \right]\varphi_1(\xi).
  \end{split}
\end{equation}
For any $g(\xi,\eta)$,
the $V(\xi,\eta)$ given by~\eqref{eq_31} satisfies~\eqref{eq_27a},
\eqref{eq_27c}--\eqref{eq_27e},
\eqref{eq_22a}, \eqref{eq_23a} and~\eqref{eq_21a}--\eqref{eq_21b},
by noting~\eqref{eq_n27}.
%
In light of~\eqref{eq_31},
$V_{\eta}(\xi,\eta)$ on $\overline{BC}$  is reduced to,
\begin{align}
  V_{\eta}(1,\eta) = g_{\eta}(1,\eta) &- \left[g(1,-1) - F(1,-1) \right]\varphi_0'(\eta)
  - \left[g(1,1) - F(1,1) \right]\varphi_1'(\eta) \notag \\
  &
  - \lambda_B\left[g_{\eta}(1,-1) - F_{\eta}(1,-1) \right]\psi_0'(\eta)
  - \lambda_C\left[g_{\eta}(1,1) - F_{\eta}(1,1) \right]\psi_1'(\eta). \label{eq_33}
\end{align}

Finally, we update the terms in the transfinite interpolant $PF(\xi,\eta)$
that involve the unknown function $V(\xi,\eta)$ (step \#4).
In~\eqref{eq_30} and~\eqref{eq_33}, we update $F_{\eta}(1,-1)$ and $F_{\eta}(1,1)$
using~\eqref{eq_n27}.
$F_{\xi}(1,\eta)$ in~\eqref{eq_30} is given by~\eqref{eq_18b}, in which we replace
$V_{\eta}(1,\eta)$ by the expression~\eqref{eq_33}.
We define the updated terms,
\begin{subequations}
  \begin{align}
    F_{\xi}^{g}(1,\eta) =&\ T_{BC}(\eta) - S_{BC}(\eta)\left\{
      g_{\eta}(1,\eta) - \left[g(1,-1) - F(1,-1) \right]\varphi_0'(\eta)
      - \left[g(1,1) - F(1,1) \right]\varphi_1'(\eta) \right. \notag \\
      &\qquad\qquad\qquad
      \left.
      - \lambda_B\left[g_{\eta}(1,-1) - F^a_{\eta}(1,-1) \right]\psi_0'(\eta)
      - \lambda_C\left[g_{\eta}(1,1) - F^a_{\eta}(1,1) \right]\psi_1'(\eta)
      \right\}, \label{eq_34a} \\
    PF^{g}(\xi,\eta) =&\
      F(-1,\eta)\varphi_0(\xi) + F^g_{\xi}(1,\eta)\psi_1(\xi)
    + F(\xi,-1)\varphi_0(\eta) + F(\xi,1)\varphi_1(\eta) \notag \\
    & -\left[F(-1,-1)\varphi_0(\eta) + F(-1,1)\varphi_1(\eta) \right]\varphi_0(\xi)
    - \left[F_{\xi}(1,-1)\varphi_0(\eta) + F_{\xi}(1,1)\varphi_1(\eta) \right]\psi_1(\xi) \notag \\
    & + \left[\lambda_B F^a_{\eta}(1,-1)\psi_0(\eta)
      + \lambda_C F^a_{\eta}(1,1)\psi_1(\eta) \right]\varphi_1(\xi). \label{eq_34b}
  \end{align}
\end{subequations}

The final form for $V(\xi,\eta)$ is then given by
\begin{align}\label{eq_35}
  & V(\xi,\eta) = g(\xi,\eta) - Pg(\xi,\eta) + PF^g(\xi,\eta).
\end{align}
Here $g(\xi,\eta)$ is the free function, $Pg(\xi,\eta)$ is given by~\eqref{eq_32},
and $PF^{g}(\xi,\eta)$ is defined by~\eqref{eq_34b}.
$F(-1,\eta)$, $F(\xi,-1)$ and $F(\xi,1)$ are given by~\eqref{eq_13a}
and~\eqref{eq_13c}--\eqref{eq_13d}.
$F(-1,-1)=u_A$ and $F(-1,1)=u_D$ in light of~\eqref{eq_13a},~\eqref{eq_13d},~\eqref{eq_11a}
and~\eqref{eq_11d}.
$F(1,-1)=\lim_{\xi\rightarrow 1}F(\xi,-1)=u_B$ and
$F(1,1)=\lim_{\xi\rightarrow 1}F(\xi,1)=u_C$ in light of~\eqref{eq_13a} and~\eqref{eq_13c}.
$F_{\xi}(1,-1)=\lim_{\xi\rightarrow 1}F_{\xi}(\xi,-1)$ and
$F_{\xi}(1,1)=\lim_{\xi\rightarrow 1}F_{\xi}(\xi,1)$.
$F_{\eta}^a(1,-1)$ and $F_{\eta}^a(1,1)$ are given by~\eqref{eq_22a} and~\eqref{eq_23a},
and $\lambda_B$ and $\lambda_C$ are defined in~\eqref{eq_25}.
$S_{BC}(\eta)$ and $T_{BC}(\eta)$ are defined in~\eqref{eq_19}.

\begin{theorem}
  The form $V(\xi,\eta)$ given by~\eqref{eq_35} satisfies the conditions~\eqref{eq_27a},
  \eqref{eq_27d},~\eqref{eq_27e} and~\eqref{eq_18a},
  for any $g(\xi,\eta)$ therein that is sufficiently differentiable.
\end{theorem}
\begin{proof}
  We only verify the Neumann condition~\eqref{eq_18a} here. The verification
  of the other conditions is straightforward.
  From~\eqref{eq_35}, we have
  \begin{align*}
    &
    V_{\xi}(1,\eta) = F_{\xi}^g(1,\eta), \\
    &
    V_{\eta}(1,\eta) = g_{\eta}(1,\eta) - \left[g(1,-1) - F(1,-1) \right]\varphi_0'(\eta)
    - \left[g(1,1) - F(1,1) \right]\varphi_1'(\eta) \\
    &\qquad\qquad\qquad\qquad\quad
    - \lambda_B\left[g_{\eta}(1,-1) - F^a_{\eta}(1,-1) \right]\psi_0'(\eta)
      - \lambda_C\left[g_{\eta}(1,1) - F^a_{\eta}(1,1) \right]\psi_1'(\eta).
  \end{align*}
  In light of~\eqref{eq_34a}, we conclude that~\eqref{eq_18a} holds for any $g(\xi,\eta)$.
\end{proof}


\begin{remark}
  The use of $C^1$ Hermite interpolation polynomials $\varphi_0$, $\varphi_1$, $\psi_0$
  and $\psi_1$ in the construction of $V(\xi,\eta)$ is crucial, which enables one
  to de-couple $V_{\xi}(1,\eta)$ and $V_{\eta}(1,\eta)$ when handling
  the Neumann condition~\eqref{eq_18}.
\end{remark}


\begin{remark}
  If the boundaries $\overline{AB}$ and $\overline{BC}$ connect smoothly
  at vertex $B$, i.e. having a common tangent,
  the Jacobian matrix of the map $\mbs x(\xi,\eta)$ will be singular
  at $B$. Similarly, if the boundary curve is smooth at vertex $C$
  the Jacobian matrix will be singular there.
  Let us suppose the boundary is smooth at both $B$ and $C$, and
  we next comment on how to handle this situation.
  In this case the form~\eqref{eq_18} for the Neumann condition
  holds only for $\eta\in(-1,1)$.
  At these vertices the formulas~\eqref{eq_22a} and~\eqref{eq_23a} for
  computing $F_{\eta}^a(1,-1)$ and $F_{\eta}^a(1,1)$ are no longer valid.
  However, $F_{\eta}^a(1,-1)$ and $F_{\eta}^a(1,1)$
  can still be determined, based on the existence of a
  common tangent at these vertices.
  Specifically, let
  \begin{align}
    &
    \bm\tau_{BC}(\eta)=\frac{\mbs x_{\eta}(1,\eta)}{\|\mbs x_{\eta}(1,\eta) \|}, \quad
    \bm\tau_{AB}(\xi)=\frac{\mbs x_{\xi}(\xi,-1)}{\|\mbs x_{\xi}(\xi,-1) \|}, \quad
    \bm\tau_{CD}(\xi)=\frac{\mbs x_{\xi}(\xi,1)}{\|\mbs x_{\xi}(\xi,1) \|},
  \end{align}
  denote the unit tangent vectors on $\overline{BC}$, $\overline{AB}$ and $\overline{CD}$.
  The existence of a common tangent at $B$ and $C$ implies that
  \begin{align}\label{eq_37}
    & \bm \tau_{BC}(-1) = \bm\tau_{AB}(1), \quad
    \bm\tau_{BC}(1) = -\bm\tau_{CD}(1).
  \end{align}
  This leads to the relations
  \begin{align}\label{eq_38}
    & \frac{\mbs x_{\eta}(1,-1)}{\|\mbs x_{\eta}(1,-1) \|}
    = \frac{\mbs x_{\xi}(1,-1)}{\|\mbs x_{\xi}(1,-1) \|}, \quad
    \frac{\mbs x_{\eta}(1,1)}{\|\mbs x_{\eta}(1,1) \|}
    = - \frac{\mbs x_{\xi}(1,1)}{\|\mbs x_{\xi}(1,1) \|}.
  \end{align}
  At vertices $B$ and $C$, the relations~\eqref{eq_21c}--\eqref{eq_21d} are still valid
  due to the Dirichlet BCs on $\overline{AB}$ and $\overline{CD}$.
  Employing the chain rule, we have
  \begin{subequations}\label{eq_39}
  \begin{align}
    V_{\eta}(1,-1) =&\ \nabla u(\mbs x_B) \cdot\mbs x_{\eta}(1,-1)
    = \nabla u(\mbs x_B) \cdot\mbs x_{\xi}(1,-1)
    \frac{\|\mbs x_{\eta}(1,-1) \|}{\|\mbs x_{\xi}(1,-1) \|}
    = \frac{\|\mbs x_{\eta}(1,-1) \|}{\|\mbs x_{\xi}(1,-1) \|}V_{\xi}(1,-1) \notag \\
    =&\ \frac{\|\mbs x_{\eta}(1,-1) \|}{\|\mbs x_{\xi}(1,-1) \|}F_{\xi}(1,-1)
    =F_{\eta}^a(1,-1), \\
    V_{\eta}(1,1) =&\ \nabla u(\mbs x_B) \cdot\mbs x_{\eta}(1,1)
    = -\nabla u(\mbs x_B) \cdot\mbs x_{\xi}(1,1)
    \frac{\|\mbs x_{\eta}(1,1) \|}{\|\mbs x_{\xi}(1,1) \|}
    = -\frac{\|\mbs x_{\eta}(1,1) \|}{\|\mbs x_{\xi}(1,1) \|}V_{\xi}(1,1) \notag \\
    =&\ -\frac{\|\mbs x_{\eta}(1,1) \|}{\|\mbs x_{\xi}(1,1) \|}F_{\xi}(1,1)
    =F_{\eta}^a(1,1),
  \end{align}
  \end{subequations}
  where we have used~\eqref{eq_38} and~\eqref{eq_21c}--\eqref{eq_21d}.
  The $V(\xi,\eta)$ given by~\eqref{eq_35} is still valid, in which
  $F_{\eta}^a(1,-1)$ and $F_{\eta}^a(1,1)$ should now be computed using~\eqref{eq_39}.
  The gradients $\nabla u(\mbs x_B)$ and $\nabla u(\mbs x_C)$ can be
  determined by combining the Neumann condition~\eqref{eq_17b}, evaluated
  at these vertices, with
  the equations~\eqref{eq_21c}--\eqref{eq_21d}, and by applying the chain rule
  and using~\eqref{eq_37}.
  This leads to the result,
  \begin{subequations}
  \begin{align}
    &
    \nabla u(\mbs x_B) = \begin{bmatrix} F_{nBC}(-1) &
      \frac{1}{\|\mbs x_{\xi}(1,-1) \|}F_{\xi}(1,-1)
    \end{bmatrix} \begin{bmatrix}
      \tau_{yBC}(-1) & -\tau_{xBC}(-1) \\ \tau_{xBC}(-1) & \tau_{yBC}(-1)
    \end{bmatrix}, \\
    &
    \nabla u(\mbs x_C) = \begin{bmatrix} F_{nBC}(1) &
      -\frac{1}{\|\mbs x_{\xi}(1,1) \|}F_{\xi}(1,1)
    \end{bmatrix} \begin{bmatrix}
      \tau_{yBC}(1) & -\tau_{xBC}(1) \\ \tau_{xBC}(1) & \tau_{yBC}(1)
    \end{bmatrix},
  \end{align}
  \end{subequations}
  where $\bm\tau_{BC}(\eta)=(\tau_{xBC}, \tau_{yBC})$, and $F_{nBC}(\eta)$
  is defined in~\eqref{eq_19}.
   
\end{remark}


\begin{remark}
  When two Neumann conditions are imposed on opposite sides of
  the quadrilateral domain $\Omega$, with Dirichlet conditions on the
  other boundaries, the general form
  for $V(\xi,\eta)$ that exactly satisfies these boundary conditions
  can be developed analogously by following the four-step procedure as
  described above.
\end{remark}

\subsubsection{When Two Neumann Boundaries Intersect at a Vertex}
\label{sec_232}

For this case we focus on the setting with the Neumann conditions
imposed on two adjacent boundaries of the quadrilateral domain and with the rest
being Dirichlet boundaries.
The settings with more than two Neumann boundaries are discussed in
a remark at the end of this section.
Without loss of generality, we assume that the Neumann conditions
are imposed on the boundaries $\overline{BC}$ and $\overline{CD}$ and that
$\overline{AB}$ and $\overline{AD}$ are Dirichlet boundaries.
The Jacobian matrix of the map $\mbs x(\xi,\eta)$
is assumed to be non-singular everywhere in the domain.

Specifically, we seek a scalar field function $u(\mbs x)$,
for $\mbs x\in\Omega=\overline{ABCD}$ in Figure~\ref{fg_1}(a),
which satisfies the following boundary conditions,
\begin{subequations}\label{eq_41}
  \begin{align}
    & \left.u\right|_{\mbs x\in \overline{AB}} = u_{AB}(\mbs x), \label{eq_41a} \\
    & \left.\mbs n\cdot \nabla u\right|_{\mbs x\in \overline{BC}} = u_{nBC}(\mbs x),
    \label{eq_41b}\\
    & \left.\mbs n\cdot \nabla u\right|_{\mbs x\in \overline{CD}} = u_{nCD}(\mbs x),
     \label{eq_41c} \\
    & \left.u\right|_{\mbs x\in \overline{AD}} = u_{AD}(\mbs x), \label{eq_41d}
  \end{align}
\end{subequations}
where $u_{nCD}(\mbs x)$ is the prescribed Neumann boundary distribution
on $\overline{CD}$, and the other notations follow those in the previous sections.
The prescribed Dirichlet and Neumann boundary functions
must be compatible on the shared vertices.

Employing the map $\mbs x(\xi,\eta)$, we transform $u(\mbs x)$, $u_{AB}(\mbs x)$
and $u_{AD}(\mbs x)$ into $V(\xi,\eta)$, $F(\xi,-1)$ and $F(-1,\eta)$
according to equations~\eqref{eq_12},~\eqref{eq_13a} and~\eqref{eq_13d}.
The Neumann condition~\eqref{eq_41b} is accordingly transformed into~\eqref{eq_18}.
The Neumann condition~\eqref{eq_41c} becomes
\begin{subequations} \label{eq_42}
\begin{align}
  & V_{\eta}(\xi,1) + S_{CD}(\xi)V_{\xi}(\xi,1) = T_{CD}(\xi), \quad \xi\in[-1,1],
  \label{eq_42a} \\
    \text{or}\ &
    V_{\eta}(\xi,1) = T_{CD}(\xi) - S_{CD}(\xi)V_{\xi}(\xi,1) = F_{\eta}(\xi,1),
    \label{eq_42b}
\end{align}
\end{subequations}
where we have used~\eqref{eq_20} and
\begin{equation}\label{eq_43}
  \left\{
  \begin{split}
    &
    S_{CD}(\xi) = \frac{K_{xCD}(\xi)}{K_{yCD}(\xi)}, \quad
    T_{CD}(\xi) = \frac{F_{nCD}(\xi)}{K_{yCD}(\xi)}
    = \frac{u_{nCD}(\mbs x(\xi,1))}{K_{yCD}(\xi)}, \\
    &
    \mbs K_{CD}(\xi)=\begin{bmatrix}K_{xCD}(\xi) \\ K_{yCD}(\xi) \end{bmatrix}
  = \mbs J^{-1}(\xi,1)\begin{bmatrix}n_{xCD}(\xi) \\ n_{yCD}(\xi) \end{bmatrix}
  = \frac{1}{ \det{\mbs J(\xi,1)}}
  \begin{bmatrix}
    -\frac{\mbs x_{\xi}(\xi,1)\cdot\mbs x_{\eta}(\xi,1)}
    {\|\mbs x_{\xi}(\xi,1) \|} \\
    \|\mbs x_{\xi}(\xi,1) \|
  \end{bmatrix}.
  \end{split}
  \right.
\end{equation}


The Neumann boundary $\overline{BC}$ and the Dirichlet
boundary $\overline{AB}$ intersect at vertex $B$, inducing
the compatibility constraints~\eqref{eq_21a},~\eqref{eq_21c}, and~\eqref{eq_22}
at vertex $B$.
Similar compatibility conditions  exist at vertex $D$, where
the Neumann boundary $\overline{CD}$ and Dirichlet boundary $\overline{AD}$
intersect. These are
\begin{subequations}\label{eq_44}
  \begin{align} 
    & V(-1,1) = F(-1,1) = \lim_{\eta\rightarrow 1}F(-1,\eta), \label{eq_44a} \\
    &V_{\eta}(-1,1) = F_{\eta}(-1,1) = \lim_{\eta\rightarrow 1}F_{\eta}(-1,\eta), \label{eq_44b}
  \end{align}
\end{subequations}
and
\begin{subequations}\label{eq_45}
  \begin{align}
    &V_{\xi}(-1,1) =  \frac{1}{S_{CD}(-1)}\left[ T_{CD}(-1) - F_{\eta}(-1,1) \right]
    =F_{\xi}^a(-1,1),
    &
     \text{if}\ \overline{AD}\notperp\overline{CD}\ \text{at}\ D; \label{eq_45a}  \\
  &F_{\eta}(-1,1) - T_{CD}(-1) = 0, & 
  \text{if}\ \overline{AD}\perp\overline{CD}\ \text{at}\ D. \label{eq_45b}
  \end{align}
\end{subequations}


The Neumann conditions~\eqref{eq_18} and~\eqref{eq_42} must be compatible
at vertex $C$. Evaluating~\eqref{eq_18a} and~\eqref{eq_42a} at vertex $C$ and combining
them leads to
\begin{subequations} \label{eq_46}
  \begin{align}
    &
    V_{\xi}(1,1) = \frac{T_{BC}(1) - S_{BC}(1)T_{CD}(1)}{1-S_{BC}(1)S_{CD}(1)}
    = F_{\xi}^a(1,1), \label{eq_46a} \\
    &
    V_{\eta}(1,1) = \frac{T_{CD}(1) - S_{CD}(1)T_{BC}(1)}{1-S_{BC}(1)S_{CD}(1)}
    = F_{\eta}^a(1,1). \label{eq_46b}
  \end{align}
\end{subequations}
Note that $S_{BC}(1)S_{CD}(1)<1$ by Cauchy-Schwarz inequality for a non-singular Jacobian matrix
at $C$.
By differentiating~\eqref{eq_18b} with respect to (w.r.t.) $\eta$
and~\eqref{eq_42b} w.r.t.~$\xi$, and evaluating them at vertex $C$,
we get
\begin{subequations}\label{eq_47}
\begin{align}
  & V_{\xi\eta}(1,1) = T'_{BC}(1) - S'_{BC}(1)F_{\eta}^a(1,1) - S_{BC}(1)V_{\eta\eta}(1,1)
  =F_{\xi\eta}(1,1),
  \label{eq_47a} \\
  & V_{\eta\xi}(1,1) = T_{CD}'(1) - S_{CD}'(1)F_{\xi}^a(1,1) - S_{CD}(1)V_{\xi\xi}(1,1)
  = F_{\eta\xi}(1,1),
  \label{eq_47b}
\end{align}
\end{subequations}
where we have used~\eqref{eq_46}.
A combination of~\eqref{eq_47a} and \eqref{eq_47b}
(requiring $V_{\xi\eta}(1,1)=V_{\eta\xi}(1,1)$) leads to the following
compatibility constraints,
\begin{subequations}
  \begin{align}
    &
    V_{\xi\xi}(1,1)= \frac{S_{BC}(1)}{S_{CD}(1)}V_{\eta\eta}(1,1)
    + \frac{R_C}{S_{CD}(1)}
    = F_{\xi\xi}(1,1), & 
    \text{if}\ \overline{BC}\notperp\overline{CD}\ \text{at}\ C, \label{eq_48a} \\
    &
    T_{CD}'(1) - S_{CD}'(1)F_{\xi}^a(1,1) = T_{BC}'(1) - S_{BC}'(1)F_{\eta}^a(1,1),
    & 
    \text{if}\ \overline{BC}\perp\overline{CD}\ \text{at}\ C, \label{eq_48b}
  \end{align}
\end{subequations}
where
$R_C=\left[T_{CD}'(1) - S_{CD}'(1)F_{\xi}^a(1,1) \right]
- \left[T_{BC}'(1) - S_{BC}'(1)F_{\eta}^a(1,1) \right]$,
and we have used the fact that
$S_{BC}(1)=S_{CD}(1)=0$ when $\overline{BC}$ and $\overline{CD}$
are orthogonal at vertex $C$.
Equation~\eqref{eq_48b} imposes a constraint on the prescribed Neumann
boundary data when $\overline{BC}$ and $\overline{CD}$ are orthogonal at $C$,
and if they are not orthogonal,
equation~\eqref{eq_48a} imposes a constraint between $V_{\xi\xi}(1,1)$ and
$V_{\eta\eta}(1,1)$ at vertex $C$.


Our goal is to formulate the field function $V(\xi,\eta)$,
for $(\xi,\eta)\in\Omega_{st}$,
so that it exactly
satisfies the boundary
conditions~\eqref{eq_14a}, \eqref{eq_14d},
\eqref{eq_18}, and~\eqref{eq_42}, together with
the compatibility constraints~\eqref{eq_21a}, \eqref{eq_22a},
\eqref{eq_44a}, \eqref{eq_45a}, \eqref{eq_46}, \eqref{eq_48a}
and~\eqref{eq_47a}.

\begin{table}
  \centering
  \begin{tabular}{c|ccc|ccc|ccc}
    \hline
    $\xi$ & $\rho_0(\xi)$ & $\rho_0'(\xi)$ & $\rho_0''(\xi)$
    & $\rho_1(\xi)$ & $\rho_1'(\xi)$ & $\rho_1''(\xi)$
    & $\upsilon_0(\xi)$ & $\upsilon_0'(\xi)$ & $\upsilon_0''(\xi)$ \\ \hline
    $-1$ & 1 & 0 & 0 & 0 & 0 & 0 & 0 & 1 & 0 \\
    $1$ & 0 & 0 & 0 & 1 & 0 & 0 & 0 & 0 & 0 \\
    \hline\hline
    $\xi$ & $\upsilon_1(\xi)$ & $\upsilon_1'(\xi)$ & $\upsilon_1''(\xi)$
    & $\omega_0(\xi)$ & $\omega_0'(\xi)$ & $\omega_0''(\xi)$
    & $\omega_1(\xi)$ & $\omega_1'(\xi)$ & $\omega_1''(\xi)$ \\ \hline
    $-1$ & 0 & 0 & 0 & 0 & 0 & 1 & 0 & 0 & 0 \\
    $1$ & 0 & 1 & 0 & 0 & 0 & 0 & 0 & 0 & 1 \\
    \hline
  \end{tabular}
  \caption{Interpolation properties of $C^2$ Hermite interpolation
    polynomials defined on $\xi\in[-1,1]$.
  }
  \label{tab_2}
\end{table}

To facilitate the subsequent discussions, we recall the $C^2$ Hermite interpolation
polynomials $\rho_0(\xi)$, $\rho_1(\xi)$,
$\upsilon_0(\xi)$, $\upsilon_1(\xi)$, 
$\omega_0(\xi)$ and $\omega_1(\xi)$ defined on $\xi\in[-1,1]$
that satisfy the interpolation properties listed in Table~\ref{tab_2}.
These polynomials are given by
\begin{equation}
  \begin{array}{lll}
    \rho_0(\xi) = \phi_0^3(\xi)\left[1+3\phi_1(\xi)+6\phi_1^2(\xi) \right],
    & \upsilon_0(\xi) = 2\phi_0^3(\xi)\phi_1(\xi)\left[1+3\phi_1(\xi) \right],
    & \omega_0(\xi) = 2\phi_0^3(\xi)\phi_1^2(\xi), \\[2pt]
    \rho_1(\xi) = \phi_1^3(\xi)\left[1+3\phi_0(\xi)+6\phi_0^2(\xi) \right],
    & \upsilon_1(\xi) = -2\phi_1^3(\xi)\phi_0(\xi)\left[1+3\phi_0(\xi) \right],
    & \omega_1(\xi) = 2\phi_1^3(\xi)\phi_0^2(\xi),
  \end{array}
\end{equation}
where $\phi_0(\xi)$ and $\phi_1(\xi)$ are defined in~\eqref{eq_7}.
Besides the constants $\lambda_B$ and $\lambda_C$ defined in~\eqref{eq_25},
we define an additional constant $\lambda_D$ to flag whether the boundaries
$\overline{CD}$ and $\overline{AD}$ are orthogonal at vertex $D$,
\begin{equation}
  \lambda_D=\left\{ \begin{array}{ll}
    0, & \text{if}\ \overline{CD}\perp\overline{AD}\ \text{at}\ D, \\
    1, & \text{if}\ \overline{CD}\notperp\overline{AD}\ \text{at}\ D.
  \end{array}
  \right.
\end{equation}


We follow the four-step procedure as described in Section~\ref{sec_231}
to develop the general form for $V(\xi,\eta)$.
First, we note that the conditions~\eqref{eq_21a}, \eqref{eq_22a}
and~\eqref{eq_46b} are actually constraints on $V(1,\eta)$,
which can be satisfied by the following profile on $\overline{BC}$,
\begin{equation}\label{eq_51}
  F(1,\eta) = F(1,-1)\rho_0(\eta) + \lambda_B F_{\eta}(1,-1)\upsilon_0(\eta)
  + F_{\eta}(1,1)\upsilon_1(\eta),
\end{equation}
where the constant $\lambda_B$ ensures that the condition~\eqref{eq_22a}
is enforced only when $\overline{AB}$ and $\overline{BC}$
are not orthogonal at $B$, and
\begin{equation}\label{eq_52}
  F_{\eta}(1,-1)=F_{\eta}^a(1,-1), \quad
  F_{\eta}(1,1) = F_{\eta}^a(1,1),
\end{equation}
with $F_{\eta}^a(1,-1)$ and $F_{\eta}^a(1,1)$ given in~\eqref{eq_22a}
and~\eqref{eq_46b}.
The conditions~\eqref{eq_44a}, \eqref{eq_45a}, \eqref{eq_46a} and~\eqref{eq_48a}
are actually constraints on $V(\xi,1)$. They can be satisfied by
the following profile on $\overline{CD}$,
\begin{equation}\label{eq_53}
  F(\xi,1) = F(-1,1)\rho_0(\xi) + \lambda_D F_{\xi}(-1,1)\upsilon_0(\xi)
  + F_{\xi}(1,1)\upsilon_1(\xi) + \lambda_C F_{\xi\xi}(1,1)\omega_1(\xi),
\end{equation}
where the constants $\lambda_D$ and $\lambda_C$ ensure that
the conditions~\eqref{eq_45a} and \eqref{eq_48a} are only imposed
when the boundary curves are not orthogonal at $D$ or $C$,
$F_{\xi\xi}(1,1)$ is given in~\eqref{eq_48a}, and
\begin{equation}\label{eq_54}
  F_{\xi}(-1,1) = F_{\xi}^a(-1,1), \quad
  F_{\xi}(1,1) = F_{\xi}^a(1,1),
\end{equation}
with $F_{\xi}^a(-1,1)$ and $F_{\xi}^a(1,1)$ defined in~\eqref{eq_45a}
and \eqref{eq_46a}.
Note that the profile~\eqref{eq_51} for $\overline{BC}$ and
the profile~\eqref{eq_53} for $\overline{CD}$ are compatible
at vertex $C$, resulting in
$ 
  F(1,1) = 0.
$ 

With $F(1,\eta)$ and $F(\xi,1)$ introduced above, our task
is then reduced to: find $V(\xi,\eta)$ such that
\begin{subequations}\label{eq_56}
  \begin{align}
    & V(-1,\eta) = F(-1,\eta), \label{eq_56a} \\
    & V(1,\eta) = F(1,\eta) =
    F(1,-1)\rho_0(\eta) + \lambda_B F_{\eta}(1,-1)\upsilon_0(\eta)
    + F_{\eta}(1,1)\upsilon_1(\eta), \\
    & V_{\xi}(1,\eta) = F_{\xi}(1,\eta), \\
    & V(\xi,-1) = F(\xi,-1), \label{eq_56d} \\
    & V(\xi,1) = F(\xi,1) =
    F(-1,1)\rho_0(\xi) + \lambda_D F_{\xi}(-1,1)\upsilon_0(\xi)
    + F_{\xi}(1,1)\upsilon_1(\xi) + \lambda_C F_{\xi\xi}(1,1)\omega_1(\xi), \\
    & V_{\eta}(\xi,1) = F_{\eta}(\xi,1), \\
    & V_{\xi\eta}(1,1) = F_{\xi\eta}(1,1),
  \end{align}
\end{subequations}
where $F(-1,\eta)$ and $F(\xi,-1)$ are given in~\eqref{eq_13d} and~\eqref{eq_13a},
 $F_{\xi}(1,\eta)$ is given in~\eqref{eq_18b},
$F_{\eta}(\xi,1)$ is given in~\eqref{eq_42b},
and $F_{\xi\eta}(1,1)$ is given in~\eqref{eq_47a}, and
we have used~\eqref{eq_51} and \eqref{eq_53}.

The transfinite interpolation for the conditions in~\eqref{eq_56}
is given by
\begin{align}
  PF(\xi,\eta) =&\ (P_1\oplus P_2)F(\xi,\eta)
  = P_1F(\xi,\eta) + P_2F(\xi,\eta) - P_1P_2F(\xi,\eta) \notag \\
  =&\ F(-1,\eta)\rho_0(\xi) + F_{\xi}(1,\eta)\upsilon_1(\xi)
  + F(\xi,-1)\rho_0(\eta) + F_{\eta}(\xi,1)\upsilon_1(\eta) \notag \\
  & -\left[F(-1,-1)\rho_0(\eta) + F_{\eta}(-1,1)\upsilon_1(\eta) \right]\rho_0(\xi)
  - \left[F_{\xi}(1,-1)\rho_0(\eta)
    + F_{\xi\eta}(1,1)\upsilon_1(\eta) \right]\upsilon_1(\xi) \notag \\
  & + \lambda_B F_{\eta}(1,-1)\upsilon_0(\eta)\rho_1(\xi)
  + \left[\lambda_D F_{\xi}(-1,1)\upsilon_0(\xi)
    +\lambda_C F_{\xi\xi}(1,1)\omega_1(\xi) \right]\rho_1(\eta), \label{eq_57}
\end{align}
where
\begin{equation}
  \left\{
  \begin{array}{l}
    P_1F(\xi,\eta) = F(-1,\eta)\rho_0(\xi) + F_{\xi}(1,\eta)\upsilon_1(\xi)
    + F(1,\eta)\rho_1(\xi) \\
    \qquad\qquad\ = F(-1,\eta)\rho_0(\xi) + F_{\xi}(1,\eta)\upsilon_1(\xi) \\
    \qquad\qquad\quad\
    + \left[F(1,-1)\rho_0(\eta) + \lambda_B F_{\eta}(1,-1)\upsilon_0(\eta)
  + F_{\eta}(1,1)\upsilon_1(\eta) \right]\rho_1(\xi),
    \\
    P_2F(\xi,\eta) = F(\xi,-1)\rho_0(\eta) + F_{\eta}(\xi,1)\upsilon(\eta)
    + F(\xi,1)\rho_1(\eta) \\
    \qquad\qquad\ = F(\xi,-1)\rho_0(\eta) + F_{\eta}(\xi,1)\upsilon(\eta) \\
    \qquad\qquad\quad\
    + \left[F(-1,1)\rho_0(\xi) + \lambda_D F_{\xi}(-1,1)\upsilon_0(\xi)
  + F_{\xi}(1,1)\upsilon_1(\xi) + \lambda_C F_{\xi\xi}(1,1)\omega_1(\xi) \right] \rho_1(\eta).
  \end{array}
  \right.
\end{equation}

This leads to
the preliminary form $V(\xi,\eta)$ for the conditions in~\eqref{eq_56},
\begin{align}\label{eq_n58}
  & V(\xi,\eta) = g(\xi,\eta) - Pg(\xi,\eta) + PF(\xi,\eta),
\end{align}
where $g(\xi,\eta)$ is a free (arbitrary) function, and
\begin{align}
  Pg(\xi,\eta) =&\ g(-1,\eta)\rho_0(\xi) + g_{\xi}(1,\eta)\upsilon_1(\xi)
  + g(\xi,-1)\rho_0(\eta) + g_{\eta}(\xi,1)\upsilon_1(\eta) \notag \\
  & -\left[g(-1,-1)\rho_0(\eta) + g_{\eta}(-1,1)\upsilon_1(\eta) \right]\rho_0(\xi)
  - \left[g_{\xi}(1,-1)\rho_0(\eta)
    + g_{\xi\eta}(1,1)\upsilon_1(\eta) \right]\upsilon_1(\xi) \notag \\
  & + \lambda_B g_{\eta}(1,-1)\upsilon_0(\eta)\rho_1(\xi)
  + \left[\lambda_D g_{\xi}(-1,1)\upsilon_0(\xi)
    +\lambda_C g_{\xi\xi}(1,1)\omega_1(\xi) \right]\rho_1(\eta). \label{eq_60}
\end{align}
$V(\xi,\eta)$ from~\eqref{eq_n58} has the following properties,
\begin{subequations}\label{eq_61}
  \begin{align}
    & V_{\eta\eta}(1,1) = g_{\eta\eta}(1,1), \\
    & V_{\eta}(1,\eta) = g_{\eta}(1,\eta) - \left[g(1,-1)-F(1,-1) \right]\rho_0'(\eta)
    - \left[g_{\eta}(1,1) - F_{\eta}(1,1) \right]\upsilon_1'(\eta) \notag \\
    & \qquad\qquad\ 
    -\lambda_B\left[g_{\eta}(1,-1)-F_{\eta}(1,-1) \right]\upsilon_0'(\eta), \\
    & V_{\xi}(\xi,1) = g_{\xi}(\xi,1) - \left[g(-1,1) - F(-1,1) \right]\rho_0'(\xi)
    - \left[g_{\xi}(1,1) - F_{\xi}(1,1) \right]\upsilon_1'(\xi) \notag \\
    & \qquad\qquad\ 
    - \lambda_D\left[g_{\xi}(-1,1) - F_{\xi}(-1,1) \right]\upsilon_0'(\xi)
    - \lambda_C\left[g_{\xi\xi}(1,1) - F_{\xi\xi}(1,1) \right]\omega_1'(\xi).
  \end{align}
\end{subequations}

Employing~\eqref{eq_18b}, \eqref{eq_42b}, \eqref{eq_47a}, \eqref{eq_48a},
\eqref{eq_52}, \eqref{eq_54}, and~\eqref{eq_61}, we update
the transfinite interpolation $PF(\xi,\eta)$ in~\eqref{eq_57} by,
\begin{align}
  PF^g(\xi,\eta) =&\ F(-1,\eta)\rho_0(\xi) + F_{\xi}^g(1,\eta)\upsilon_1(\xi)
  + F(\xi,-1)\rho_0(\eta) + F_{\eta}^g(\xi,1)\upsilon_1(\eta) \notag \\
  & -\left[F(-1,-1)\rho_0(\eta) + F_{\eta}(-1,1)\upsilon_1(\eta) \right]\rho_0(\xi)
  - \left[F_{\xi}(1,-1)\rho_0(\eta)
    + F_{\xi\eta}^g(1,1)\upsilon_1(\eta) \right]\upsilon_1(\xi) \notag \\
  & + \lambda_B F_{\eta}^a(1,-1)\upsilon_0(\eta)\rho_1(\xi)
  + \left[\lambda_D F_{\xi}^a(-1,1)\upsilon_0(\xi)
    +\lambda_C F_{\xi\xi}^g(1,1)\omega_1(\xi) \right]\rho_1(\eta), \label{eq_62}
\end{align}
where
\begin{subequations}
  \begin{align}
    & F_{\xi\xi}^g(1,1) = \frac{S_{BC}(1)}{S_{CD}(1)}g_{\eta\eta}(1,1) + \frac{R_C}{S_{CD}(1)},
    \\
    & F_{\xi\eta}^g(1,1) = T_{BC}'(1)-S_{BC}'(1)F_{\eta}^a(1,1) - S_{BC}(1)g_{\eta\eta}(1,1), \\
    & F_{\xi}^g(1,\eta) = T_{BC}(\eta) - S_{BC}(\eta)\left\{
    g_{\eta}(1,\eta) - \left[g(1,-1)-F(1,-1) \right]\rho_0'(\eta)  
    - \left[g_{\eta}(1,1) - F_{\eta}^a(1,1) \right]\upsilon_1'(\eta) \right. \notag \\
    &\hspace{1.8in} \left.
    -\lambda_B\left[g_{\eta}(1,-1)-F_{\eta}^a(1,-1) \right]\upsilon_0'(\eta)
    \right\}, \\
    & F_{\eta}^g(\xi,1) = T_{CD}(\xi) - S_{CD}(\xi)\left\{
    g_{\xi}(\xi,1) - \left[g(-1,1) - F(-1,1) \right]\rho_0'(\xi)
    - \left[g_{\xi}(1,1) - F_{\xi}^a(1,1) \right]\upsilon_1'(\xi) \right. \notag \\
    &\hspace{1.5in} \left.
    - \lambda_D\left[g_{\xi}(-1,1) - F_{\xi}^a(-1,1) \right]\upsilon_0'(\xi)
    - \lambda_C\left[g_{\xi\xi}(1,1) - F_{\xi\xi}^g(1,1) \right]\omega_1'(\xi)
    \right\}.
  \end{align}
\end{subequations}

The final form for $V(\xi,\eta)$ is then given by
\begin{equation}\label{eq_64}
  V(\xi,\eta) = g(\xi,\eta) - Pg(\xi,\eta) + PF^g(\xi,\eta)
\end{equation}
where $g(\xi,\eta)$ is a free function, $Pg(\xi,\eta)$ is given by~\eqref{eq_60},
and $PF^g(\xi,\eta)$ is given by~\eqref{eq_62}.
In this expression $F(-1,-1)=\lim_{\xi\rightarrow -1}F(\xi,-1)=u_A$, and
$F_{\eta}(-1,1)$ and $F_{\xi}(1,-1)$ are given by~\eqref{eq_44b}
and~\eqref{eq_21c}. $F_{\xi}^a(-1,1)$ and $F_{\eta}^a(1,-1)$ are given by~\eqref{eq_45a}
and~\eqref{eq_22a}. $F_{\xi}^a(1,1)$ and $F_{\eta}^a(1,1)$ are given
in~\eqref{eq_46}.

\begin{theorem}
  $V(\xi,\eta)$ given by~\eqref{eq_64} satisfies the Dirichlet conditions~\eqref{eq_56a}
  and~\eqref{eq_56d} and the Neumann conditions~\eqref{eq_18a} and~\eqref{eq_42a},
  for any $g(\xi,\eta)$ therein that is sufficiently differentiable.
\end{theorem}
\begin{proof}
  To verify~\eqref{eq_56a} and~\eqref{eq_56d}, one only needs to notice that
  $F_{\eta}^g(-1,1)=F_{\eta}(-1,1)$ and $F_{\xi}^g(1,-1)=F_{\xi}(1,-1)$,
  due to~\eqref{eq_45} and~\eqref{eq_22}. Consequently
  $PF^g(-1,\eta)=F(-1,\eta)$ and $PF^g(\xi,-1) = F(\xi,-1)$.

  To verify~\eqref{eq_42a}, one notes that $F_{\xi,\eta}^g(1,1) = F_{\xi\eta}^g(1,1)$.
  Hence $PF_{\eta}^g(\xi,1)=F_{\eta}^g(\xi,1)$ and $V_{\eta}(\xi,1)=F_{\eta}^g(\xi,1)$.
  On the other hand,
  \begin{align*}
    V_{\xi}(\xi,1) =&\ g_{\xi}(\xi,1) - \left[g(-1,1)-F(-1,1) \right]\rho_0'(\xi)
    -\left[g_{\xi}(1,1)-F_{\xi}^a(1,1) \right]\upsilon_1'(\xi) \notag \\
    &
    - \lambda_D\left[g_{\xi}(-1,1)-F_{\xi}^a(-1,1) \right]\upsilon_0'(\xi)
    -\lambda_C\left[g_{\xi\xi}(1,1) - F_{\xi\xi}^g(1,1) \right]\omega_1'(\xi).
  \end{align*}
  Therefore~\eqref{eq_42a} holds.

  To verify~\eqref{eq_18a}, one notes that $F_{\eta,\xi}^g(1,1) = F_{\xi\eta}^g(1,1)$.
  Hence $PF_{\xi}^g(1,\eta) = F_{\xi}^g(1,\eta)$ and $V_{\xi}(1,\eta)=F_{\xi}^g(1,\eta)$.
  On the other hand,
  \begin{align*}
    V_{\eta}(1,\eta) =&\ g_{\eta}(1,\eta)-\left[g(1,-1)-F(1,-1) \right]\rho_0'(\eta)
    -\left[g_{\eta}(1,1)-F_{\eta}^a(1,1) \right]\upsilon_1'(\eta) \notag \\
    &\qquad\qquad
    -\lambda_B\left[g_{\eta}(1,-1)-F_{\eta}^a(1,-1) \right]\upsilon_0'(\eta).
  \end{align*}
  Therefore~\eqref{eq_18a} holds.
\end{proof}


\begin{remark}
  When the domain involves more than two Neumann boundaries, the compatibility
  constraints as discussed above exist at any vertex where two Neumann boundaries
  intersect, and those constraints discussed in Section~\ref{sec_231} exist at
  any vertex where a Neumann boundary and a Dirichlet boundary intersect.
  The field function that exactly satisfies these conditions can be formulated
  analogously using the four-step procedure described in Section~\ref{sec_231}.
\end{remark}


\subsection{Enforcing Robin Boundary Conditions}
\label{sec_24}

We next look into how to enforce Robin boundary conditions (RBC) exactly
on the general quadrilateral domain $\Omega$ as shown in Figure~\ref{fg_1}(a).
The formulation for Robin condition is largely similar to that
for the Neumann condition, apart from the complication caused by
the unknown function value on the Robin boundary.
Here we assume that the domain involves a combination of
Dirichlet and Robin type boundaries,
with at least one Robin boundary and one Dirichlet boundary.
The cases when a Robin boundary only intersects with Dirichlet boundaries
and when two Robin boundaries intersect with each other
are discussed individually below.

\subsubsection{When Robin Boundary Only Intersects with Dirichlet Boundary}
\label{sec_241}

For this case we focus on the setting in which the domain has
a single Robin boundary,
with the rest being Dirichlet types.
Without loss of generality, we assume that the Robin condition is imposed
on $\overline{BC}$. The goal is to formulate the
field function $u(\mbs x)$, for $\mbs x\in\Omega=\overline{ABCD}$,
which satisfies the following conditions,
\begin{subequations}\label{eq_65}
  \begin{align}
    & \left.u\right|_{\mbs x\in \overline{AB}} = u_{AB}(\mbs x), \label{eq_65a} \\
    & \left.\mbs n\cdot \nabla u\right|_{\mbs x\in \overline{BC}}
    + \alpha_{BC}\left.u\right|_{\mbs x\in\overline{BC}}
    = u_{rBC}(\mbs x),
    \label{eq_65b}\\
    & \left. u\right|_{\mbs x\in \overline{CD}} = u_{CD}(\mbs x),
     \label{eq_65c} \\
    & \left.u\right|_{\mbs x\in \overline{AD}} = u_{AD}(\mbs x), \label{eq_65d}
  \end{align}
\end{subequations}
where $\alpha_{BC}$ is a prescribed constant, and $u_{rBC}(\mbs x)$ denotes
the prescribed Robin boundary function.

By leveraging the map $\mbs x(\xi,\eta)$, we transform $u(\mbs x)$ 
into $V(\xi,\eta)$ according to~\eqref{eq_12}, and the conditions~\eqref{eq_65a}
and~\eqref{eq_65c}--\eqref{eq_65d} into~\eqref{eq_14a} and~\eqref{eq_14c}--\eqref{eq_14d}
according to~\eqref{eq_13a} and \eqref{eq_13c}--\eqref{eq_13d}.
The Robin condition~\eqref{eq_65b} is transformed into
\begin{subequations}\label{eq_66}
  \begin{align}
    & V_{\xi}(1,\eta) + S_{BC}(\eta)V_{\eta}(1,\eta)
    + \alpha_{BC} W_{BC}(\eta) V(1,\eta) = T_{rBC}(\eta),
    \quad \eta\in[-1,1], \label{eq_66a} \\
    \text{or}\ &
    V_{\xi}(1,\eta) = T_{rBC}(\eta) - S_{BC}(\eta)V_{\eta}(1,\eta)
    - \alpha_{BC} W_{BC}(\eta) V(1,\eta) = F_{\xi}(1,\eta), \label{eq_66b}
  \end{align}
\end{subequations}
where $S_{BC}(\eta)$ and $K_{xBC}(\eta)$ are defined in~\eqref{eq_19}, and
\begin{align}\label{eq_67}
  & W_{BC}(\eta) = \frac{1}{K_{xBC}(\eta)}, \quad
  T_{rBC}(\eta) = F_{rBC}(\eta)W_{BC}(\eta), \quad
  F_{rBC}(\eta) = u_{rBC}(\mbs x(1,\eta)).
\end{align}
These conditions must be compatible at the shared vertices.

The compatibility between the Robin condition~\eqref{eq_66} and
the Dirichlet condition~\eqref{eq_14a} at vertex $B$
results in~\eqref{eq_21a} and~\eqref{eq_21c}, together with
\begin{subequations}\label{eq_68}
  \begin{align}
    &V_{\eta}(1,-1) =  \frac{1}{S_{BC}(-1)}\left[ T_{rBC}(-1) - F_{\xi}(1,-1) 
    -\alpha_{BC}W_{BC}(-1)F(1,-1) \right] \notag \\
    &\hspace{0.55in}  =F_{\eta}^a(1,-1),
    &
     \text{if}\ \overline{AB}\notperp\overline{BC}\ \text{at}\ B, \label{eq_68a}  \\
  &F_{\xi}(1,-1) + \alpha_{BC}W_{BC}(-1)F(1,-1) - T_{rBC}(-1) = 0, & 
  \text{if}\ \overline{AB}\perp\overline{BC}\ \text{at}\ B, \label{eq_68b}
  \end{align}
\end{subequations}
by noting that $S_{BC}(-1)=0$ when $\overline{AB}\perp\overline{BC}$ at $B$.
The compatibility between~\eqref{eq_66} and~\eqref{eq_14c} at vertex $C$
results in~\eqref{eq_21b} and~\eqref{eq_21d}, together with
\begin{subequations}\label{eq_69}
  \begin{align}
    &V_{\eta}(1,1) =  \frac{1}{S_{BC}(1)}\left[ T_{rBC}(1) - F_{\xi}(1,1)
      - \alpha_{BC}W_{BC}(1)F(1,1) \right]
    =F_{\eta}^a(1,1),
    &
     \text{if}\ \overline{CD}\notperp\overline{BC}\ \text{at}\ C, \label{eq_69a} \\
  &F_{\xi}(1,1) + \alpha_{BC}W_{BC}(1)F(1,1) - T_{rBC}(1) = 0, & 
  \text{if}\ \overline{CD}\perp\overline{BC}\ \text{at}\ C. \label{eq_69b}
  \end{align}
\end{subequations}

The conditions~\eqref{eq_14a}, \eqref{eq_14c}, \eqref{eq_14d}, \eqref{eq_66},
\eqref{eq_68a}, \eqref{eq_69a}, and~\eqref{eq_21a}--\eqref{eq_21b}
constitute the constraints that the function $V(\xi,\eta)$
to be formulated must satisfy.

We follow the four-step procedure, to first introduce the same
transfinite interpolation as in~\eqref{eq_30}, where
$F_{\eta}(1,-1)$ and $F_{\eta}(1,1)$ are given by~\eqref{eq_n27},
with $F_{\eta}^a(1,-1)$ and $F_{\eta}^a(1,1)$ therein now given
by~\eqref{eq_68a} and~\eqref{eq_69a}, and $F_{\xi}(1,\eta)$ is now given
by~\eqref{eq_66b}.
The preliminary form for $V(\xi,\eta)$ is given by~\eqref{eq_31},
in which $Pg(\xi,\eta)$ is given by~\eqref{eq_32}.
This preliminary form has the following properties on $\overline{BC}$,
\begin{subequations}\label{eq_70}
  \begin{align}
    V(1,\eta) =&\ g(1,\eta) - \left[g(1,-1)-F(1,-1) \right]\varphi_0(\eta)
    - \left[g(1,1)-F(1,1) \right]\varphi_1(\eta) \notag \\
    &\ 
    - \lambda_B\left[g_{\eta}(1,-1)-F_{\eta}(1,-1) \right]\psi_0(\eta)
    - \lambda_C\left[g_{\eta}(1,1)-F_{\eta}(1,1) \right]\psi_1(\eta), \\
    V_{\eta}(1,\eta) =&\ g_{\eta}(1,\eta) - \left[g(1,-1)-F(1,-1) \right]\varphi_0'(\eta)
    - \left[g(1,1)-F(1,1) \right]\varphi_1'(\eta) \notag \\
    &\ 
    - \lambda_B\left[g_{\eta}(1,-1)-F_{\eta}(1,-1) \right]\psi_0'(\eta)
    - \lambda_C\left[g_{\eta}(1,1)-F_{\eta}(1,1) \right]\psi_1'(\eta).
  \end{align}
\end{subequations}

In light of~\eqref{eq_70} and \eqref{eq_n27},
we update $F_{\xi}(1,\eta)$ in~\eqref{eq_66b} by
\begin{align}\label{eq_71}
  F_{\xi}^g(1,\eta) =&\ T_{rBC}(\eta) - S_{BC}(\eta)\left\{
  g_{\eta}(1,\eta) - \left[g(1,-1)-F(1,-1) \right]\varphi_0'(\eta)
  - \left[g(1,1)-F(1,1) \right]\varphi_1'(\eta) \right. \notag \\
  &\qquad \left.
  - \lambda_B\left[g_{\eta}(1,-1)-F_{\eta}^a(1,-1) \right]\psi_0'(\eta)
  - \lambda_C\left[g_{\eta}(1,1)-F_{\eta}^a(1,1) \right]\psi_1'(\eta)  \right\} \notag \\
  &
  -\alpha_{BC}W_{BC}(\eta)\left\{
  g(1,\eta) - \left[g(1,-1)-F(1,-1) \right]\varphi_0(\eta)
  - \left[g(1,1)-F(1,1) \right]\varphi_1(\eta) \right. \notag \\
  &\qquad \left.
  - \lambda_B\left[g_{\eta}(1,-1)-F_{\eta}^a(1,-1) \right]\psi_0(\eta)
    - \lambda_C\left[g_{\eta}(1,1)-F_{\eta}^a(1,1) \right]\psi_1(\eta)
  \right\}.
\end{align}
Therefore, the updated transfinite interpolation is given by~\eqref{eq_34b},
in which $F_{\xi}^g(1,\eta)$ is now given by~\eqref{eq_71},
and $F_{\eta}^a(1,-1)$ and $F_{\eta}^a(1,1)$ are now given
by~\eqref{eq_68a} and~\eqref{eq_69a}.
The final $V(\xi,\eta)$ has the same form as in~\eqref{eq_64}.

\begin{theorem}
  $V(\xi,\eta)$ given by~\eqref{eq_64}, with $F_{\xi}^g(1,\eta)$, $F_{\eta}^a(1,-1)$
  and $F_{\eta}^a(1,1)$ therein given by~\eqref{eq_71},~\eqref{eq_68a}
  and~\eqref{eq_69a}, satisfies the Dirichlet conditions~\eqref{eq_14a} and
  \eqref{eq_14c}--\eqref{eq_14d} and the Robin condition~\eqref{eq_66a},
  for any $g(\xi,\eta)$ therein that is sufficiently differentiable.
\end{theorem}
\begin{proof}
  By verification.
\end{proof}

\begin{remark}
  When two Robin boundaries are imposed on opposite sides of the quadrilateral
  domain, with
  the rest being Dirichlet boundaries, the general form for $V(\xi,\eta)$
  that exactly satisfies these conditions
  can be formulated analogously by following the four-step procedure.
\end{remark}

\subsubsection{When Two Robin Boundaries Intersect at a Vertex}
\label{sec_242}

For this case we focus on the setting in which the Robin conditions are imposed
on two adjacent boundaries with the rest being Dirichlet boundaries.
Without loss of generality we assume that the Robin conditions are
imposed on the boundaries $\overline{BC}$ and $\overline{CD}$.

Specifically, we seek a field function $u(\mbs x)$, for $\mbs x\in\Omega=\overline{ABCD}$,
which satisfies the following boundary conditions,
\begin{subequations}\label{eq_72}
  \begin{align}
    & \left.u\right|_{\mbs x\in \overline{AB}} = u_{AB}(\mbs x), \label{eq_72a} \\
    & \left.\mbs n\cdot \nabla u\right|_{\mbs x\in \overline{BC}}
    + \alpha_{BC}\left. u\right|_{\mbs x\in \overline{BC}} = u_{rBC}(\mbs x),
    \label{eq_72b}\\
    & \left.\mbs n\cdot \nabla u\right|_{\mbs x\in \overline{CD}}
    +\alpha_{CD}\left. u\right|_{\mbs x\in \overline{CD}} = u_{rCD}(\mbs x),
     \label{eq_72c} \\
    & \left.u\right|_{\mbs x\in \overline{AD}} = u_{AD}(\mbs x), \label{eq_72d}
  \end{align}
\end{subequations}
where $\alpha_{BC}$ and $\alpha_{CD}$ are prescribed constants,
and $u_{rBC}(\mbs x)$ and $u_{rCD}(\mbs x)$ are prescribed Robin
boundary functions.

Employing the map $\mbs x(\xi,\eta)$,
we transform $u(\mbs x)$ into $V(\xi,\eta)$ according to~\eqref{eq_12},
and the Dirichlet conditions~\eqref{eq_72a} and~\eqref{eq_72d}
into~\eqref{eq_14a} and~\eqref{eq_14d} based on~\eqref{eq_13a}
and~\eqref{eq_13d}.
The Robin condition~\eqref{eq_72b} is accordingly
transformed into~\eqref{eq_66}.
The Robin condition~\eqref{eq_72c} is transformed into
\begin{subequations}\label{eq_73}
  \begin{align}
    &
    V_{\eta}(\xi,1) + S_{CD}(\xi)V_{\xi}(\xi,1) + \alpha_{CD}W_{CD}(\xi)V(\xi,1)
    = T_{rCD}(\xi), \quad \xi\in[-1,1], \label{eq_73a} \\
    \text{or}\ &
    V_{\eta}(\xi,1) = T_{rCD}(\xi) - S_{CD}(\xi)V_{\xi}(\xi,1)
    + \alpha_{CD}W_{CD}(\xi)V(\xi,1) = F_{\eta}(\xi,1),
    \label{eq_73b}
  \end{align}
\end{subequations}
where 
\begin{align}
  & W_{CD}(\xi) = \frac{1}{K_{yCD}(\xi)}, \quad
  T_{rCD}(\xi) = W_{CD}(\xi)F_{rCD}(\xi) = W_{CD}(\xi)u_{rCD}(\mbs x(\xi,1)),
\end{align}
and $S_{CD}(\xi)$ and $K_{yCD}(\xi)$ are defined in~\eqref{eq_43}.
The objective here is to formulate $V(\xi,\eta)$ to exactly
satisfy the conditions~\eqref{eq_14a}, \eqref{eq_14d}, \eqref{eq_66}
and~\eqref{eq_73}.


The boundary conditions must be compatible at the shared vertices.
The Robin condition~\eqref{eq_66} on $\overline{BC}$ and the Dirichlet
condition~\eqref{eq_14a} on $\overline{AB}$ should be compatible
at vertex $B$, leading to
the conditions~\eqref{eq_21a}, \eqref{eq_21c},
and~\eqref{eq_68}. Similarly, at vertex $D$ the Robin condition~\eqref{eq_73}
on $\overline{CD}$ and the Dirichlet condition~\eqref{eq_14d} on $\overline{AD}$
should be compatible with each other, inducing the
conditions~\eqref{eq_44a} and \eqref{eq_44b}, together with
\begin{subequations}\label{eq_75}
  \begin{align}
    &V_{\xi}(-1,1) = \frac{1}{S_{CD}(-1)}\left[
      T_{rCD}(-1)-F_{\eta}(-1,1) - \alpha_{CD}W_{CD}(-1)F(-1,1)
      \right] \notag \\
    &\hspace{0.55in}
    = F_{\xi}^a(-1,1),
    & \text{if}\ \overline{AD}\notperp\overline{CD}\ \text{at}\ D, \label{eq_75a} \\
    & F_{\eta}(-1,1) + \alpha_{CD}W_{CD}(-1)F(-1,1) - T_{rCD}(-1)=0,
    & \text{if}\ \overline{AD}\perp\overline{CD}\ \text{at}\ D, \label{eq_75b}
  \end{align}
\end{subequations}
by noting that $S_{CD}(-1)=0$ if $\overline{AD}\perp\overline{CD}$ at $C$.

The Robin conditions~\eqref{eq_66} and~\eqref{eq_73} should be compatible
at the common vertex $C$. This leads to
\begin{subequations}
  \begin{align}
    & V_{\xi}(1,1) + S_{BC}(1)V_{\eta}(1,1) = T_{rBC}(1)
    - \alpha_{BC}W_{BC}(1)V(1,1), \\
    & S_{CD}(1)V_{\xi}(1,1) + V_{\eta}(1,1) = T_{rCD}(1) - \alpha_{CD}W_{CD}(1)V(1,1).
  \end{align}
\end{subequations}
It follows that
\begin{subequations}\label{eq_77}
  \begin{align}
    & V_{\xi}(1,1) = F_{\xi}^a(1,1) - F_{\xi}^b(1,1)V(1,1) = F_{\xi}(1,1), \label{eq_77a} \\
    & V_{\eta}(1,1) = F_{\eta}^a(1,1) - F_{\eta}^b(1,1)V(1,1) = F_{\eta}(1,1), \label{eq_77b}
  \end{align}
\end{subequations}
where
\begin{subequations}\label{eq_78}
\begin{align}
  &
  \begin{bmatrix} F_{\xi}^a(1,1)  \\
    F_{\eta}^a(1,1)  \end{bmatrix}
  = \frac{1}{1-S_{BC}(1)S_{CD}(1)}\begin{bmatrix}
    T_{BC}(1)-S_{BC}(1)T_{CD}(1)  \\
    T_{CD}(1)-S_{CD}(1)T_{BC}(1)
  \end{bmatrix}, \\
  &
  \begin{bmatrix}  F_{\xi}^b(1,1) \\[3pt]
    F_{\eta}^b(1,1) \end{bmatrix}
  = \frac{1}{1-S_{BC}(1)S_{CD}(1)}\begin{bmatrix}
    \alpha_{BC}W_{BC}(1) - S_{BC}(1)\alpha_{CD}W_{CD}(1) \\
    \alpha_{CD}W_{CD}(1) - S_{CD}(1)\alpha_{BC}W_{BC}(1)
  \end{bmatrix}.
\end{align}
\end{subequations}
Differentiating~\eqref{eq_66b} w.r.t.~$\eta$ and evaluating it at vertex $C$,
we get
\begin{align}\label{eq_79}
  & V_{\xi\eta}(1,1) = Q_{BC}^a(1) - S_{BC}(1)V_{\eta\eta}(1,1) + Q_{BC}^b(1)V(1,1)
  = F_{\xi\eta}(1,1),
\end{align}
where \eqref{eq_77b} has been used, and
\begin{subequations}
\begin{align}
  & Q_{BC}^a(1) = T_{BC}'(1) - \left[S_{BC}'(1) + \alpha_{BC}W_{BC}(1) \right]F_{\eta}^a(1,1), \\
  &
  Q_{BC}^b(1) = \left[S_{BC}'(1)+\alpha_{BC}W_{BC}(1) \right]F_{\eta}^b(1,1)
  -\alpha_{BC}W_{BC}'(1).
\end{align}
\end{subequations}
Similarly, differentiating~\eqref{eq_73b} w.r.t.~$\xi$
and evaluating it at vertex $C$ result in
\begin{align}\label{eq_81}
  V_{\eta\xi}(1,1) = Q_{CD}^a(1,1) - S_{CD}(1)V_{\xi\xi}(1,1) + Q_{CD}^b(1)V(1,1)
  =F_{\eta\xi}(1,1),
\end{align}
where~\eqref{eq_77a} has been used and
\begin{subequations}
\begin{align}
  & Q_{CD}^a(1) = T_{CD}'(1) - \left[S_{CD}'(1) + \alpha_{CD}W_{CD}(1) \right]F_{\xi}^a(1,1), \\
  &
  Q_{CD}^b(1) = \left[S_{CD}'(1)+\alpha_{CD}W_{CD}(1) \right]F_{\xi}^b(1,1)
  -\alpha_{CD}W_{CD}'(1).
\end{align}
\end{subequations}
Combining~\eqref{eq_79} and~\eqref{eq_81} and requiring that
$V_{\xi\eta}(1,1)=V_{\eta\xi}(1,1)$, we have the following constraints,
\begin{subequations}\label{eq_83}
  \begin{align}
    V_{\xi\xi}(1,1) =&\
    \frac{S_{BC}(1)}{S_{CD}(1)}V_{\eta\eta}(1,1)
    - \frac{1}{S_{CD}(1)}\left[Q_{BC}^b(1) - Q_{CD}^b(1) \right]V(1,1) \notag \\
    & -\frac{1}{S_{CD}(1)}\left[Q_{BC}^a(1)-Q_{CD}^a(1) \right] = F_{\xi\xi}(1,1),
    \qquad \text{if}\ \overline{BC}\notperp\overline{CD}\ \text{at}\ C; \label{eq_83a} \\
    V(1,1) =&\ -\frac{Q_{BC}^a(1)-Q_{CD}^a(1)}{Q_{BC}^b(1)-Q_{CD}^b(1)} = F^a(1,1),
    \qquad \text{if}\ \overline{BC}\perp\overline{CD}\ \text{at}\ C\
    \text{and}\ Q_{BC}^b(1)\neq Q_{CD}^b(1); \label{eq_83b} \\
    Q_{BC}^a(1) =&\ Q_{CD}^a(1),
    \hspace{1.79in} \text{if}\ \overline{BC}\perp\overline{CD}\ \text{at}\ C\
    \text{and}\ Q_{BC}^b(1)= Q_{CD}^b(1). \label{eq_83c}
  \end{align}
\end{subequations}
Equation~\eqref{eq_83c} is a constraint on the prescribed Robin
boundary data $u_{rBC}$ and $u_{rCD}$ when $\overline{BC}\perp\overline{CD}$ at
vertex $C$ and $Q_{BC}^b(1)= Q_{CD}^b(1)$, while otherwise~\eqref{eq_83b}
is a constraint on the value $V(1,1)$ and~\eqref{eq_83a} imposes a
relation on $V_{\xi\xi}(1,1)$, $V_{\eta\eta}(1,1)$ and $V(1,1)$.

The equations~\eqref{eq_14a}, \eqref{eq_14d}, \eqref{eq_66b}, \eqref{eq_73b},
\eqref{eq_21a},  \eqref{eq_68a}, \eqref{eq_44a},
\eqref{eq_75a}, \eqref{eq_77}, \eqref{eq_83a}--\eqref{eq_83b},
and~\eqref{eq_79} constitute the set of constraints
the function $V(\xi,\eta)$ to be formulated must satisfy.
In addition to the flags $\lambda_B$, $\lambda_C$ and $\lambda_D$
introduced previously, we define another constant $\gamma_C$ to flag
whether $Q_{BC}^b(1)=Q_{CD}^b(1)$,
\begin{align}
  \gamma_C = \left\{
  \begin{array}{ll}
    0, & \text{if}\ Q_{BC}^b(1)=Q_{CD}^b(1), \\
    1, & \text{if}\ Q_{BC}^b(1)\neq Q_{CD}^b(1).
  \end{array}
  \right.
\end{align}

We follow the four-step procedure described in Section~\ref{sec_231}
to formulate $V(\xi,\eta)$. To handle the conditions~\eqref{eq_21a},
\eqref{eq_68a}, \eqref{eq_44a}, \eqref{eq_75a}, \eqref{eq_77a}--\eqref{eq_77b},
and~\eqref{eq_83a}--\eqref{eq_83b},
we introduce
\begin{subequations}\label{eq_85}
\begin{align}
  F(1,\eta) =&\ F(1,-1)\rho_0(\eta) + \lambda_BF_{\eta}(1,-1)\upsilon_0(\eta)
  + (1-\lambda_C)\gamma_C F(1,1)\rho_1(\eta) + F_{\eta}(1,1)\upsilon(\eta), \\
  F(\xi,1) =&\ F(-1,1)\rho_0(\xi) + \lambda_DF_{\xi}(-1,1)\upsilon_0(\xi)
  + (1-\lambda_C)\gamma_C F(1,1)\rho_1(\xi) + F_{\xi}(1,1)\upsilon_1(\xi) \notag \\
  & + \lambda_C F_{\xi\xi}(1,1)\omega_1(\xi),
\end{align}
\end{subequations}
where $F_{\eta}(1,1)$, $F_{\xi}(1,1)$ and $F_{\xi\xi}(1,1)$ are
defined in~\eqref{eq_77} and~\eqref{eq_83a}, and
\begin{align}
  F_{\eta}(1,-1)=F_{\eta}^a(1,-1), \quad F_{\xi}(-1,1) = F_{\xi}^a(-1,1), \quad
  F(1,1) = F^a(1,1),
\end{align}
with $F_{\eta}^a(1,1)$, $F_{\xi}^a(1,1)$ and $F^a(1,1)$ defined
in~\eqref{eq_68a}, \eqref{eq_75a} and~\eqref{eq_83b}.
Then the construction problem becomes the following: find $V(\xi,\eta)$ such that
\begin{subequations}\label{eq_87}
  \begin{align}
    & V(-1,\eta) = F(-1,\eta), \label{eq_87a} \\
    & V(1,\eta) = F(1,\eta)
    = F(1,-1)\rho_0(\eta) + \lambda_BF_{\eta}(1,-1)\upsilon_0(\eta)
    + (1-\lambda_C)\gamma_C F(1,1)\rho_1(\eta) \notag \\
    &\hspace{1.2in} + F_{\eta}(1,1)\upsilon(\eta),
    \label{eq_87b} \\
    & V_{\xi}(1,\eta) = F_{\xi}(1,\eta), \label{eq_87c} \\
    & V(\xi,-1) = F(\xi,-1), \label{eq_87d} \\
    & V(\xi,1) = F(\xi,1)
    = F(-1,1)\rho_0(\xi) + \lambda_DF_{\xi}(-1,1)\upsilon_0(\xi)
    + (1-\lambda_C)\gamma_C F(1,1)\rho_1(\xi) \notag \\
    &\hspace{1.2in}
    + F_{\xi}(1,1)\upsilon_1(\xi)
    + \lambda_C F_{\xi\xi}(1,1)\omega_1(\xi),
    \label{eq_87e} \\
    & V_{\eta}(\xi,1) = F_{\eta}(\xi,1), \label{eq_87f} \\
    & V_{\xi\eta}(1,1) = F_{\xi\eta}(1,1), \label{eq_87g}
  \end{align}
\end{subequations}
where $F_{\xi}(1,\eta)$, $F_{\eta}(\xi,1)$ and $F_{\xi\eta}(1,1)$ are defined
in~\eqref{eq_66b}, \eqref{eq_73b}, and~\eqref{eq_79}, respectively.

We define the transfinite interpolation
\begin{align}
  PF(\xi,\eta)
  =&\ F(-1,\eta)\rho_0(\xi) + F_{\xi}(1,\eta)\upsilon_1(\xi)
  + F(\xi,-1)\rho_0(\eta) + F_{\eta}(\xi,1)\upsilon_1(\eta) \notag \\
  & -\left[F(-1,-1)\rho_0(\eta) + F_{\eta}(-1,1)\upsilon_1(\eta) \right]\rho_0(\xi)
  -\left[F_{\xi}(1,-1)\rho_0(\eta) + F_{\xi\eta}(1,1)\upsilon_1(\eta) \right]\upsilon_1(\xi) \notag \\
  & + \left[
    \lambda_DF_{\xi}(-1,1)\upsilon_0(\xi) + \lambda_C F_{\xi\xi}(1,1)\omega_1(\xi)
    + (1-\lambda_C)\gamma_C F(1,1)\rho_1(\xi)
    \right]\rho_1(\eta) \notag \\
  &
  + \lambda_B F_{\eta}(1,-1)\upsilon_0(\eta)\rho_1(\xi).
\end{align}
One can verify that the function $V(\xi,\eta)=PF(\xi,\eta)$ satisfies
the conditions in~\eqref{eq_87} exactly.

The preliminary general form for $V(\xi,\eta)$  is then given by,
\begin{align}\label{eq_89}
  V(\xi,\eta) = g(\xi,\eta) - Pg(\xi,\eta) + PF(\xi,\eta),
\end{align}
where $g(\xi,\eta)$ is a free (arbitrary) function, and
\begin{align}\label{eq_90}
  Pg(\xi,\eta) =&\ g(-1,\eta)\rho_0(\xi) + g_{\xi}(1,\eta)\upsilon_1(\xi)
  + g(\xi,-1)\rho_0(\eta) + g_{\eta}(\xi,1)\upsilon_1(\eta) \notag \\
  & -\left[g(-1,-1)\rho_0(\eta) + g_{\eta}(-1,1)\upsilon_1(\eta) \right]\rho_0(\xi)
  -\left[g_{\xi}(1,-1)\rho_0(\eta) + g_{\xi\eta}(1,1)\upsilon_1(\eta) \right]\upsilon_1(\xi) \notag \\
  & + \left[
    \lambda_D g_{\xi}(-1,1)\upsilon_0(\xi) + \lambda_C g_{\xi\xi}(1,1)\omega_1(\xi)
    + (1-\lambda_C)\gamma_C g(1,1)\rho_1(\xi)
    \right]\rho_1(\eta) \notag \\
  &
  + \lambda_B g_{\eta}(1,-1)\upsilon_0(\eta)\rho_1(\xi).
\end{align}

The modified transfinite interpolation is,
\begin{align}\label{eq_91}
  PF^g(\xi,\eta)
  =&\ F(-1,\eta)\rho_0(\xi) + F_{\xi}^g(1,\eta)\upsilon_1(\xi)
  + F(\xi,-1)\rho_0(\eta) + F_{\eta}^g(\xi,1)\upsilon_1(\eta) \notag \\
  & -\left[F(-1,-1)\rho_0(\eta) + F_{\eta}(-1,1)\upsilon_1(\eta) \right]\rho_0(\xi)
  -\left[F_{\xi}(1,-1)\rho_0(\eta)
    + F_{\xi\eta}^g(1,1)\upsilon_1(\eta) \right]\upsilon_1(\xi) \notag \\
  & + \left[
    \lambda_DF_{\xi}^a(-1,1)\upsilon_0(\xi) + \lambda_C F_{\xi\xi}^g(1,1)\omega_1(\xi)
    + (1-\lambda_C)\gamma_C F^a(1,1)\rho_1(\xi)
    \right]\rho_1(\eta) \notag \\
  &
  + \lambda_B F_{\eta}^a(1,-1)\upsilon_0(\eta)\rho_1(\xi).
\end{align}
Here $F_{\xi}^a(-1,1)$ and $F_{\eta}^a(1,-1)$ are given
in~\eqref{eq_75a} and~\eqref{eq_68a}, and $F^a(1,1)$ is given in~\eqref{eq_83b}.
In addition,
\begin{subequations}\label{eq_92}
  \begin{align}
    F_{\xi\eta}^g(1,1) =&\ Q_{BC}^a(1,1) - S_{BC}(1)g_{\eta\eta}(1,1) \notag \\
    & + Q_{BC}^b(1)\left[
      \left(1-(1-\lambda_C)\gamma_C \right)g(1,1)
      + (1-\lambda_C)\gamma_C F^a(1,1)
      \right]; \label{eq_92a} \\
    F_{\xi\xi}^g(1,1) =&\ \frac{S_{BC}(1)}{S_{CD}(1)}g_{\eta\eta}(1,1)
    - \frac{1}{S_{CD}(1)}\left[Q_{BC}^b(1) - Q_{CD}^b(1) \right]
    \left[\left(1-(1-\lambda_C)\gamma_C \right)g(1,1)\right. \notag \\
     & \left. + (1-\lambda_C)\gamma_C F^a(1,1)
      \right]
    -\frac{1}{S_{CD}(1)}\left[Q_{BC}^a(1)-Q_{CD}^a(1) \right]; \label{eq_92b} \\
    F_{\xi}^g(1,\eta) =&\ T_{BC}(\eta) - S_{BC}(\eta)V_{\eta}^g(1,\eta)
    -\alpha_{BC}W_{BC}(\eta)V^g(1,\eta); \label{eq_92c} \\
    F_{\eta}^g(\xi,1) =&\ T_{CD}(\xi) - S_{CD}V_{\xi}^g(\xi,1)
    -\alpha_{CD}W_{CD}(\xi)V^g(\xi,1); \label{eq_92d} \\
    V^g(1,\eta) =&\ g(1,\eta) - \left[g(1,-1)-F(1,-1) \right]\rho_0(\eta)
    - \left[g_{\eta}(1,1)-F_{\eta}^g(1,1) \right]\upsilon_1(\eta) \notag \\
    & - \lambda_B\left[g_{\eta}(1,-1)-F_{\eta}^a(1,-1) \right]\upsilon_0(\eta)
    -(1-\lambda_C)\gamma_C\left[g(1,1)-F^a(1,1) \right]\rho_1(\eta); \label{eq_92e} \\
    V_{\eta}^g(1,\eta) =&\ g_{\eta}(1,\eta) - \left[g(1,-1)-F(1,-1) \right]\rho_0'(\eta)
    - \left[g_{\eta}(1,1)-F_{\eta}^g(1,1) \right]\upsilon_1'(\eta) \notag \\
    & - \lambda_B\left[g_{\eta}(1,-1)-F_{\eta}^a(1,-1) \right]\upsilon_0'(\eta)
    -(1-\lambda_C)\gamma_C\left[g(1,1)-F^a(1,1) \right]\rho_1'(\eta); \label{eq_92f} \\
    V^g(\xi,1)=&\ g(\xi,1) - \left[g(-1,1)-F(-1,1) \right]\rho_0(\xi)
    -\left[g_{\xi}(1,1)-F_{\xi}^g(1,1) \right]\upsilon_1(\xi) \notag \\
    & -\lambda_D\left[g_{\xi}(-1,1)-F_{\xi}^a(-1,1) \right]\upsilon_0(\xi)
    -\lambda_C\left[g_{\xi\xi}(1,1)-F_{\xi\xi}^g(1,1) \right]\omega_1(\xi) \notag \\
    & -(1-\lambda_C)\gamma_C\left[g(1,1)-F^a(1,1) \right]\rho_1(\xi); \label{eq_92g} \\
    V_{\xi}^g(\xi,1)=&\ g_{\xi}(\xi,1) - \left[g(-1,1)-F(-1,1) \right]\rho_0'(\xi)
    -\left[g_{\xi}(1,1)-F_{\xi}^g(1,1) \right]\upsilon_1'(\xi) \notag \\
    & -\lambda_D\left[g_{\xi}(-1,1)-F_{\xi}^a(-1,1) \right]\upsilon_0'(\xi)
    -\lambda_C\left[g_{\xi\xi}(1,1)-F_{\xi\xi}^g(1,1) \right]\omega_1'(\xi) \notag \\
    & -(1-\lambda_C)\gamma_C\left[g(1,1)-F^a(1,1) \right]\rho_1'(\xi); \label{eq_92h} \\
    F_{\eta}^g(1,1) =&\ F_{\eta}^a(1,1) - F_{\eta}^b(1,1)\left[
      \left(1-(1-\lambda_C)\gamma_C \right)g(1,1)
      + (1-\lambda_C)\gamma_C F^a(1,1)
      \right]; \label{eq_92i} \\
    F_{\xi}^g(1,1) =&\ F_{\xi}^a(1,1) - F_{\xi}^b(1,1)\left[
      \left(1-(1-\lambda_C)\gamma_C \right)g(1,1)
      + (1-\lambda_C)\gamma_C F^a(1,1)
      \right]. \label{eq_92j}
  \end{align}
\end{subequations}
This gives rise to the final form for $V(\xi,\eta)$,
\begin{align}\label{eq_93}
  V(\xi,\eta) = g(\xi,\eta) - Pg(\xi,\eta) + PF^g(\xi,\eta),
\end{align}
where $g(\xi,\eta)$ is a free (arbitrary) function,
$Pg(\xi,\eta)$ is given by~\eqref{eq_90},
and $PF^g(\xi,\eta)$ is given by~\eqref{eq_91}.

\begin{theorem}
  $V(\xi,\eta)$ given by~\eqref{eq_93} satisfies
  the Dirichlet conditions~\eqref{eq_14a} and~\eqref{eq_14d}
  and the Robin conditions~\eqref{eq_66a} and~\eqref{eq_73a},
  for any $g(\xi,\eta)$ therein that is sufficiently differentiable,
\end{theorem}
\begin{proof}
  By verification. The verification process relies on the following relations,
  but is otherwise straightforward albeit a little cumbersome.
  These relations are,
  \begin{subequations}
    \begin{align}
      & \lim_{\xi\rightarrow -1}F_{\eta}^g(\xi,1)= F_{\eta}^g(-1,1)
      = F_{\eta}(-1,1)=\lim_{\eta\rightarrow 1}F_{\eta}(-1,\eta), \\
      & \lim_{\eta\rightarrow -1}F_{\xi}^g(1,\eta)= F_{\xi}^g(1,-1)
      = F_{\xi}(1,-1) = \lim_{\xi\rightarrow 1}F_{\xi}(\xi,-1), \\
      & \lim_{\eta\rightarrow 1}F_{\xi}^g(1,\eta) = F_{\xi}^g(1,1), 
      \\
      & \lim_{\xi\rightarrow 1}F_{\eta}^g(\xi,1) = F_{\eta}^g(1,1), \\
      & \lim_{\xi\rightarrow 1}F_{\eta,\xi}^g(\xi,1)= F_{\eta,\xi}^g(1,1) = F_{\xi\eta}^g(1,1),\\
      & \lim_{\eta\rightarrow 1}F_{\xi,\eta}^g(1,\eta) = F_{\xi,\eta}^g(1,1)
      = F_{\xi\eta}^g(1,1),
    \end{align}
  \end{subequations}
  where $F_{\xi}^g(1,\eta)$ and $F_{\eta}^g(\xi,1)$ are defined
  in~\eqref{eq_92c} and~\eqref{eq_92d},
  $F_{\eta}^g(1,1)$ and $F_{\xi}^g(1,1)$ are defined in~\eqref{eq_92i}
  and~\eqref{eq_92j}, and $F_{\xi\eta}^g(1,1)$ is defined in~\eqref{eq_92a}.
  These relations can be verified to be true by considering
  different cases such as whether adjacent boundaries are orthogonal to
  each other or not at a vertex.
\end{proof}


\begin{remark}
  When the domain involves more than two Robin boundaries, at
  every vertex where two Robin boundaries meet,
  the compatibility constraints analogous to the aforementioned ones will apply.
  At every vertex where a Robin boundary and a Dirichlet boundary meet,
  the compatibility constraints analogous to
  those discussed in Section~\ref{sec_241} will apply. The field function
  satisfying these boundary conditions can be formulated in an
  analogous way based on the four-step procedure.
\end{remark}

\begin{remark}
  When the domain involves a combination of Dirichlet, Neumann,
  and Robin boundaries, at every vertex where a Robin boundary
  and a Neumann boundary meet, the compatibility constraints
  as discussed above for two Robin boundaries will apply. In
  this case, the Neumann boundary involved in can be
  treated as a Robin one with the Robin coefficient $\alpha$
  set to zero.
\end{remark}


\subsection{Enforcing Dirichlet/Neumann/Robin BCs Exactly for Solving PDEs by
Extreme Learning Machine}
\label{sec_25}

Let us now assume that the function formulated in
Sections~\ref{sec_22},~\ref{sec_23} and~\ref{sec_24} represents
the unknown solution field to some given PDE. Therefore, the
Dirichlet, Neumann and Robin boundary conditions involved in
the PDE problem will be automatically and exactly satisfied.
Since the free function $g(\xi,\eta)$ can be arbitrary,
one can choose a function space or use some nonlinear representation
such as artificial neural networks for $g(\xi,\eta)$ to satisfy
the PDE, 
thus giving rise to a specific numerical method.
It should be noted that, regardless of the representation or the numerical
method for computing $g(\xi,\eta)$,
the Dirichlet/Neumann/Robin boundary conditions involved in
the problem, by formulation, are exactly satisfied.

In this paper we employ the physics informed approach, and
a type of randomized neural networks known
as extreme learning machines (ELMs), for representing
the free function $g(\xi,\eta)$ to compute the PDE solution.
We refer to~\cite{DongL2021,DongY2022rm,NiD2023,WangD2024,DongY2022}
for more details on ELM for scientific machine learning.
%
In the following discussion we use a second-order
linear boundary value problem (BVP)
to illustrate the ELM technique together with
the current method for BC enforcement.
A discussion on nonlinear problems with the current method
is provided in a remark at the end of this section.

Specifically, we consider a general quadrilateral domain $\Omega=\overline{ABCD}$
as illustrated in Figure~\ref{fg_1}(a) and the following BVP on this domain,
\begin{subequations}\label{eq_95}
  \begin{align}
    \mathcal L u(\mbs x) &= f(\mbs x), \label{eq_95a} \\
    \left. \mathcal B u(\mbs x)\right|_{\mbs x\in\partial\Omega} &= f_b(\mbs x), \label{eq_95b}
  \end{align}
\end{subequations}
where $\mathcal L$ is a second-order linear differential operator, $u(\mbs x)$
is the unknown field to be solved for, and $f(\mbs x)$ and $f_b(\mbs x)$
represent prescribed source terms. $\mathcal B$ is the boundary operator,
and $\mathcal Bu(\mbs x)$ represents a set of Dirichlet, Neumann, or Robin
type conditions imposed on different domain boundaries.
Since $u(\mbs x)$ needs to satisfy the second-order PDE,
we assume in this section that each boundary curve of the domain
($\mbs x_{AB}(\xi)$, $\mbs x_{BC}(\eta)$, $\mbs x_{CD}(\xi)$, $\mbs x_{AD}(\eta)$)
should be sufficiently differentiable.

Since the Dirichlet/Neumann/Robin conditions of~\eqref{eq_95b}
are exactly enforced by formulation, we only need to focus on
the PDE~\eqref{eq_95a}.
Employing the map $\mbs x(\xi,\eta)$, we transform~\eqref{eq_95a} into
\begin{align}\label{eq_96}
  & \mathcal L V(\xi,\eta) = f(\mbs x(\xi,\eta)) = f_a(\xi,\eta),
  \quad (\xi,\eta)\in\Omega_{st},
\end{align}
where $V(\xi,\eta)$ is the transformed field function and
is related to $u(\mbs x)$ by~\eqref{eq_12}.
It should be noted that $\mathcal L$ is a differential operator defined on
the physical domain (with respect to $\mbs x=(x,y)$), while $V(\xi,\eta)$
is formulated in terms of the standard domain (with respect to $(\xi,\eta)$); see
Remark~\ref{rem_29} below for a discussion on the computation of associated terms.
The formulations of $V(\xi,\eta)$ 
from previous sections
all have the following form,
\begin{align}\label{eq_97}
  V(\xi,\eta) = g(\xi,\eta) - Pg(\xi,\eta) + PF^g(\xi,\eta).
\end{align}
We rewrite the transfinite interpolation therein into two components,
\begin{align}
PF^g(\xi,\eta)=PF^{gb}(\xi,\eta) + PF^a(\xi,\eta),
\end{align}
where $PF^{gb}(\xi,\eta)$ denotes all the terms in $PF^g(\xi,\eta)$
that involve the free function $g(\xi,\eta)$, and $PF^a(\xi,\eta)$
denotes the rest of the terms. Note that $PF^{gb}$ is linear with
respect to $g$, and that $PF^{gb}=0$ if the domain involves only
 Dirichlet boundaries.
Equation~\eqref{eq_96}  then becomes
\begin{align}\label{eq_98}
  & \mathcal L V^{gb}(\xi,\eta)  = f_a(\xi,\eta) - \mathcal L (PF^a)(\xi,\eta),
\end{align}
where
\begin{align} \label{eq_99}
  V^{gb}(\xi,\eta) = g(\xi,\eta) - Pg(\xi,\eta) + PF^{gb}(\xi,\eta).
\end{align}


To represent $g(\xi,\eta)$,
we employ a randomized feedforward neural network, whose structure is characterized
by an architectural vector $\mbs m=[m_0, m_1, \dots, m_L]$.
Here $(L+1)$ (with $L\geqslant 2$)
is the depth of the network, and $m_i$ denotes the number of nodes in
the $i$-th layer. Layer $0$, with $m_0=2$, represents the input $(\xi,\eta)$,
and the last layer, with $m_L=1$, represents the output $g(\xi,\eta)$.
Those layers in between are the hidden layers.
Following the ELM convention~\cite{DongL2021,DongL2021bip}, we
assign the hidden-layer coefficients (weights/biases) by random values
generated on the interval $[-R_m,R_m]$, where $R_m$ is a user-prescribed
constant, from a uniform distribution. Once the hidden-layer coefficients
are randomly assigned, they are fixed throughout the computation.
In addition, we require that the output layer contains no activation,
or equivalently with the activation function $\sigma(x)=x$,
and has zero bias. In ELM the hidden-layer coefficients
are randomly assigned and non-trainable, and only the output-layer coefficients
are trained~\cite{HuangZS2006,DongL2021}.


Then the NN logic of the output layer yields the following relation,
\begin{align}\label{eq_101}
  g(\xi,\eta) = \sum_{j=1}^M \beta_j\varphi_j(\xi,\eta) = \bm\Phi(\xi,\eta)\bm\beta.
\end{align}
Here $M=m_{L-1}$ is the width of the last hidden layer,
$\bm\Phi(\xi,\eta)=(\varphi_1,\dots,\varphi_M)$ denotes the set of output fields
of the last hidden layer, and $\bm\beta=(\beta_1,\dots,\beta_M)^T$
denotes the output-layer coefficients, which constitute
the set of trainable parameters in ELM.

We train the ELM network to solve equation~\eqref{eq_98}
based on a physics informed approach. By choosing a set of $Q$ collocation
points from the interior of the standard domain, 
$(\xi_i,\eta_i)\in\Omega_{st}$ for $1\leqslant i\leqslant Q$, and
enforcing equation~\eqref{eq_98} on these collocation points, we have
\begin{align}\label{eq_102}
  \left[\left.\mathcal L\bm\Phi\right|_{(\xi_i,\eta_i)}
    - \left.\mathcal L(P\bm\Phi)\right|_{(\xi_i,\eta_i)}
    +\left.\mathcal L(PF^{\bm\Phi b})\right|_{(\xi_i,\eta_i)}\right]\bm\beta
  = f_a(\xi_i,\eta_i) - \left.\mathcal L(PF^a) \right|_{(\xi_i,\eta_i)},
  \quad 1\leqslant i\leqslant Q,
\end{align}
where we have used~\eqref{eq_101}.
In this equation $P\bm\Phi = (P\varphi_1,\dots,P\varphi_M)$, and $P\varphi_i(\xi,\eta)$ is
defined in the same manner as $Pg(\xi,\eta)$.
$PF^{\bm\Phi b}=(PF^{\varphi_1 b},\dots, PF^{\varphi_M b})$, and
$PF^{\varphi_i b}(\xi,\eta)$ is defined in the same manner as $PF^{gb}(\xi,\eta)$.
Equation~\eqref{eq_102} constitutes a rectangular system of linear algebraic
equations about $\bm\beta$, with $Q$ equations and $M$ unknowns.
We seek a least squares solution and solve this system by
the linear least squares method~\cite{Bjorck1996}.
The output-layer coefficients of the ELM network are
then set by this least squares solution for $\bm\beta$,
completing the NN training. The final solution field
to the boundary value problem~\eqref{eq_95} is computed based on~\eqref{eq_12}
and~\eqref{eq_97}.


\begin{remark}\label{rem_29}
  When implementing ELM for~\eqref{eq_102}, one encounters
  derivatives like $\frac{\partial \psi}{\partial x}$,
  $\frac{\partial \psi}{\partial y}$, $\frac{\partial^2 \psi}{\partial x^2}$, and
  $\frac{\partial^2 \psi}{\partial y^2}$,
  where $\psi$ is a function defined on the standard domain $\Omega_{st}$,
  i.e.~$\psi = \psi(\xi,\eta)$. These terms can be computed using
  the Jacobian matrix $\mbs J(\xi,\eta)$ defined in~\eqref{eq_19} as follows,
  \begin{subequations}
    \begin{align}
      & \begin{bmatrix}\psi_x & \psi_y \end{bmatrix}
      = \begin{bmatrix}\psi_{\xi} & \psi_{\eta} \end{bmatrix} \mbs J^{-1}(\xi,\eta),
      \quad \mbs C=\begin{bmatrix}\psi_{x} & \psi_{y} \end{bmatrix}
      \frac{\partial\mbs J}{\partial\xi}, \quad
      \mbs D=\begin{bmatrix}\psi_{x} & \psi_{y} \end{bmatrix}
      \frac{\partial\mbs J}{\partial\eta},
      \\
      & \begin{bmatrix} \psi_{xx} & \psi_{xy} \\ \psi_{xy} & \psi_{yy} \end{bmatrix}
      = \mbs J^{-T}\left(\begin{bmatrix} \psi_{\xi\xi} & \psi_{\xi\eta} \\
        \psi_{\xi\eta} & \psi_{\eta\eta}\end{bmatrix}
      - \begin{bmatrix} \mbs C^T & \mbs D^T \end{bmatrix}
      \right)\mbs J^{-1},
    \end{align}
  \end{subequations}
  where the terms $\psi_{\xi}$, $\psi_{\eta}$, $\psi_{\xi\xi}$, $\psi_{\xi\eta}$
  and $\psi_{\eta\eta}$ can be computed either directly or by automatic
  differentiations of the neural network.
  
\end{remark}


\begin{remark}\label{rem_210}
  After the NN training is complete,
  when using the form~\eqref{eq_97} to evaluate the solution field on
  the boundary test points, we have observed from numerical simulations
  that the cancellation error, due to subtraction of nearly equal real numbers
  in the terms like $(g-Pg)$,
  can be notable
  at isolated boundary points for some problems. For example,
  when evaluating the boundary-condition
  errors using the numerically attained solution, this can lead
  to errors on the order around $10^{-10}\sim 10^{-9}$ at isolated boundary
  points, instead of the error levels such as $10^{-16}$ or lower for
  the other boundary points. We find that a combination of
  the following two measures in implementation can reduce
  the cancellation error significantly:
  \begin{itemize}
  \item Restructure the computation for terms like $(g-Pg)$ and $PF$.
    For example, implementing the terms in~\eqref{eq_15}--\eqref{eq_16b}
    using the following
    equivalent forms essentially eliminates the cancellation error,
    \begin{align*}
      g(\xi,\eta)-Pg(\xi,\eta) =&\
      \left[g(\xi,\eta) - g(\xi,-1)\phi_0(\eta) - g(\xi,1)\phi_1(\eta) \right] \\
      & -\left[g(-1,\eta) - g(-1,-1)\phi_0(\eta) - g(-1,1)\phi_1(\eta) \right]\phi_0(\xi)
       \\
      & -\left[g(1,\eta) - g(1,-1)\phi_0(\eta) - g(1,1)\phi_1(\eta) \right]\phi_1(\xi), \\
      PF(\xi,\eta)=&\
      \left[F(\xi,-1)\phi_0(\eta) + F(\xi,1)\phi_1(\eta) \right] \\
      &+ \left[F(-1,\eta) - F(-1,-1)\phi_0(\eta) - F(-1,1)\phi_1(\eta) \right]\phi_0(\xi) \\
      &+ \left[F(1,\eta) - F(1,-1)\phi_0(\eta) - F(1,1)\phi_1(\eta) \right]\phi_1(\xi).
    \end{align*}


  \item Introduce a small number of collocation points on the Dirichlet boundaries
    and enforce $g(\xi,\eta)=0$ at those points when training the neural network.
    Because the $V(\xi,\eta)$ form, with arbitrary $g(\xi,\eta)$ therein,
    satisfies the BC~\eqref{eq_95b} mathematically, we have
    only employed the PDE~\eqref{eq_95a}
    for NN training to determine $g(\xi,\eta)$. The $g(\xi,\eta)$
    determined in such a way is necessarily not unique.
    In practice, 
    the function values for $g(\xi,\eta)$ determined in this way can have large magnitudes,
    exacerbating the aforementioned cancellation error issue.
    We observe that, by additionally introducing a small number
    (e.g.~$3$ or $5$) of collocation points on each of the Dirichlet boundary
    and enforcing $g(\xi,\eta)=0$ on these points during NN training,
    the resultant function values for $g(\xi,\eta)$ will generally involve
    much smaller magnitudes.
    This can notably improve the cancellation error.
    This measure leads to the following system of equations,
    \begin{align}\label{eq_103}
      \bm\Phi(\xi_j,\eta_j)\bm\beta = 0, \quad (\xi_j,\eta_j)\in\partial\Omega_{d}, \quad
      1\leqslant j\leqslant Q_{db},
    \end{align}
    where $\partial\Omega_d$ denotes the Dirichlet boundary of the domain,
    and $Q_{db}$ denotes the number of collocation points on the Dirichlet boundaries.
    As such, the final algebraic system for computing $\bm\beta$ 
    consists of~\eqref{eq_102} and~\eqref{eq_103}. The least squares solution
    to this augmented system provides the trained output-layer coefficients
    of the ELM network.
    
  \end{itemize}
  We have incorporated these measures into our implementation of the current method
  in this work.
  
\end{remark}


\begin{remark}
  Boundary value problems consisting of~\eqref{eq_95b}
  (for Dirichlet, Neumann, or Robin conditions) and  a nonlinear PDE,
  \begin{align}\label{eq_104}
    \mathcal Lu(\mbs x) + \mathcal N(u) = f(\mbs x),
  \end{align}
  where $\mathcal N$ denotes a nonlinear operator, can be solved 
  using the current method and ELM in an analogous fashion.
  By employing the mapping function and the formulation~\eqref{eq_97}
  and enforcing~\eqref{eq_104} on the chosen collocation points, we get
  \begin{align}
    &\mathcal L V^{gb}(\xi_i,\eta_i) +
    \mathcal N\left(V^{gb}(\xi_i,\eta_i)+PF^a(\xi_i,\eta_i)\right)
    = f_a(\xi_i,\eta_i) - \mathcal L(PF^a)(\xi_i,\eta_i),
    \quad (\xi,\eta_i)\in\Omega_{st}, \\
    & \quad
    1\leqslant i\leqslant Q, \notag
  \end{align}
  where $V^{gb}$ is given by~\eqref{eq_99} and~\eqref{eq_101}.
  This is a nonlinear algebraic system about the trainable parameters $\bm\beta$,
  with $Q$ equations and $M$ unknowns.
  We seek a least squares solution and solve this system by
  the nonlinear least squares (Gauss-Newton) method~\cite{Bjorck1996}, specifically by
  the NLLSQ-perturb (Nonlinear least squares with perturbations) algorithm
  from~\cite{DongL2021}. The output-layer coefficients are
  then updated by the least squares solution for $\bm\beta$ to complete
  the NN training.
  
\end{remark}


\section{Numerical Tests}
\label{sec_tests}

\begin{figure}
  \centerline{
    \includegraphics[width=1.1in]{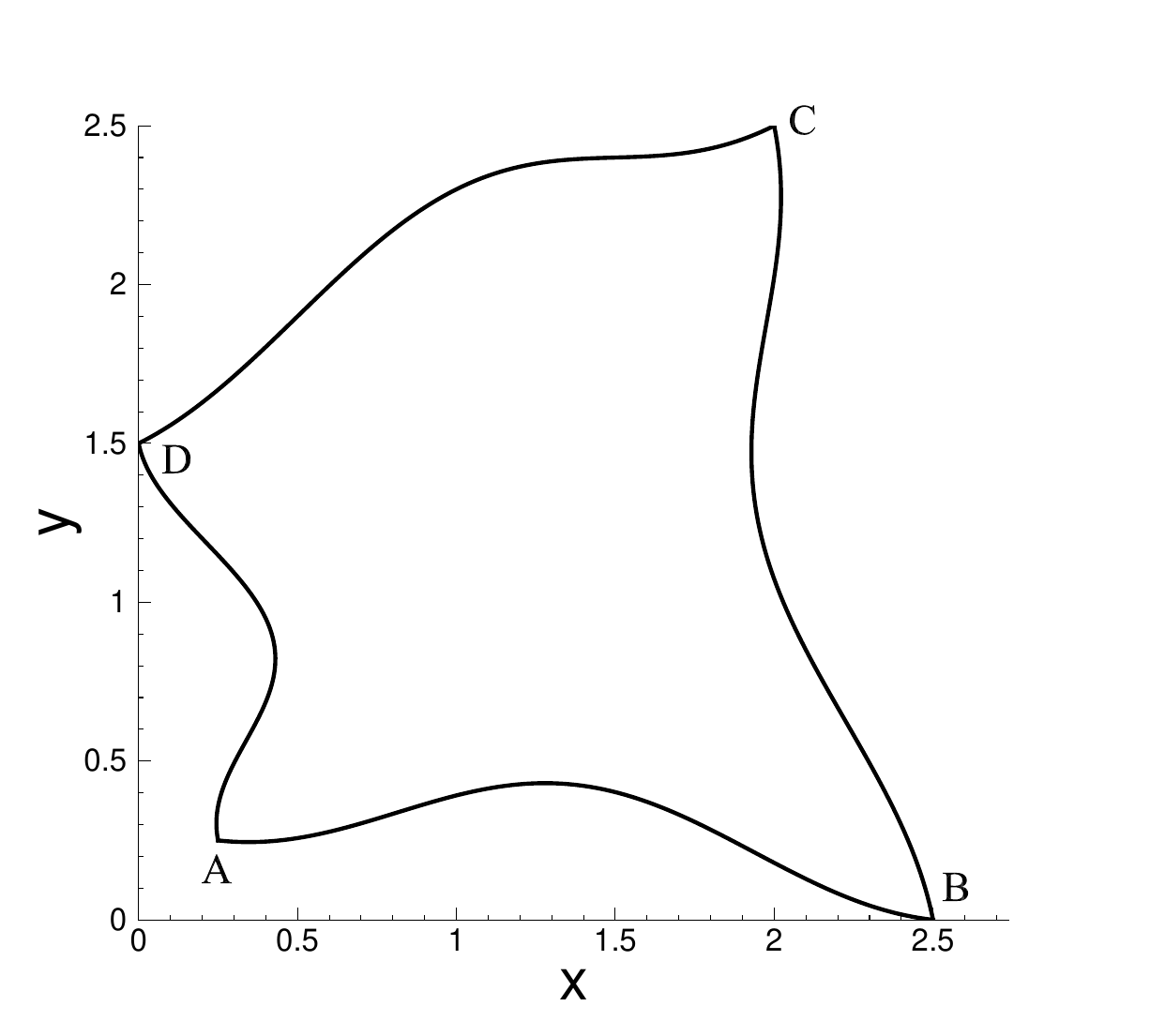}(a)
    \includegraphics[width=1.1in]{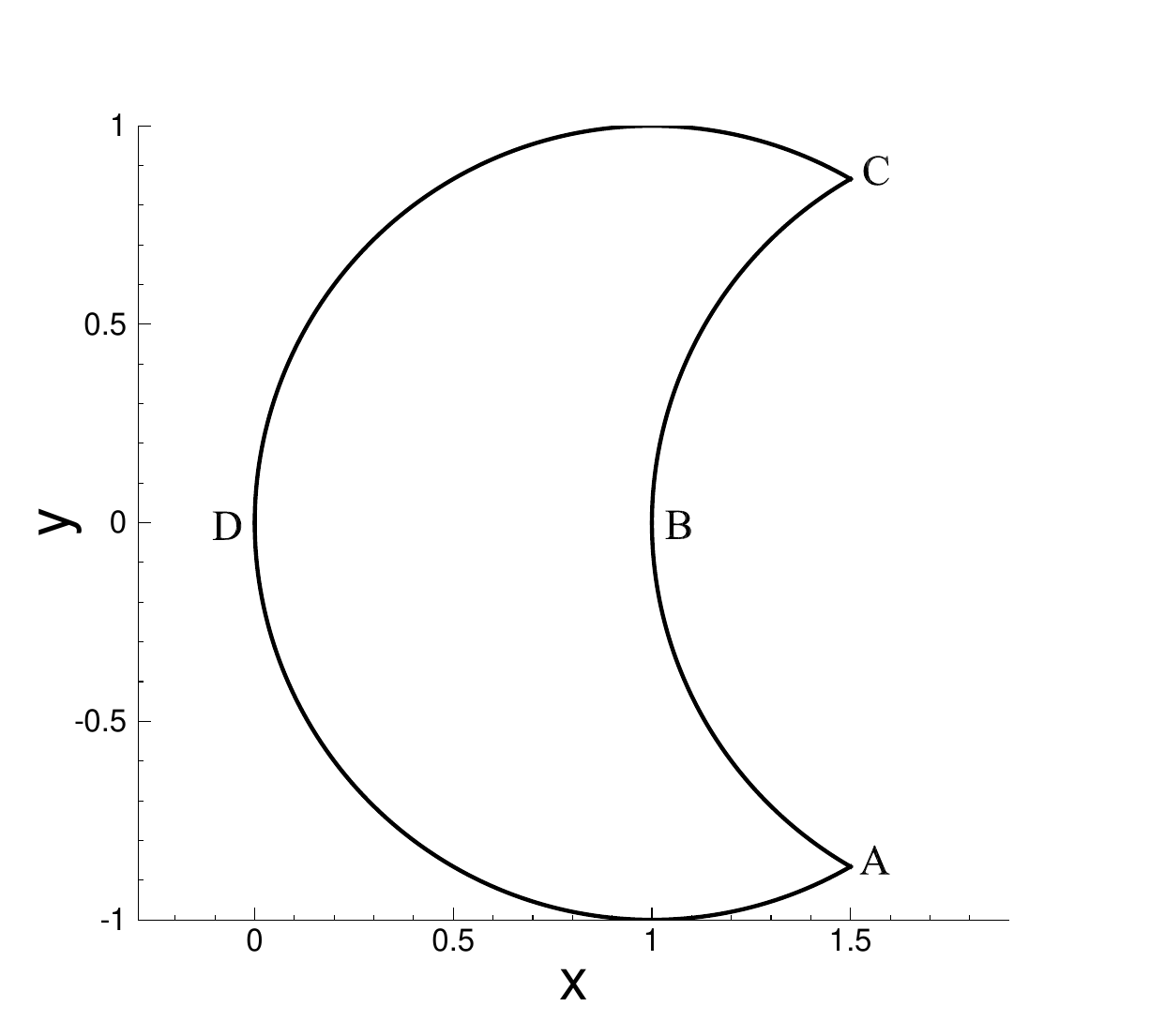}(b)
    \includegraphics[width=1.1in]{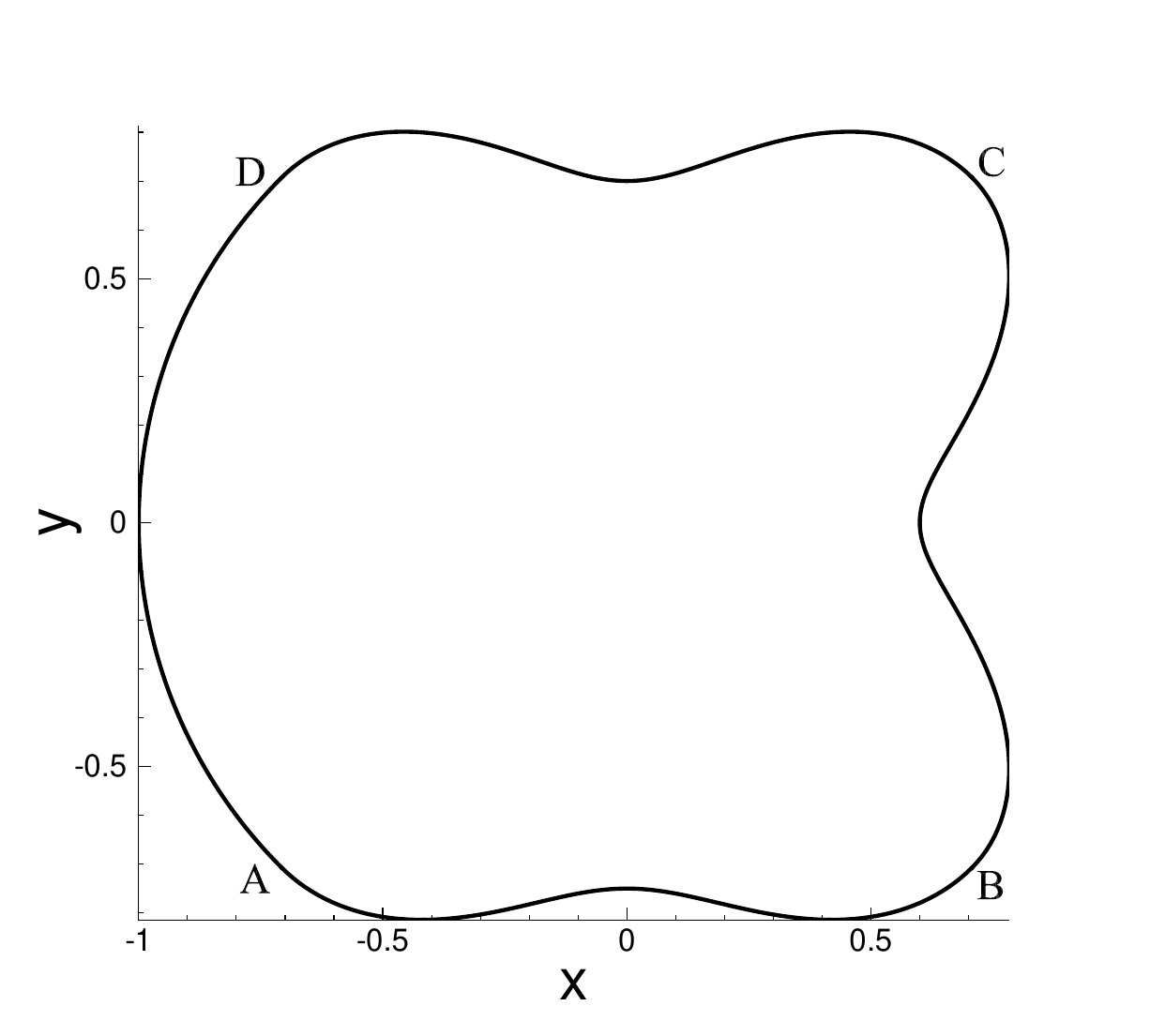}(c)
    \includegraphics[width=1.1in]{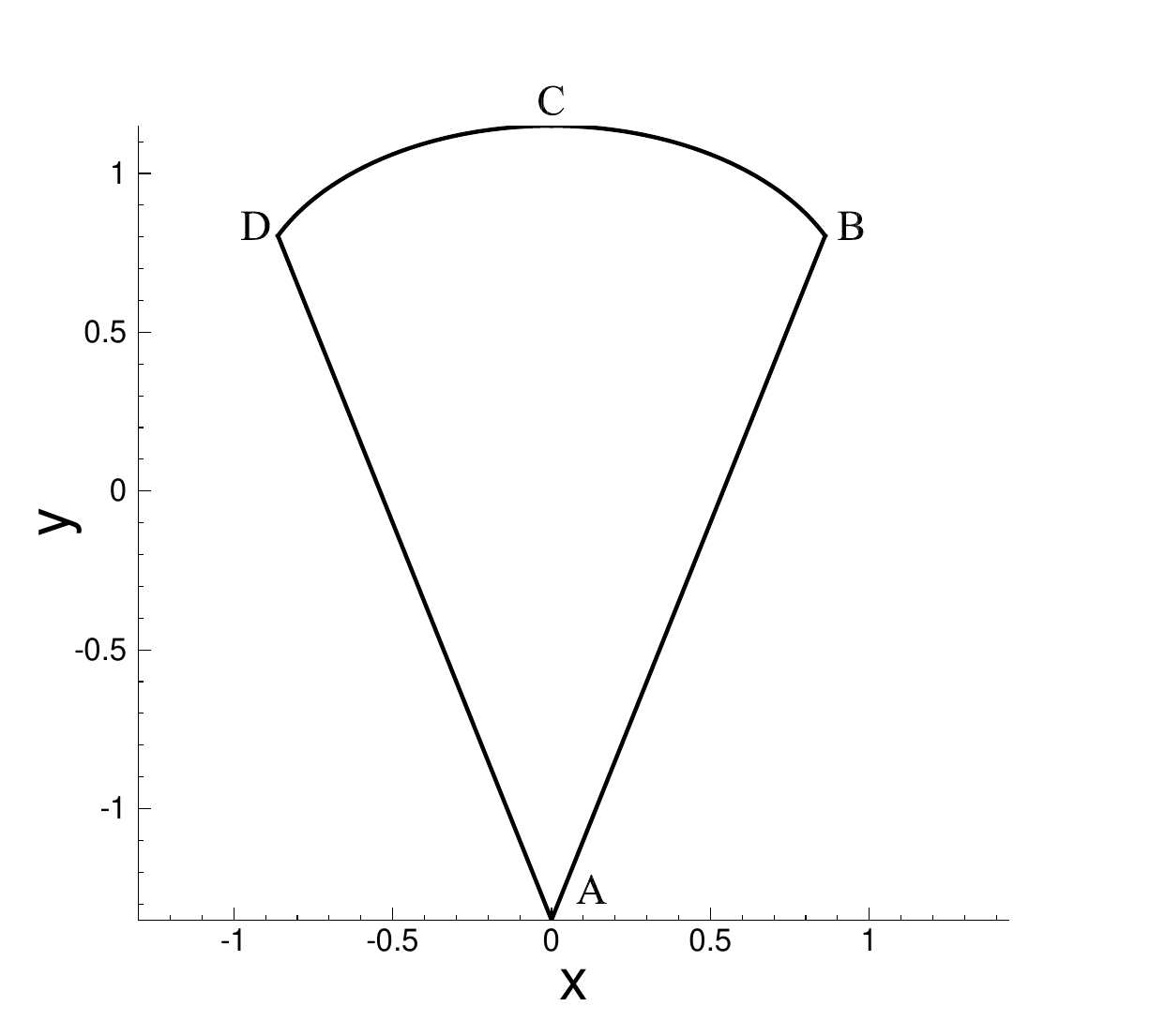}(d)
    \includegraphics[width=1.1in]{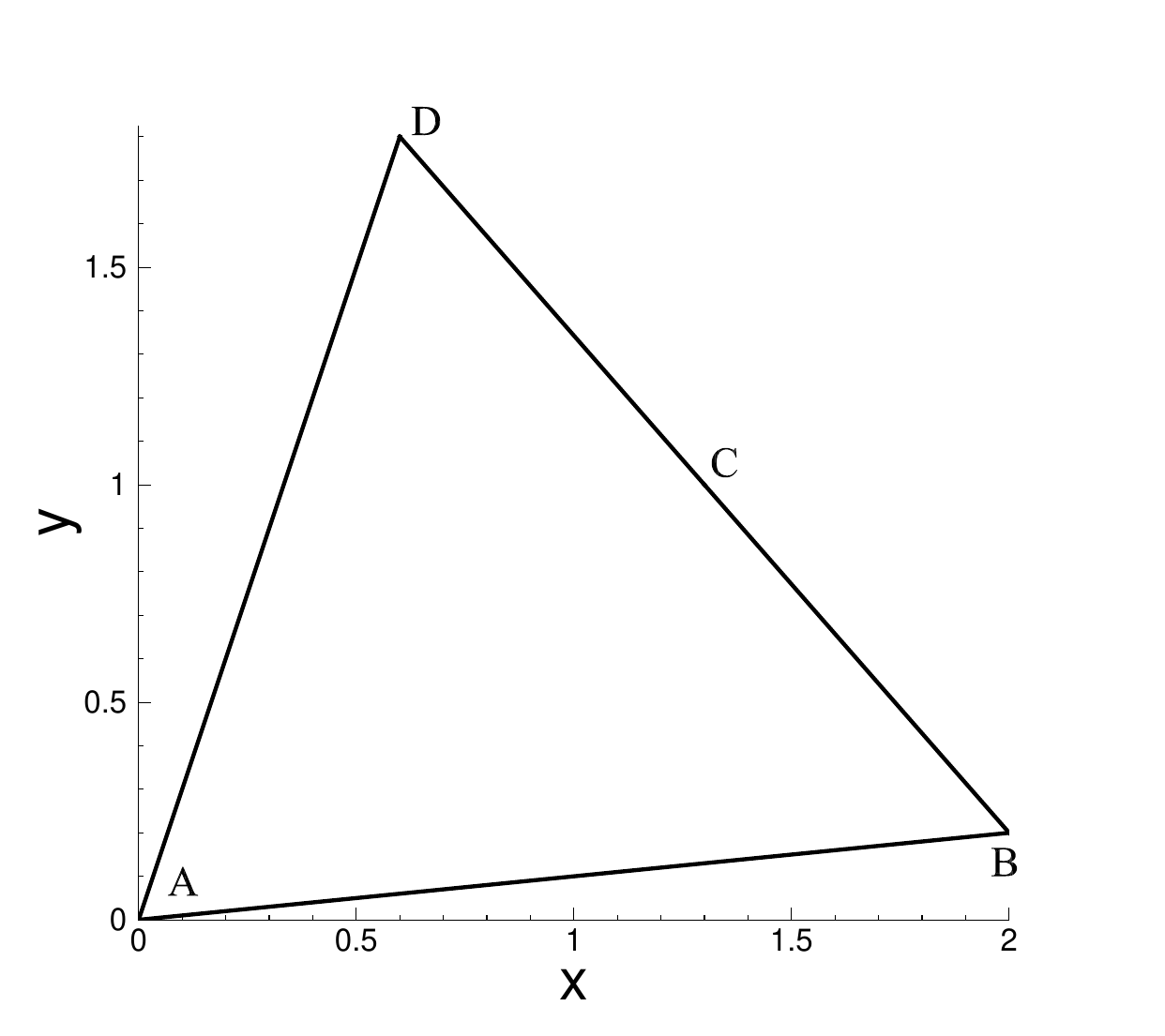}(e)
  }
  \centerline{
    \includegraphics[width=1.1in]{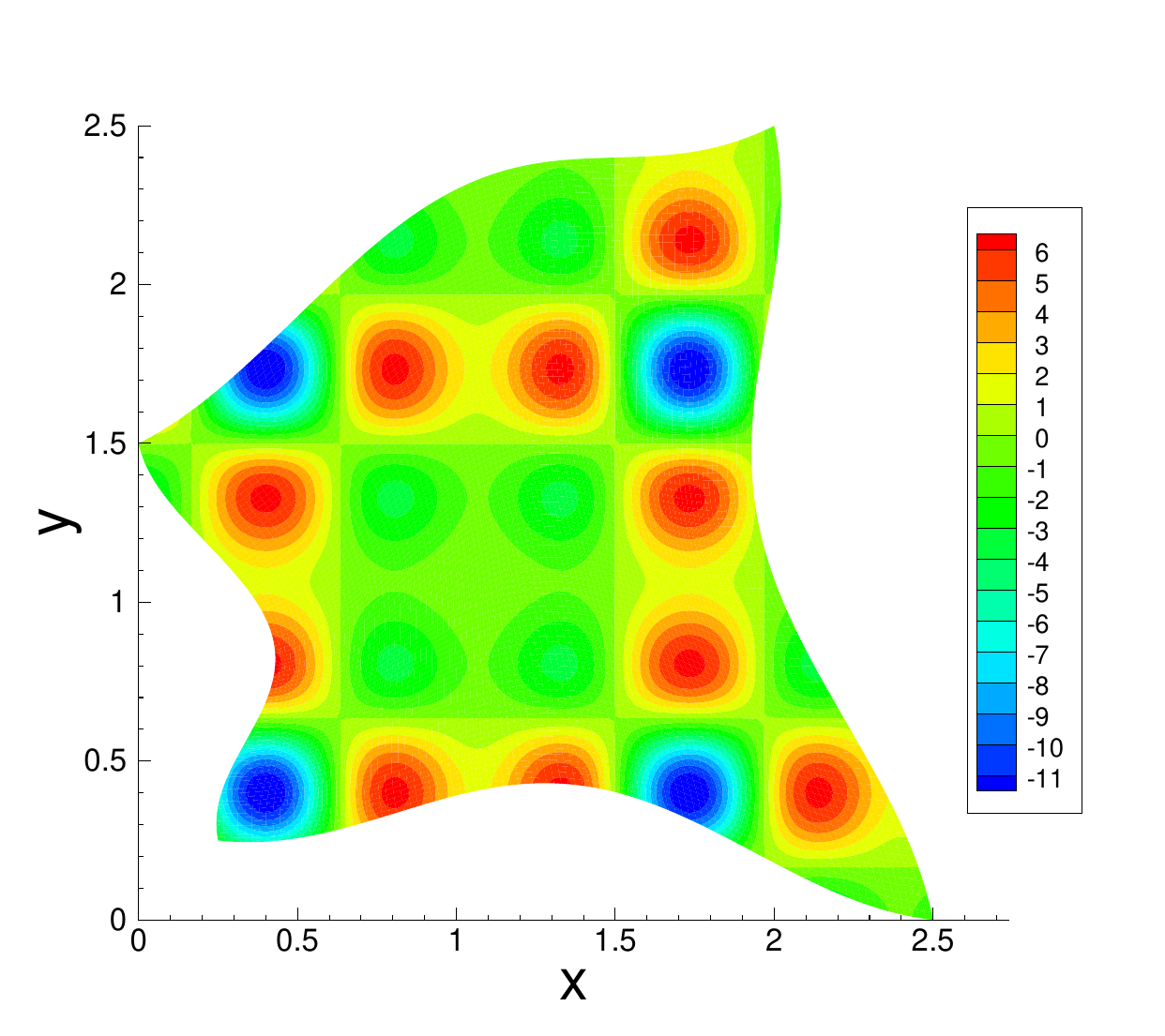}(f)
    \includegraphics[width=1.1in]{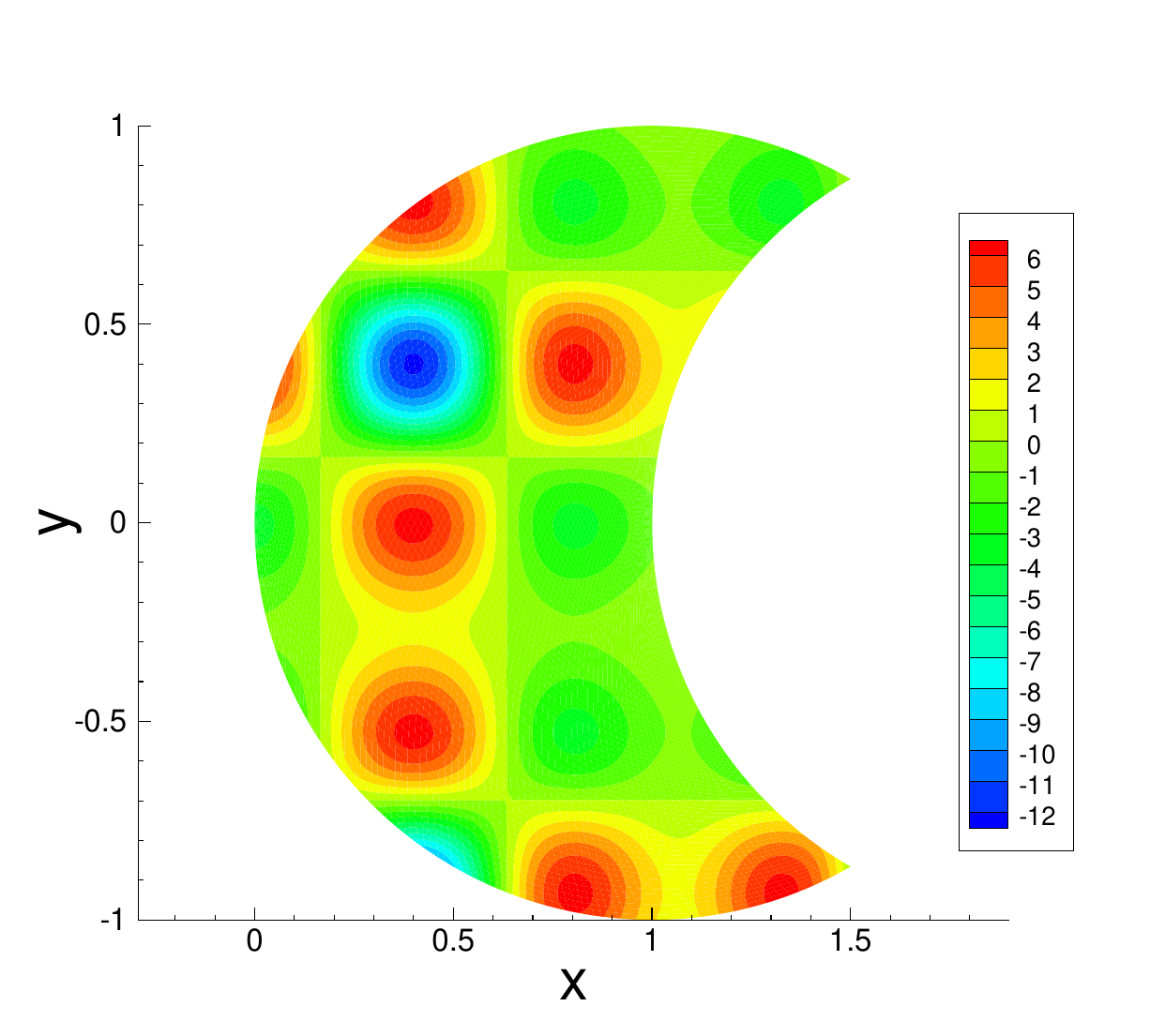}(g)
    \includegraphics[width=1.1in]{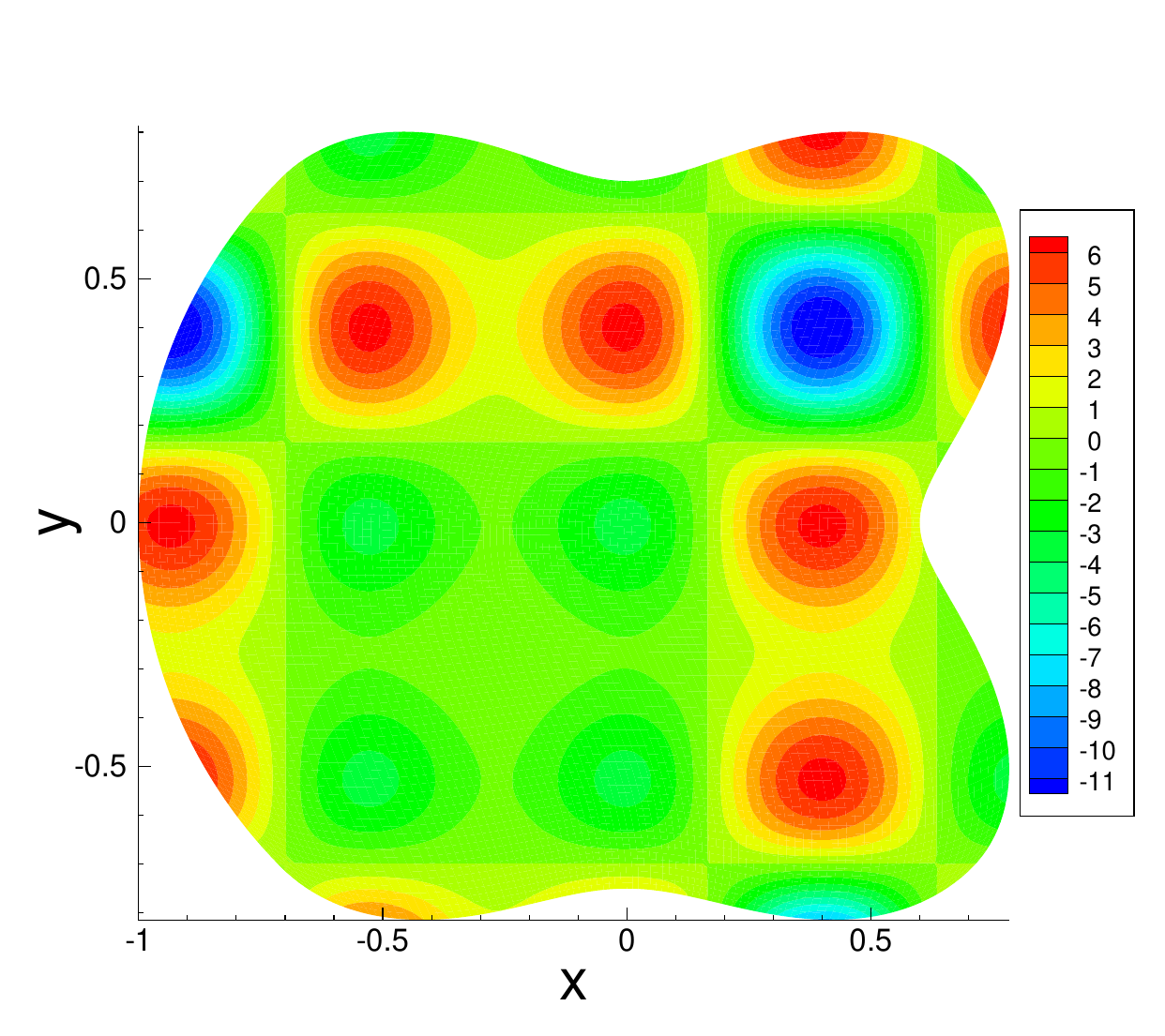}(h)
    \includegraphics[width=1.1in]{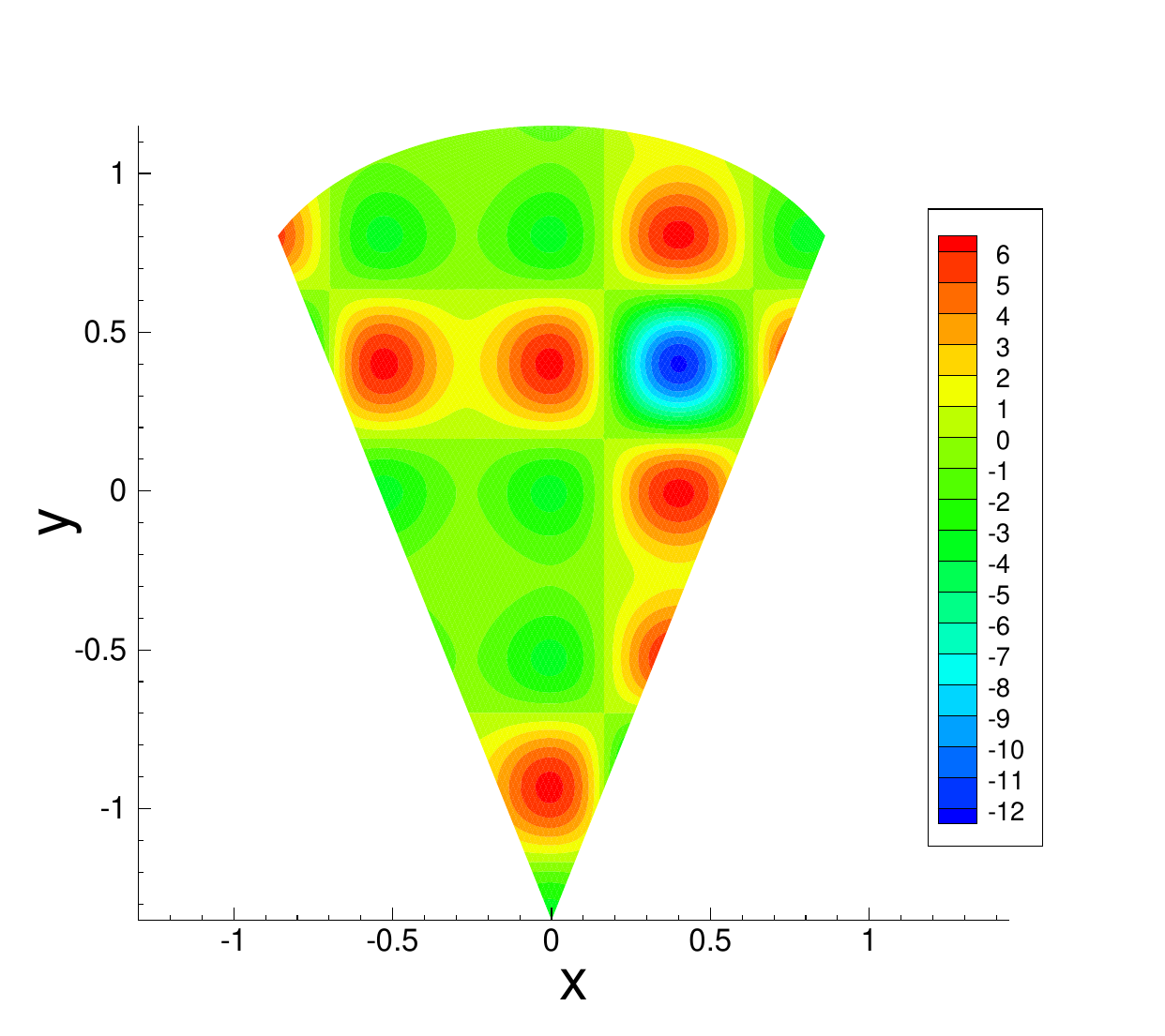}(i)
    \includegraphics[width=1.1in]{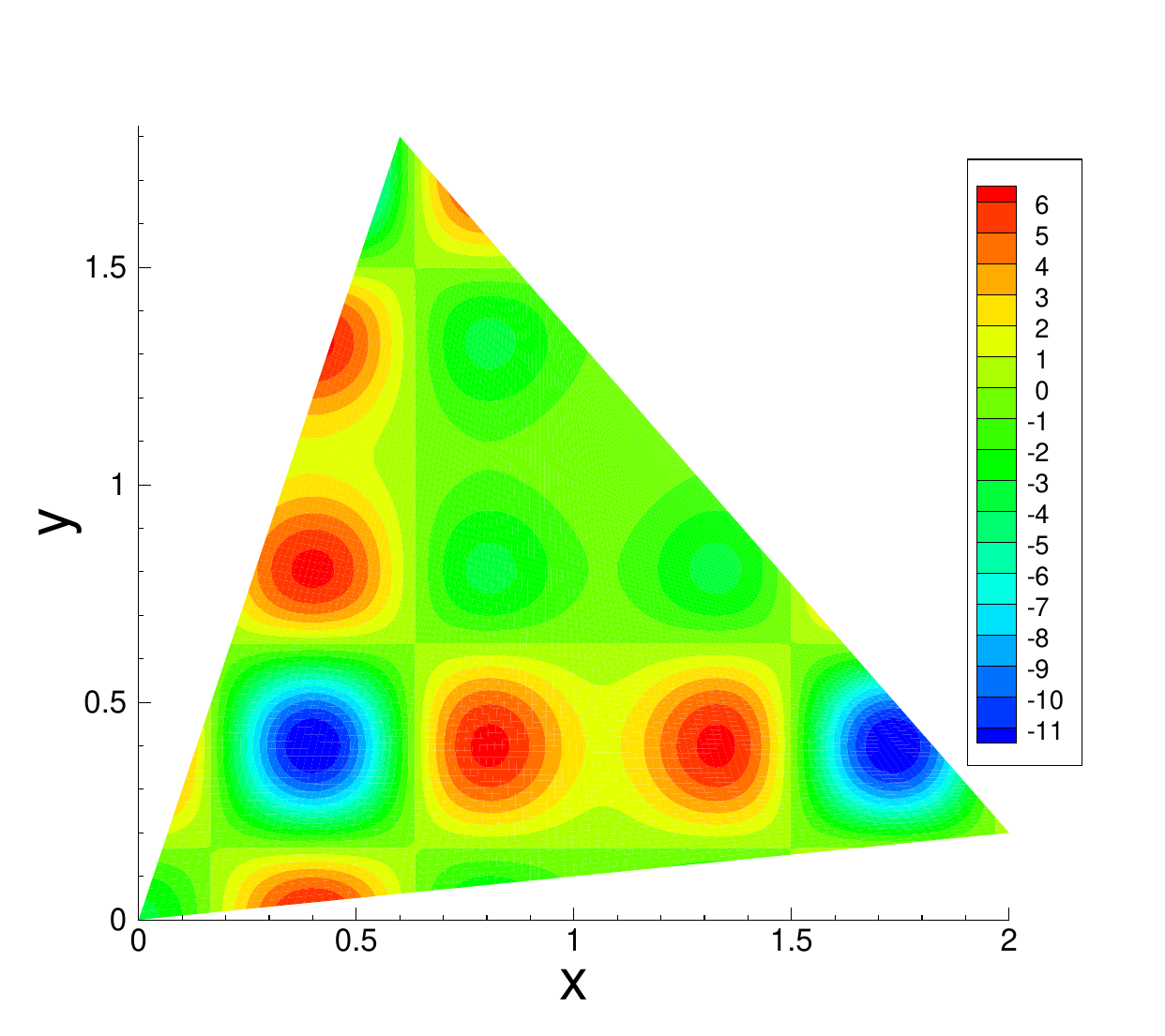}(j)
  }
  \caption{Helmholtz equation: Domain geometries (top row) and the exact solutions (bottom
    row), (a,f) domain \#1, (b,g) domain \#2, (c,h) domain \#3, (d,i) domain \#4,
    and (e,j) domain \#5.
  }
  \label{fg_2}
\end{figure}

We next present several numerical examples to
test the effectiveness of the method from the previous section
for enforcing Dirichlet/Neumann/Robin boundary conditions
(DBCs/NBCs/RBCs),
using linear and nonlinear PDEs over a number of domains with complex boundary
geometries. These include the 2D Helmholtz equation and the nonlinear Helmholtz equation,
which are time-independent, and the heat conduction equation over space-time domains
with deforming or moving boundaries.
The domains involve Dirichlet boundaries, or a combination of Dirichlet boundaries
with Neumann or Robin boundaries.

We define the maximum and root-mean-squares (rms) solution
errors ($e_{max}$ and $e_{rms}$) as follows,
\begin{align}
  e_{max} = \left\{\ \left|u(\mbs x_i)-u_{ex}(\mbs x_i) \right|  \ \right\}_{i=1}^{N_{v}}, \quad
  e_{rms} = \sqrt{\frac{1}{N_v}\sum_{i=1}^{N_v} \left|u(\mbs x_i)-u_{ex}(\mbs x_i) \right|^2 },
\end{align}
where $u(\mbs x)$ and $u_{ex}(\mbs x)$ denote the NN numerical solution 
and the exact solution, respectively, $\mbs x_i$ denotes the test points, and
$N_v$ is the number of test points. By choosing the test points over the entire domain $\Omega$
or on a specific boundary, we can define the maximum and rms solution errors over the domain
($e_{max}^{\Omega}$, $e_{rms}^{\Omega}$) or on the boundaries (e.g.~$e_{max}^{\overline{AB}}$,
$e_{rms}^{\overline{AB}}$, etc).
In addition, we define the maximum/rms boundary-condition errors (see equation~\eqref{eq_95b}),
\begin{align}
  \varepsilon_{max} = \left\{\ \left|\mathcal B u(\mbs x_i)-f_b(\mbs x_i) \right|  \ \right\}_{i=1}^{N_{b}},
  \quad
  \varepsilon_{rms} = \sqrt{\frac{1}{N_b}\sum_{i=1}^{N_b} \left|\mathcal Bu(\mbs x_i)-f_b(\mbs x_i) \right|^2 },
\end{align}
where $\mbs x_i$ denotes the boundary test points and $N_b$ is the number of such points.
By choosing the boundary test points from a specific boundary, we can define
the BC errors on specific boundaries (e.g.~$\varepsilon_{max}^{\overline{AB}}$,
$\varepsilon_{rms}^{\overline{AB}}$ etc), which can be the errors for Dirichlet, Neumann
or Robin conditions imposed there.
Unless otherwise specified, we employ $N_v=101\times 101$ test points (uniform grid points in the standard
domain $\Omega_{st}$) for computing $e_{max}^{\Omega}$ and $e_{rms}^{\Omega}$,
and $N_b=101$ or $N_v=101$ test points (uniform grids on each edge of $\Omega_{st}$) for computing
the boundary-condition errors or the boundary solution errors
($\varepsilon_{max}^{\overline{AB}}$, $\varepsilon_{rms}^{\overline{AB}}$, $e_{max}^{\overline{AB}}$,
$e_{rms}^{\overline{AB}}$, etc).

In all the numerical simulations of this section, we employ an ELM network
architecture $\mbs m=[2, M, 1]$ for representing the free function $g(\xi,\eta)$,
where $M$ is the number of hidden-layer nodes,
with the Gaussian activation function $\sigma(x)=e^{-x^2}$.
The hidden-layer coefficients are assigned to uniform random values generated on $[-R_m,R_m]$,
with the constant $R_m$ determined by the differential evolution algorithm from~\cite{DongY2022rm}.
We employ $N_c=(Q\times Q + Q_{db}$) collocation points for training the ELM.
Here $Q$ is the number of collocation points (uniform grid points)
along each direction in the interior of the standard domain, and
$Q_{db}$ is the number of points on the Dirichlet boundaries for enforcing
the condition $g(\xi,\eta)=0$ as discussed in Remark~\ref{rem_210}.
Unless otherwise specified, we employ $Q_{db}=16$  if the domain has all Dirichlet boundaries
($3$ uniform grid points on each boundary plus $4$ vertices),
$Q_{db}=19$ if the domain has a Neumann or Robin boundary with the rest being Dirichlet
boundaries ($5$ uniform points on each Dirichlet boundary plus $4$ vertices),
and $Q_{db}=23$ if the domain has two Neumann boundaries with the rest being Dirichlet
boundaries ($10$ uniform grid points on each Dirichlet boundary plus $3$ vertices).
The values for $R_m$, $Q$ and $M$ will be provided in the following discussions.
We employ~\eqref{eq_9} for the domain mapping in the numerical simulations,
unless otherwise noted.
Our implementation of the method and the neural network
is in Python, based on the Tensorflow and Keras libraries.

\subsection{Helmholtz Equation}
\label{sec_31}

In the first test we consider the five domains depicted in Figure~\ref{fg_2} (top row) and
investigate the boundary value problem with
the Helmholtz equation on these domains,
\begin{subequations}
  \begin{align}
    & \frac{\partial^2u}{\partial x^2} + \frac{\partial^2u}{\partial y^2}
    - 100 u = f(x,y), \\
    & \left.\mathcal B u(x,y)\right|_{(x,y)\in\partial\Omega}=f_b(x,y),
  \end{align}
\end{subequations}
where $u(x,y)$ is the unknown field to be computed,
$f$ and $f_b$ are the source terms, and the boundary operator $\mathcal B$
denotes Dirichlet, Neumann, or Robin conditions on different boundaries.
The specific geometric parameters for these domains (boundary curves, vertices)
are provided in the appendix (Section~\ref{sec_geom}).
We choose the source terms appropriately such that the problem has
the following exact solution for different boundary conditions,
\begin{align}
  & u(x,y) = -\left[2\cos\left(\frac32\pi x + \frac25\pi \right)
    +\frac32\cos\left(3\pi x - \frac{\pi}{5}\right)
    \right]
  \left[2\cos\left(\frac32\pi y + \frac25\pi \right)
    +\frac32\cos\left(3\pi y - \frac{\pi}{5}\right)
    \right].
\end{align}
Distributions of the exact solution on different domains are shown in
Figure~\ref{fg_2} (bottom row).

\begin{figure}
  \centerline{
    \includegraphics[width=1.1in]{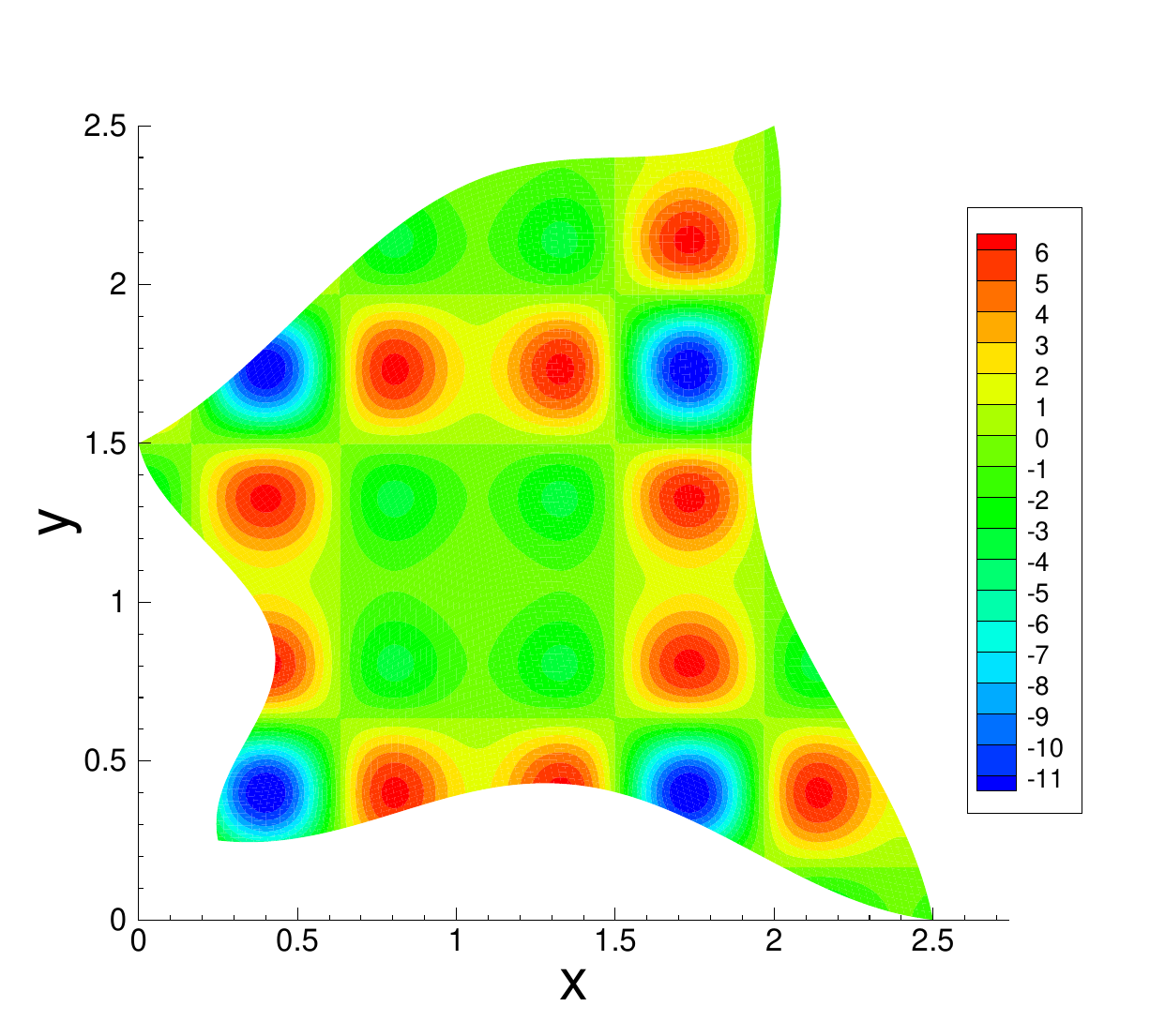}(a)
    \includegraphics[width=1.1in]{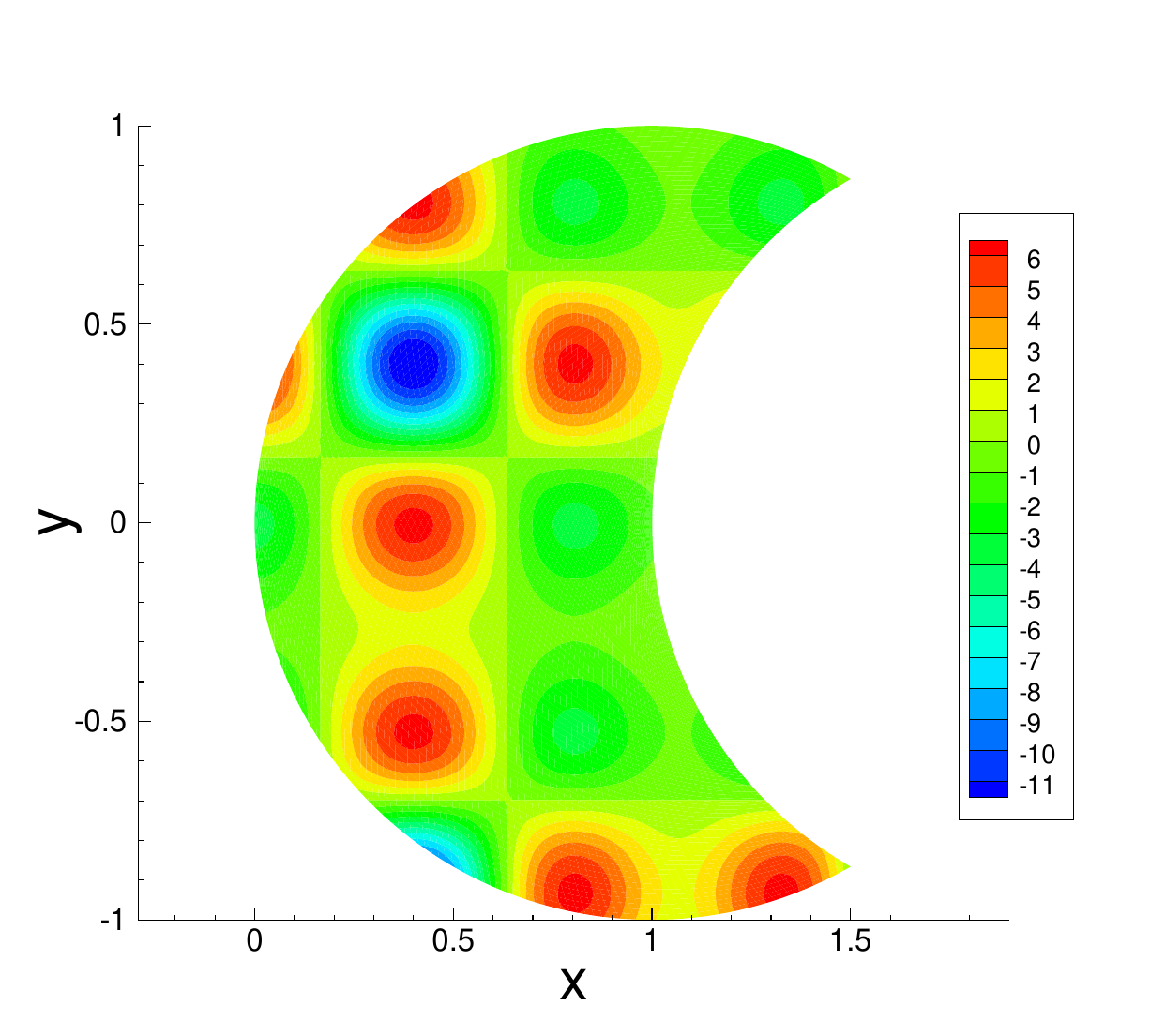}(b)
    \includegraphics[width=1.1in]{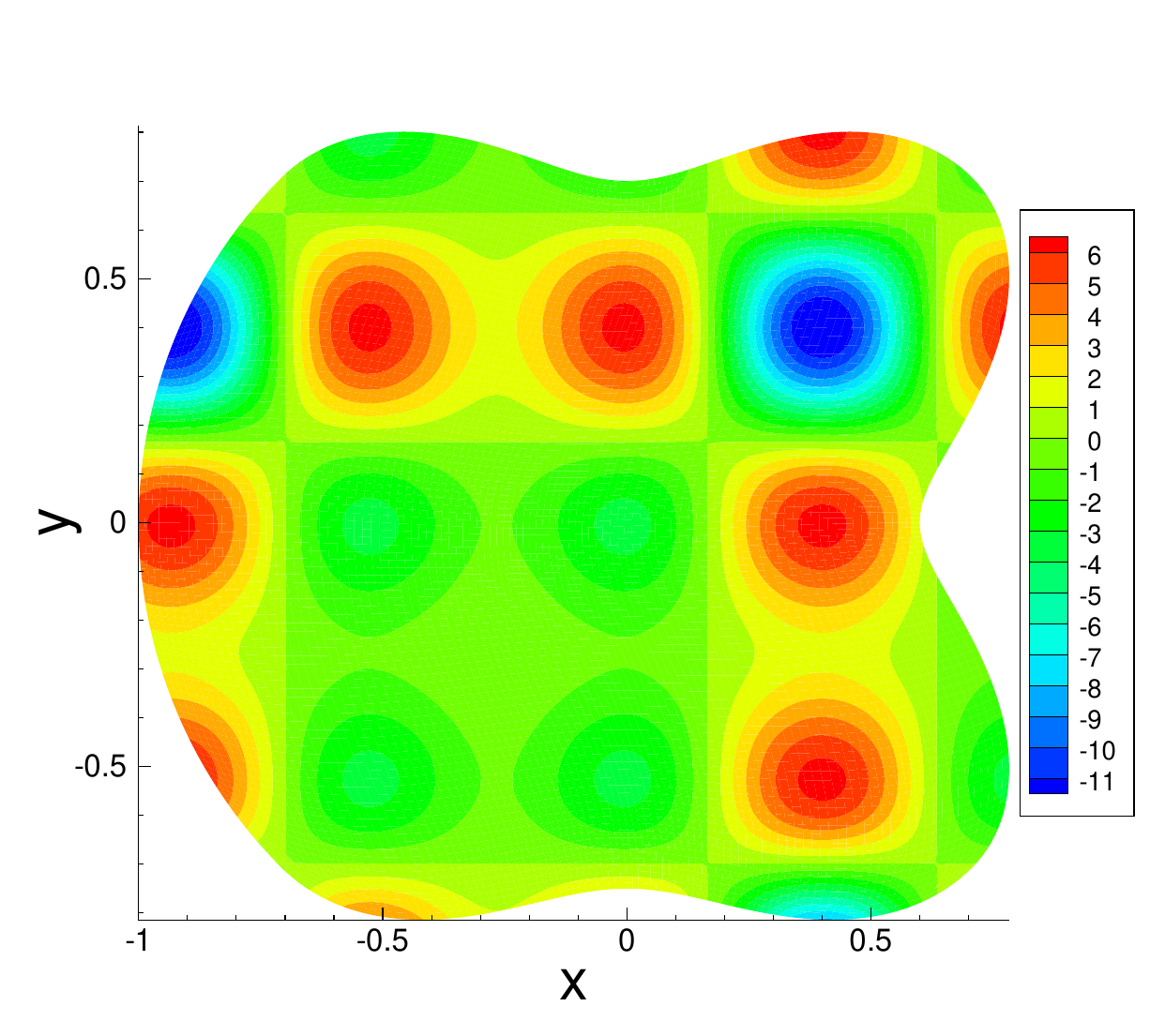}(c)
    \includegraphics[width=1.1in]{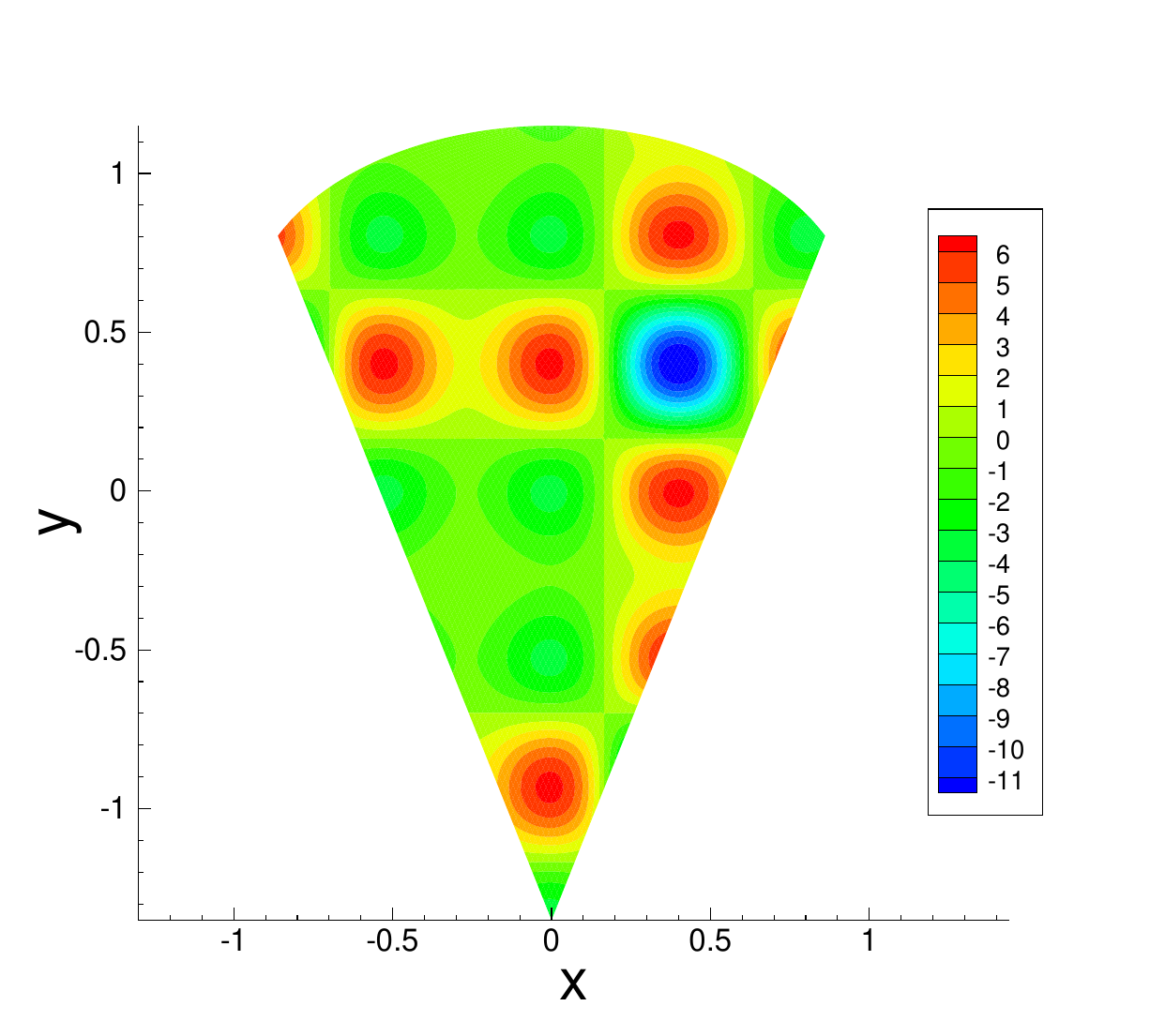}(d)
    \includegraphics[width=1.1in]{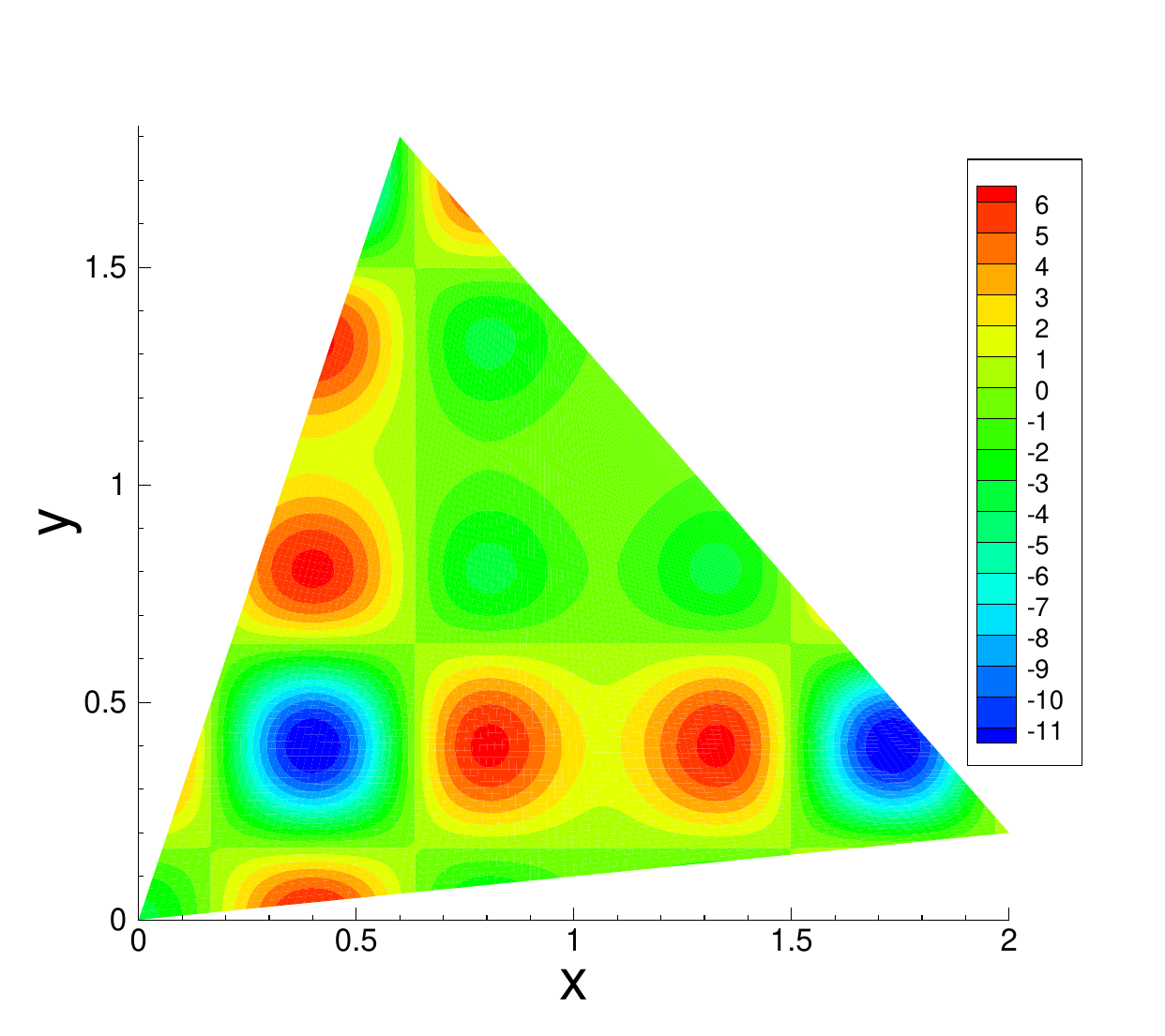}(e)
  }
  \centerline{
    \includegraphics[width=1.1in]{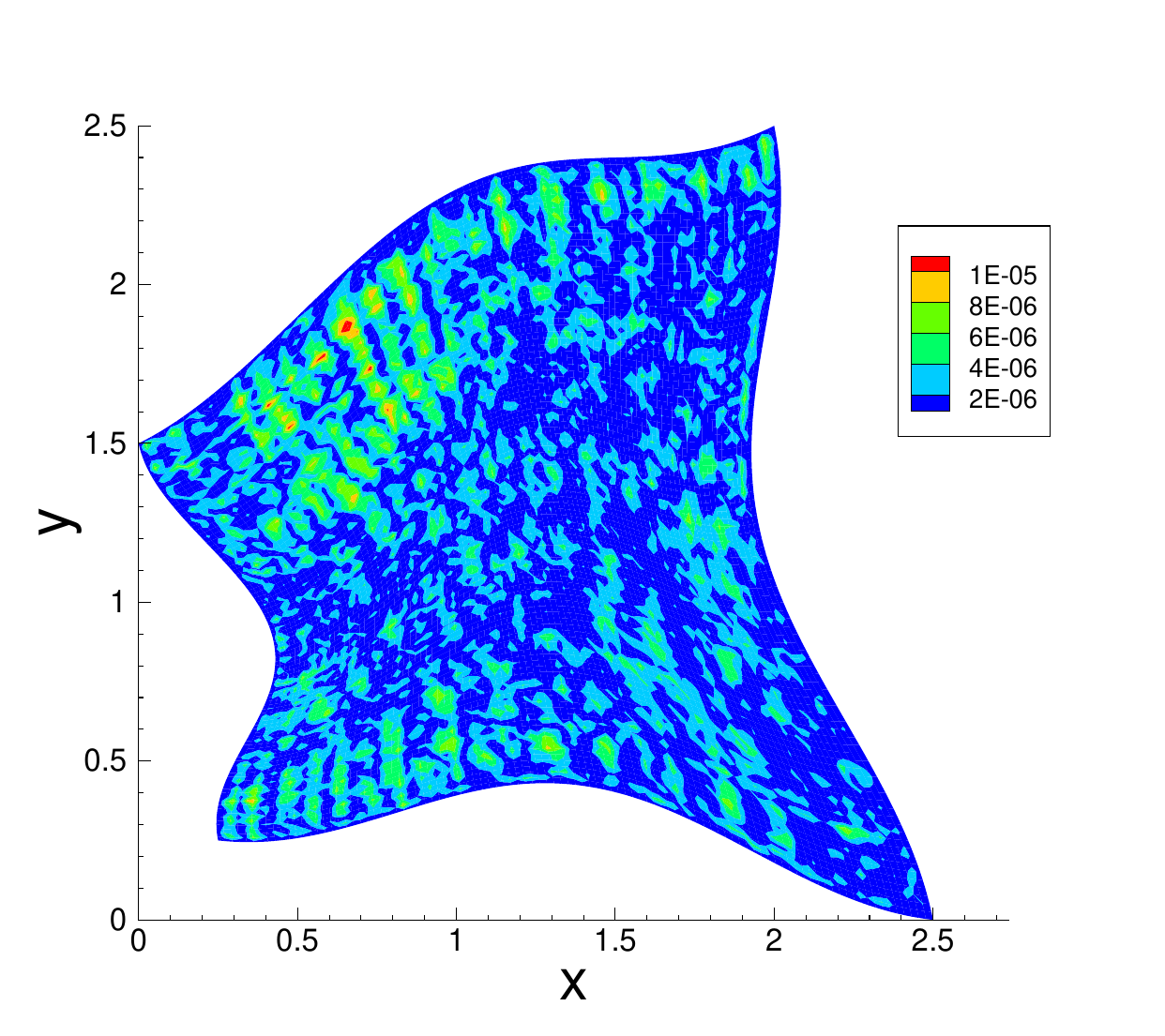}(a)
    \includegraphics[width=1.1in]{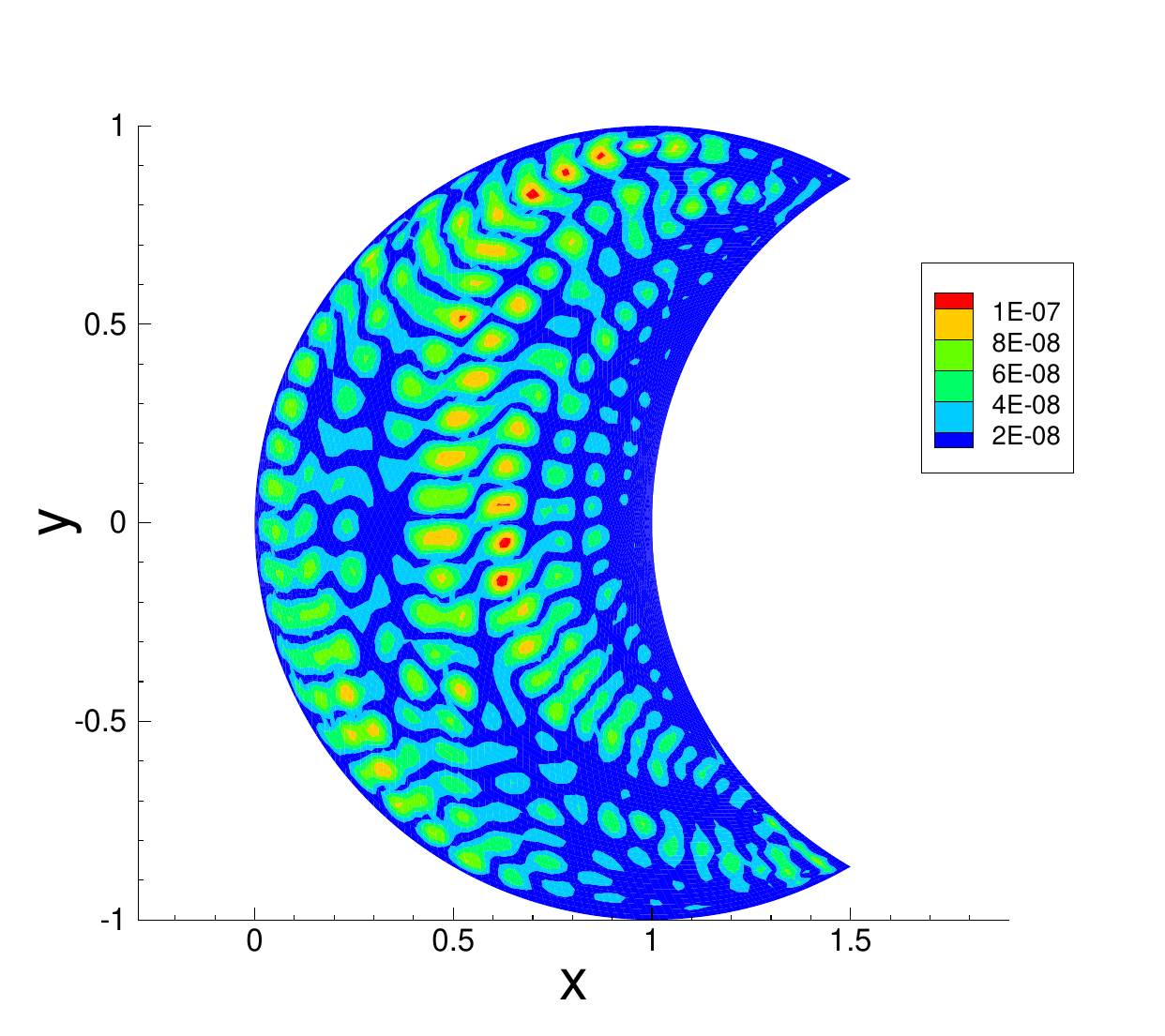}(b)
    \includegraphics[width=1.1in]{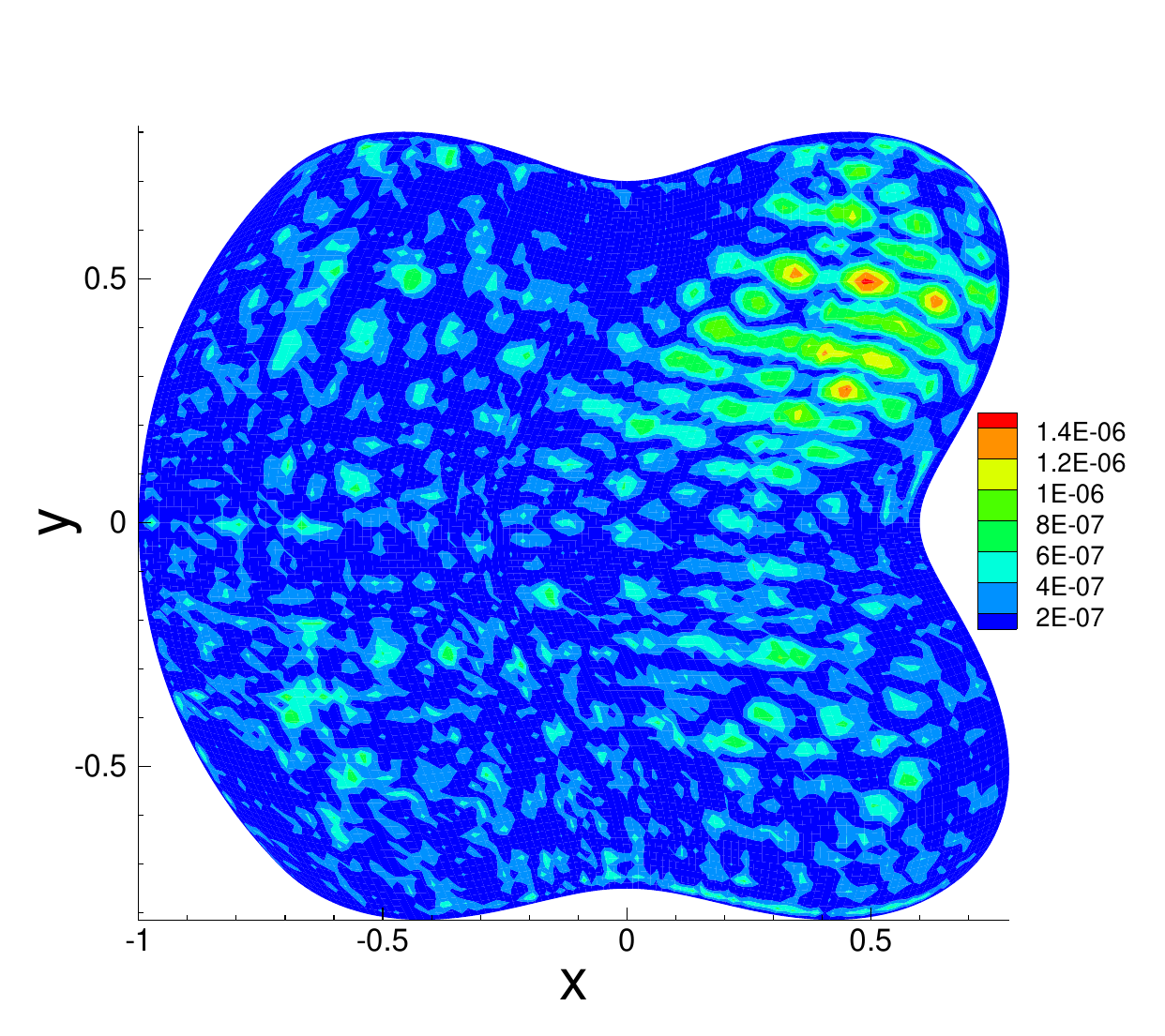}(c)
    \includegraphics[width=1.1in]{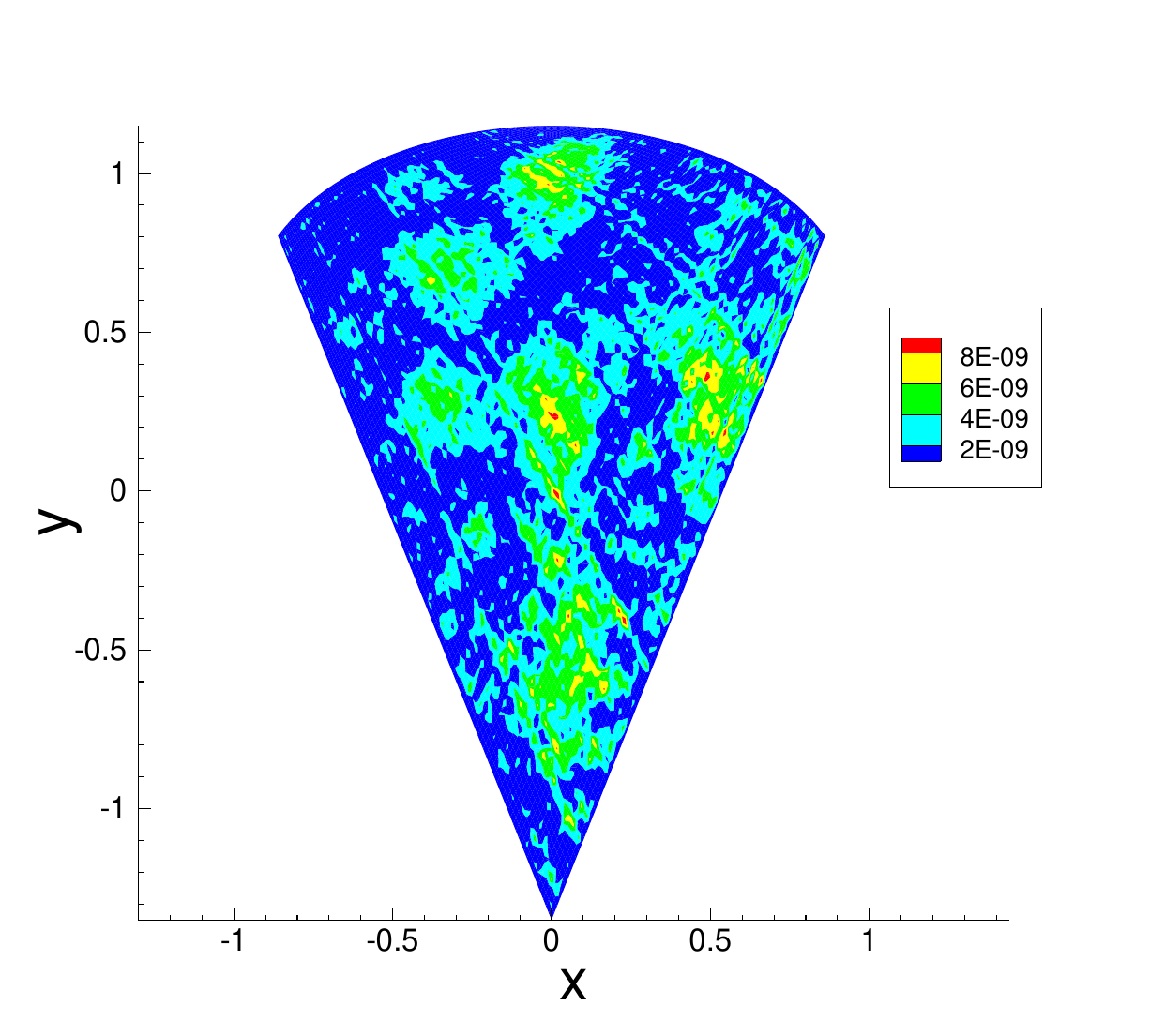}(d)
    \includegraphics[width=1.1in]{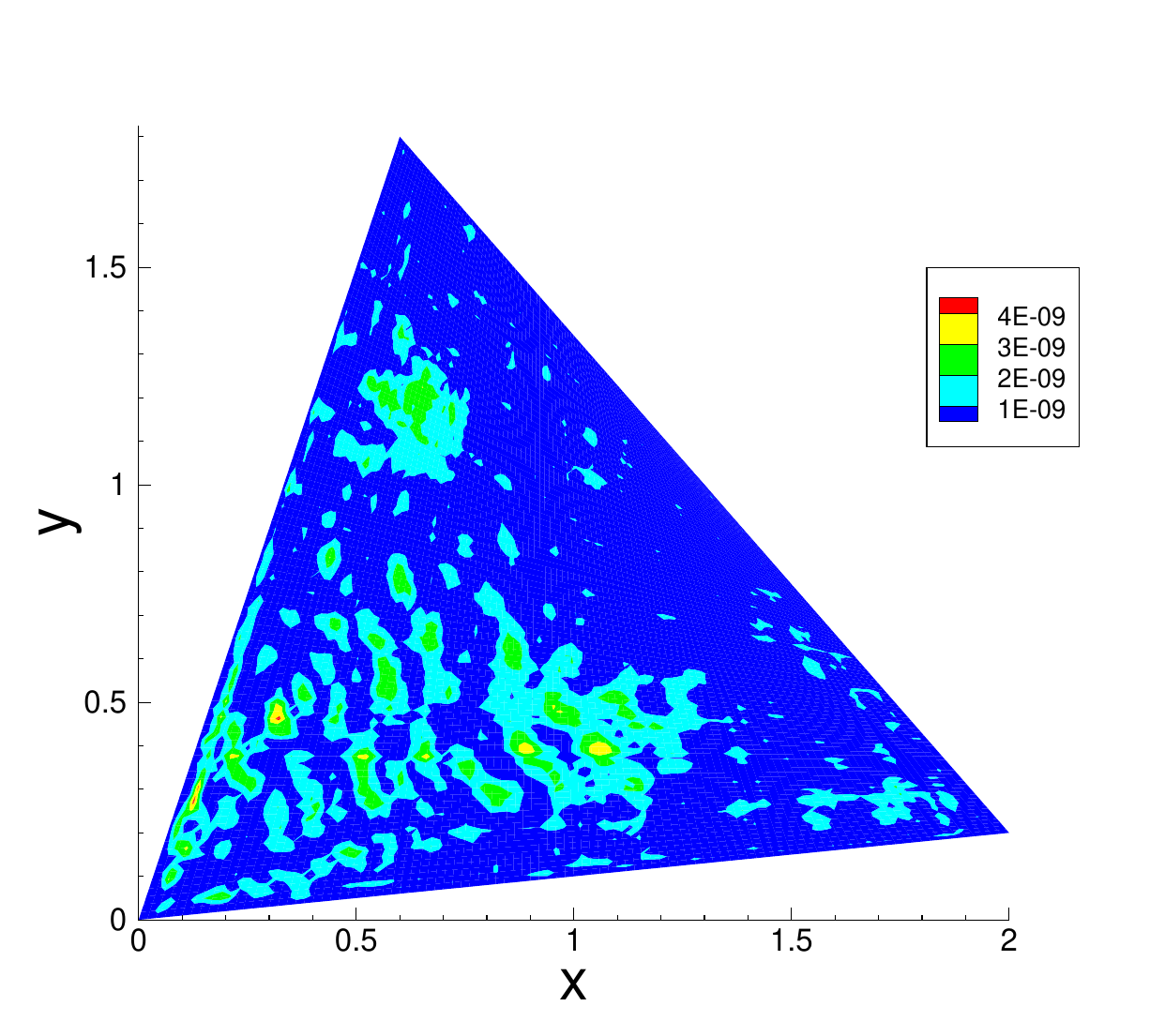}(e)
  }
  \caption{Helmholtz equation (Dirichlet BCs on all boundaries):
    Distributions of the NN solutions (top row) and their point-wise absolute
    errors (bottom row) on
    the five domains. Simulation parameters: Domain \#1, $R_m=4.62$, $Q=70$, $M=800$;
    Domain \#2, $R_m=4.0$, $Q=65$, $M=800$; Domain \#3, $R_m=4.57$, $Q=65$, $M=800$;
    Domain \#4, $R_m=3.53$, $Q=60$, $M=800$; Domain \#5, $R_m=4.17$, $Q=65$, $M=800$. 
  }
  \label{fg_n3}
\end{figure}

\begin{table}
  \centering
  \begin{tabular}{l|l|l|l|l|l}
    \hline
     & domain \#1 & domain \#2 & domain \#3 & domain \#4 & domain \#5 \\ \hline
    max-error (domain) & $1.167E-5$ & $1.115E-7$ & $1.485E-6$ & $1.026E-8$ & $4.350E-9$  \\[3pt]
    rms-error (domain) & $2.752E-6$ & $3.071E-8$ & $1.034E-7$ & $2.727E-9$ & $8.668E-10$  \\[3pt] \hline
    max DBC-error ($\overline{AB}$) & $8.882E-16$ & $4.441E-16$ & $8.882E-16$ & $8.882E-16$ & $4.441E-16$  \\[3pt]
    rms DBC-error ($\overline{AB}$) & $1.051E-16$ & $5.747E-17$ & $1.034E-16$ & $2.271E-16$ & $1.696E-16$  \\[3pt] \hline
    max DBC-error ($\overline{BC}$) & $2.220E-16$ & $0.0$ & $8.882E-16$ & $4.441E-16$ & $8.882E-16$  \\[3pt]
    rms DBC-error ($\overline{BC}$) & $5.747E-17$ & $0.0$ & $1.371E-16$ & $8.999E-17$ & $1.134E-16$  \\[3pt] \hline
    max DBC-error ($\overline{CD}$) & $1.110E-16$ & $8.882E-16$ & $4.441E-16$ & $4.441E-16$ & $2.220E-16$  \\[3pt]
    rms DBC-error ($\overline{CD}$) & $1.105E-17$ & $2.772E-16$ & $6.944E-17$ & $1.426E-16$ & $4.438E-17$  \\[3pt] \hline
    max DBC-error ($\overline{AD}$) & $8.882E-16$ & $2.220E-16$ & $3.469E-18$ & $4.441E-16$ & $4.441E-16$  \\[3pt]
    rms DBC-error ($\overline{AD}$) & $2.131E-16$ & $2.209E-17$ & $3.452E-19$ & $1.361E-16$ & $1.099E-16$  \\[3pt]
    \hline
  \end{tabular}
  \caption{Helmholtz equation (Dirichlet BC on all boundaries):
    maximum and rms NN-solution errors over
    the domain ($e_{max}^{\Omega}$, $e_{rms}^{\Omega}$),
    and the maximum and rms DBC errors on
    the four boundaries ($\varepsilon_{max}^{\overline{AB}}$, $\varepsilon_{rms}^{\overline{AB}}$,
     $\varepsilon_{max}^{\overline{BC}}$, $\varepsilon_{rms}^{\overline{BC}}$,
    $\varepsilon_{max}^{\overline{CD}}$, $\varepsilon_{rms}^{\overline{CD}}$,
    $\varepsilon_{max}^{\overline{AD}}$, $\varepsilon_{rms}^{\overline{AD}}$).
    Simulation parameters follow those of Figure~\ref{fg_n3}.
  }
  \label{tab_3}
\end{table}

We first consider Dirichlet conditions for all boundaries of these domains.
Distributions of the ELM solution over the five domains are shown in Figure~\ref{fg_n3}
(top row), and their point-wise absolute errors are also included (bottom row).
The values for the simulation parameters $R_m$, $Q$ and $M$ are provided in
the figure caption. The error levels of the ELM solution differ on different domains,
with the maximum error generally ranging from $10^{-9}$ to $10^{-5}$ in these simulations.

The boundary-condition errors of the ELM solution for the five domains
are illustrated in Table~\ref{tab_3}.
This table lists the maximum and rms DBC errors ($\varepsilon_{max}$, $\varepsilon_{rms}$)
on different boundaries ($\overline{AB}$, $\overline{BC}$,
$\overline{CD}$, $\overline{AD}$), together with the maximum/rms NN solution error
over the domains. The simulation parameters for this table follow those
of Figure~\ref{fg_n3}.
It is evident that the current method has enforced the Dirichlet BCs on these
domain geometries to the machine accuracy.


\begin{figure}
  \centerline{
    \includegraphics[width=1.5in]{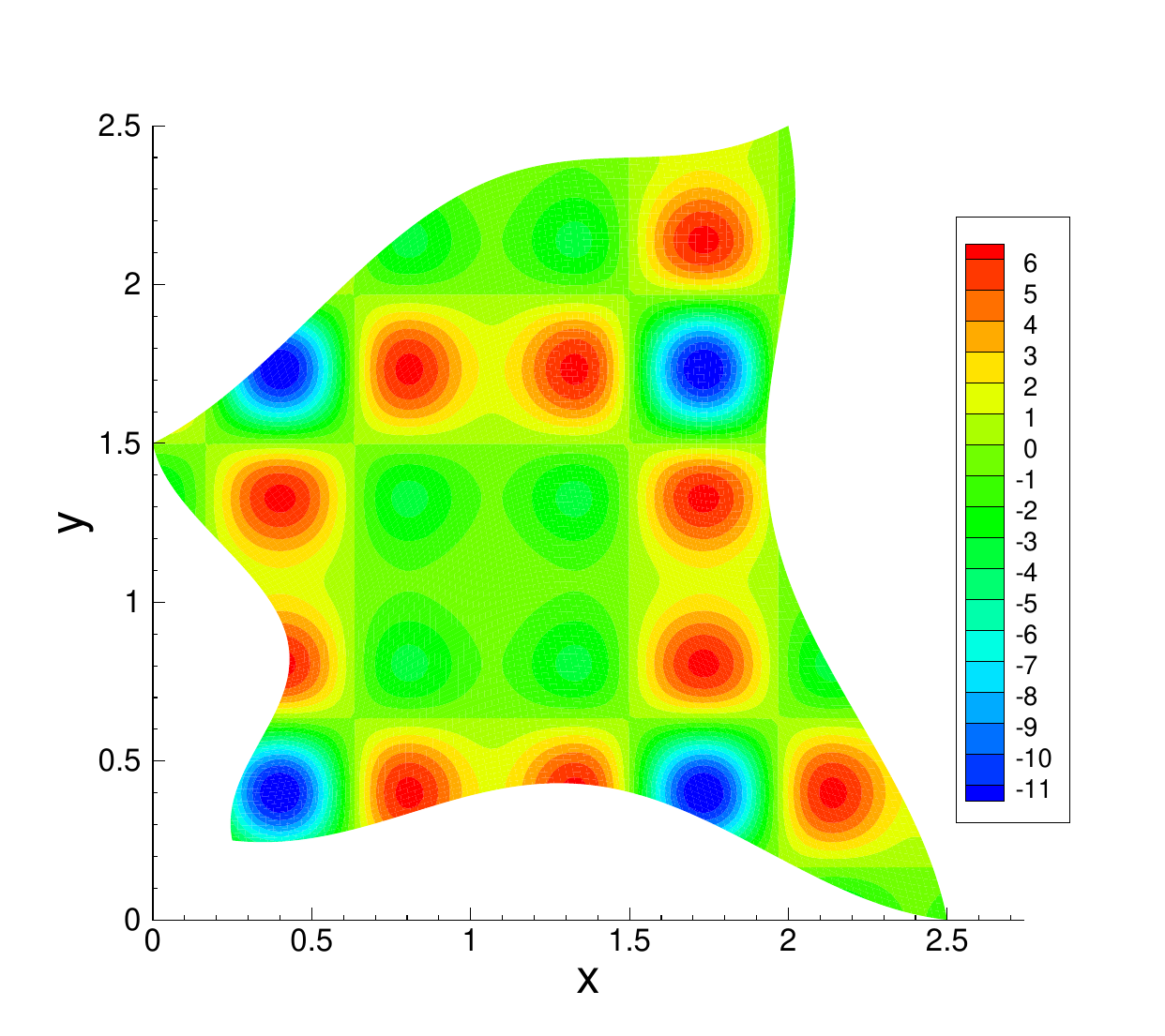}(a)
    \includegraphics[width=1.5in]{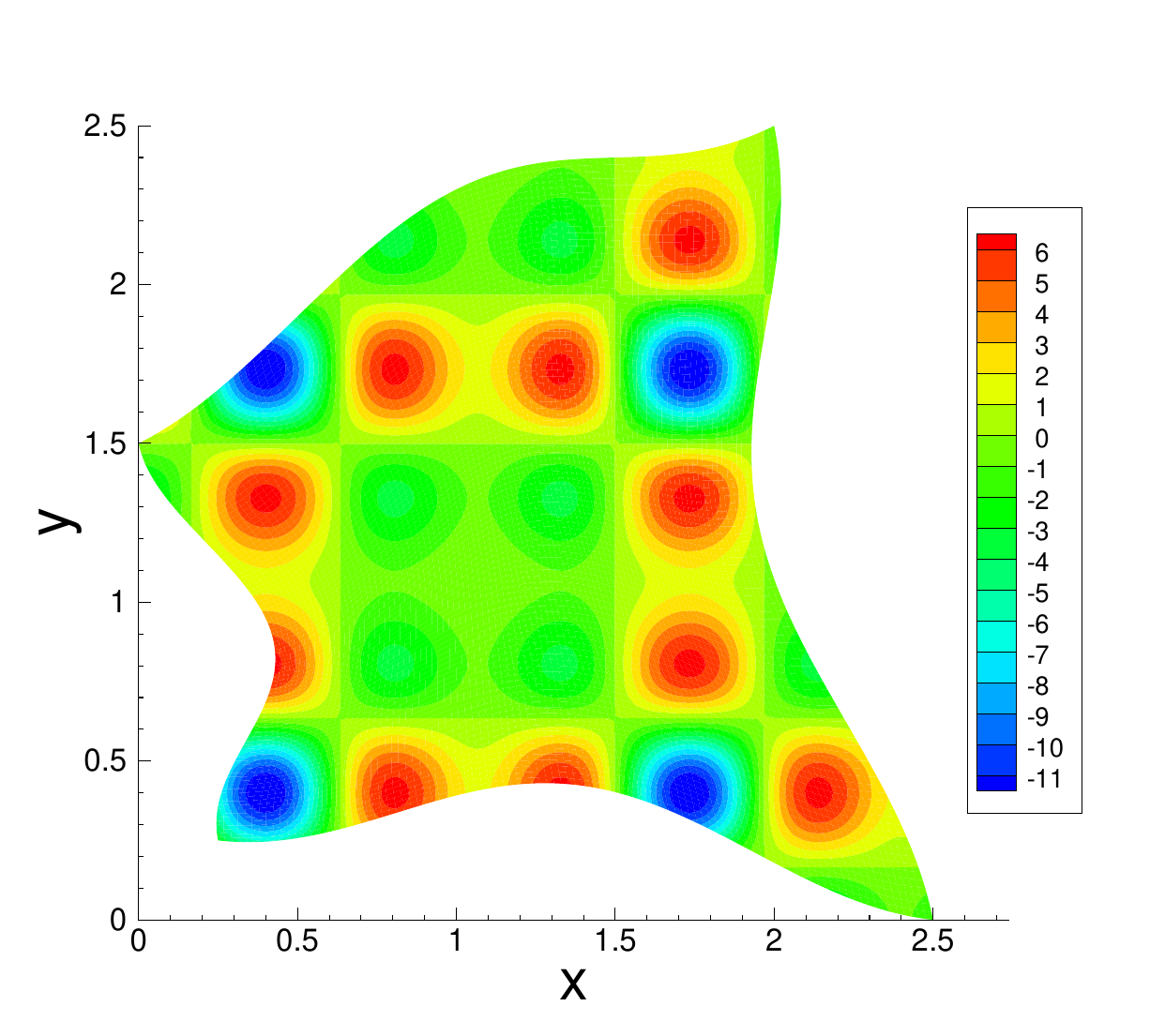}(b)
    \includegraphics[width=1.5in]{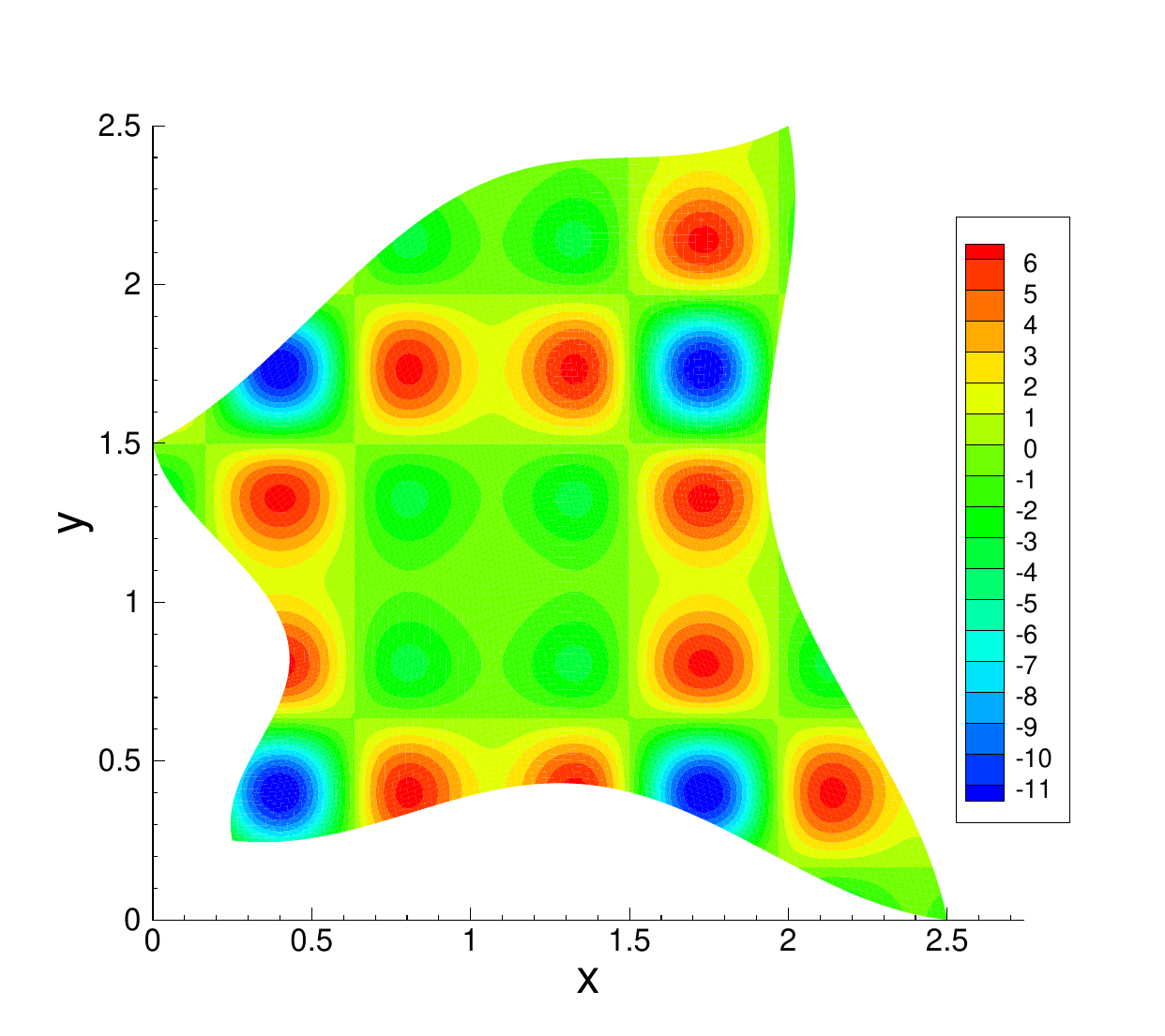}(c)
  }
  \centerline{
    \includegraphics[width=1.5in]{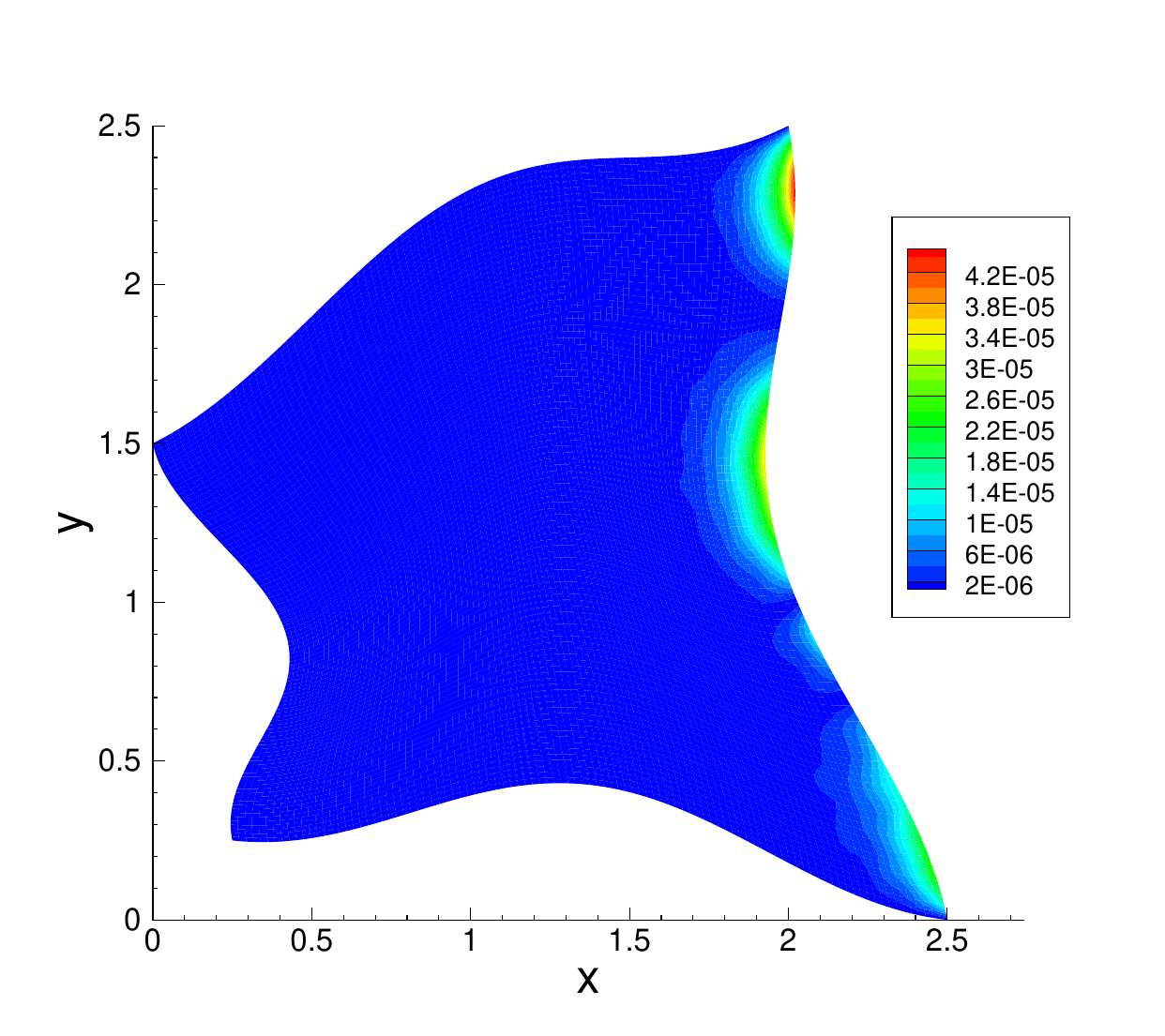}(d)
    \includegraphics[width=1.5in]{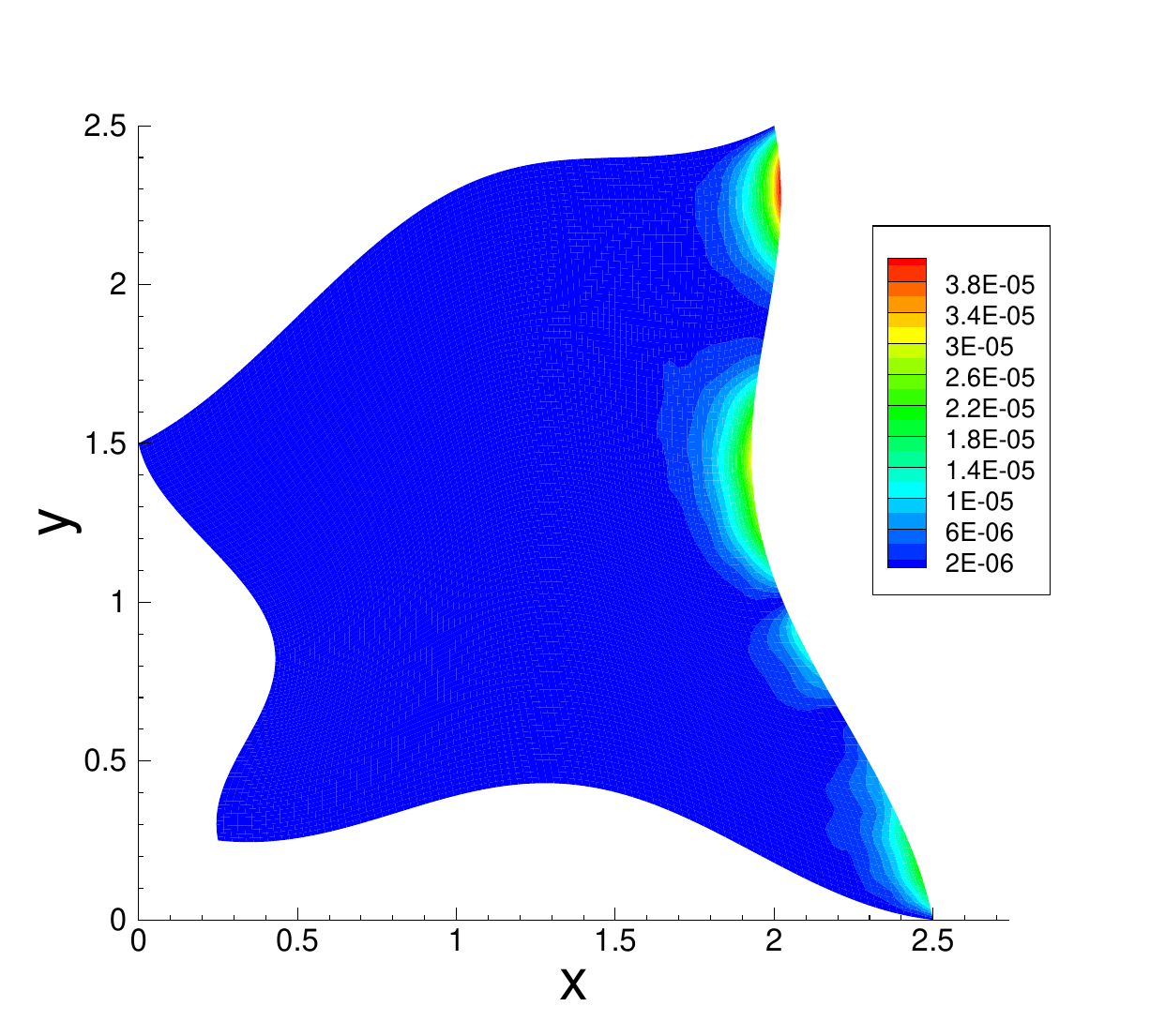}(e)
    \includegraphics[width=1.5in]{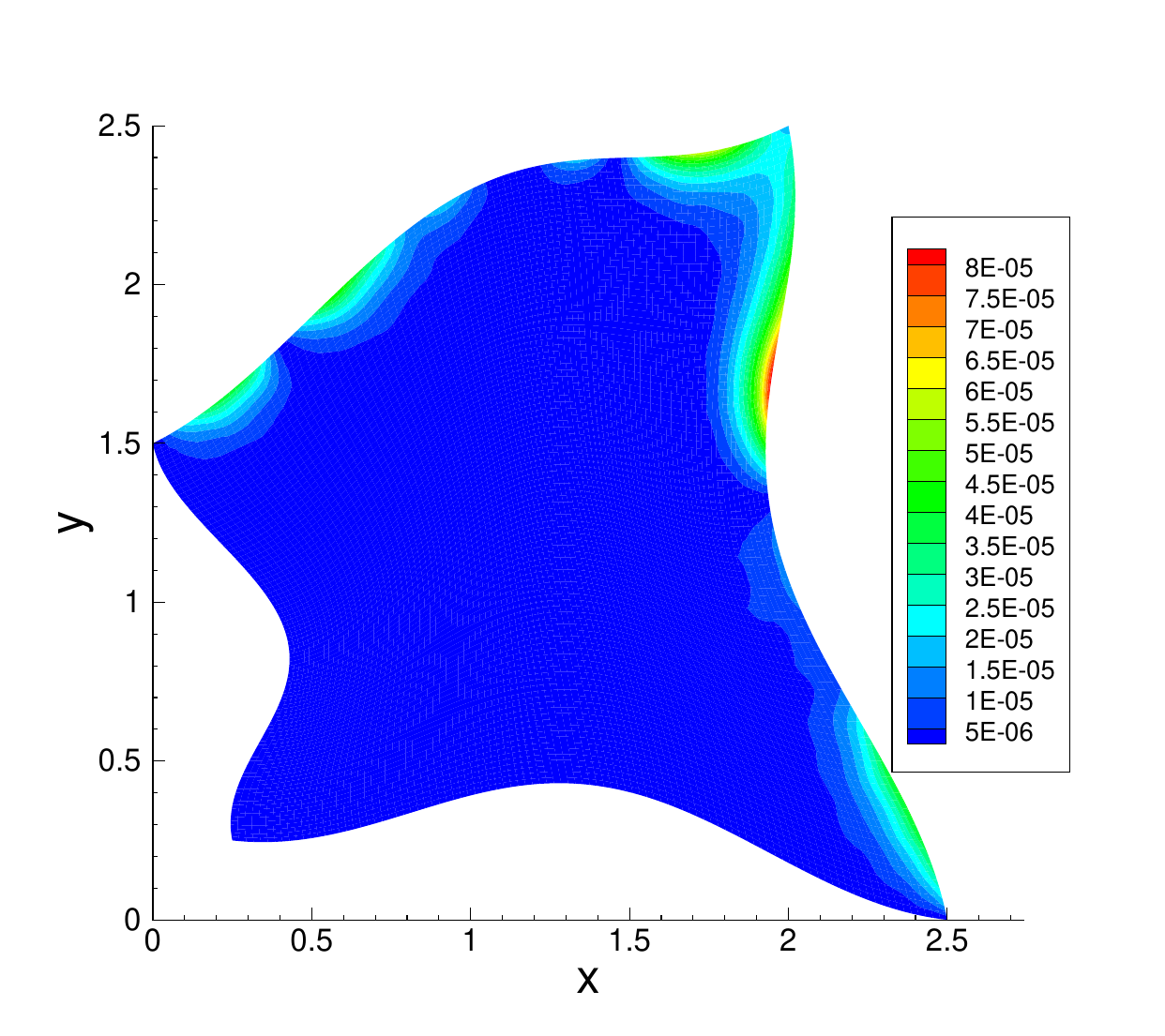}(f)
  }
  \caption{Helmholtz equation (Neumann or Robin BC on boundaries):
    Distributions of the NN solutions (top row) and their point-wise absolute
    errors (bottom row)  on domain \#1.
    (a,d) Case \#1: NBC on $\overline{BC}$, and DBCs on the other boundaries.
    (b,e) Case \#2: RBC ($\alpha=1.0$) on $\overline{BC}$, and DBCs on the other boundaries.
    (c,f) Case \#3: NBCs on $\overline{BC}$ and $\overline{CD}$,
    and DBCs on the other boundaries.
    Simulation parameters: $R_m=5.0$, $Q=70$, $M=950$ for all cases.
  }
  \label{fg_n4}
\end{figure}

\begin{table}
  \centering
  \begin{tabular}{l|c|c|c}
    \hline
     & Case \#1 & Case \#2 &  Case \#3 \\ \hline
    max solution-error (domain) & $4.536E-5$ & $4.217E-5$ & $8.295E-5$  \\[3pt]
    rms solution-error (domain) & $3.775E-6$ & $3.420E-16$ & $8.818E-6$  \\[3pt] \hline
    max DBC-error ($\overline{AB}$) & $0.0$ & $0.0$ & $0.0$  \\[3pt]
    rms DBC-error ($\overline{AB}$) & $0.0$ & $0.0$ & $0.0$  \\[3pt] \hline
    max NBC- or RBC-error ($\overline{BC}$) & $1.421E-14$ & $1.421E-14$ & $1.421E-14$  \\[3pt]
    rms NBC- or RBC-error ($\overline{BC}$) & $3.614E-15$ & $4.180E-15$ & $4.154E-15$  \\[3pt] \hline
    max DBC- or NBC-error ($\overline{CD}$) & $0.0$ & $0.0$ &  $1.421E-14$ \\[3pt]
    rms DBC- or NBC-error ($\overline{CD}$) & $0.0$ & $0.0$ & $3.065E-15$ \\[3pt] \hline
    max DBC-error ($\overline{AD}$) & $8.882E-16$ & $8.882E-16$ & $1.776E-15$  \\[3pt]
    rms DBC-error ($\overline{AD}$) & $2.947E-16$ & $2.947E-16$ & $5.738E-16$ \\[3pt]
    \hline
  \end{tabular}
  \caption{Helmholtz equation on domain \#1 (Neumann or Robin BCs):
    maximum and rms NN-solution errors over
    the domain, and the maximum and rms boundary-condition (DBC, NBC, RBC) errors on
    the boundaries.
    Different cases correspond to those in Figure~\ref{fg_n4}.
  }
  \label{tab_4}
\end{table}

We next consider Neumann and Robin conditions on the domain.
Figure~\ref{fg_n4} and Table~\ref{tab_4} illustrate the ELM results
obtained for a combination of Dirichlet conditions with Neumann or Robin
conditions. Figure~\ref{fg_n4} shows the NN solutions (top row) and
their point-wise errors (bottom row) for three cases.
In case \#1 (plots (a,d)), Neumann condition is imposed on the
boundary $\overline{BC}$ and Dirichlet conditions are imposed on the
other boundaries. In case \#2 (plots (b,e)), we impose
the Robin condition~\eqref{eq_65b} with $\alpha_{BC}=1$ on
the $\overline{BC}$, with the rest being Dirichlet boundaries.
In case \#3 (plots (c,f)), we impose Neumann conditions on
the boundaries $\overline{BC}$ and $\overline{CD}$, and Dirichlet
conditions on the other boundaries.
The simulation parameter values are specified in the figure caption.
The plots indicate that the largest solution errors generally occur
on or near the Neumann or Robin boundaries, while the
NN solution error is generally much smaller on the Dirichlet boundaries
and in the interior of the domain. The maximum solution error over the domain
is on the order of $10^{-5}$ for these cases.

Table~\ref{tab_4} lists the maximum and rms boundary-condition errors
for different boundaries of these three cases ($\varepsilon_{max}$,
$\varepsilon_{rms}$), together with the maximum/rms
solution errors over the domain ($e_{max}^{\Omega}$, $e_{rms}^{\Omega}$).
Note that the error on $\overline{BC}$ stands for the RBC error for case \#2
and the NBC error for cases \#1 and \#3, and that
the error on $\overline{CD}$ stands for the NBC error for case \#3 and
the DBC error for cases \#1 and \#2.
The data show  that the current method has enforced
the Dirichlet, Neumann, and Robin boundary conditions to the machine
accuracy.

\subsection{Nonlinear Helmholtz Equation}
\label{sec_32}

\begin{figure}
  \centerline{
    \includegraphics[width=1.4in]{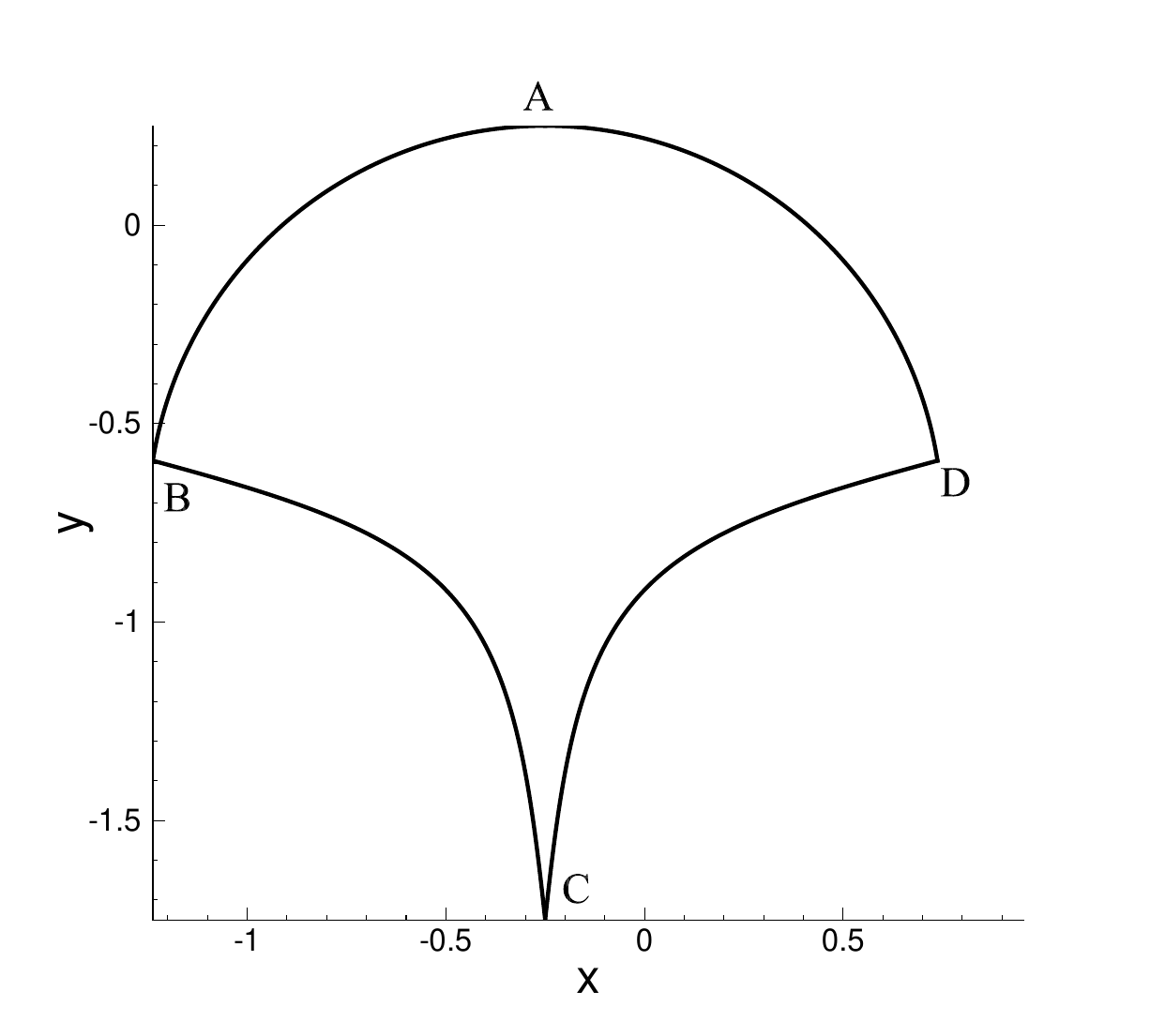}(a)
    \includegraphics[width=1.4in]{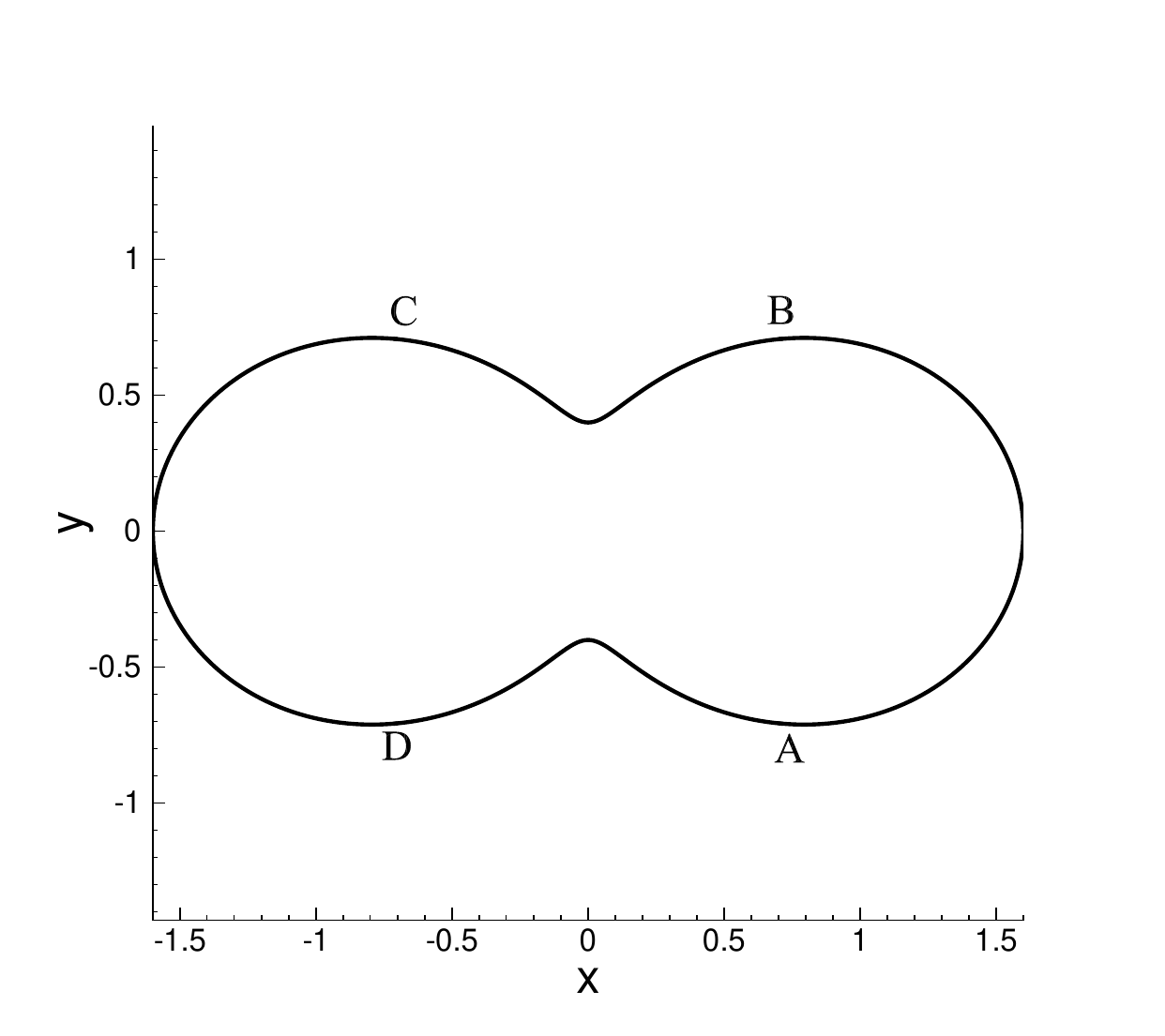}(b)
    \includegraphics[width=1.4in]{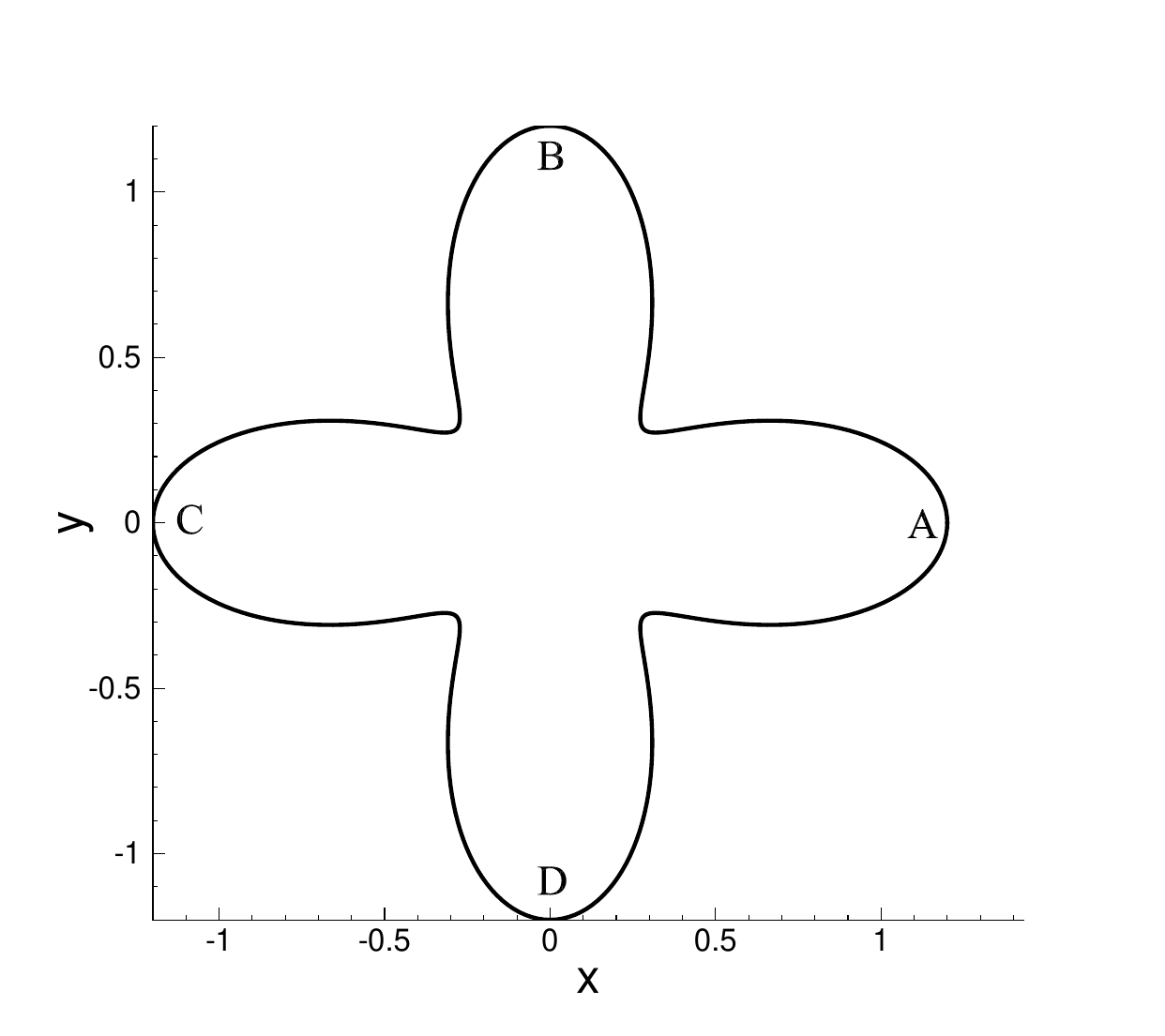}(c)
    \includegraphics[width=1.4in]{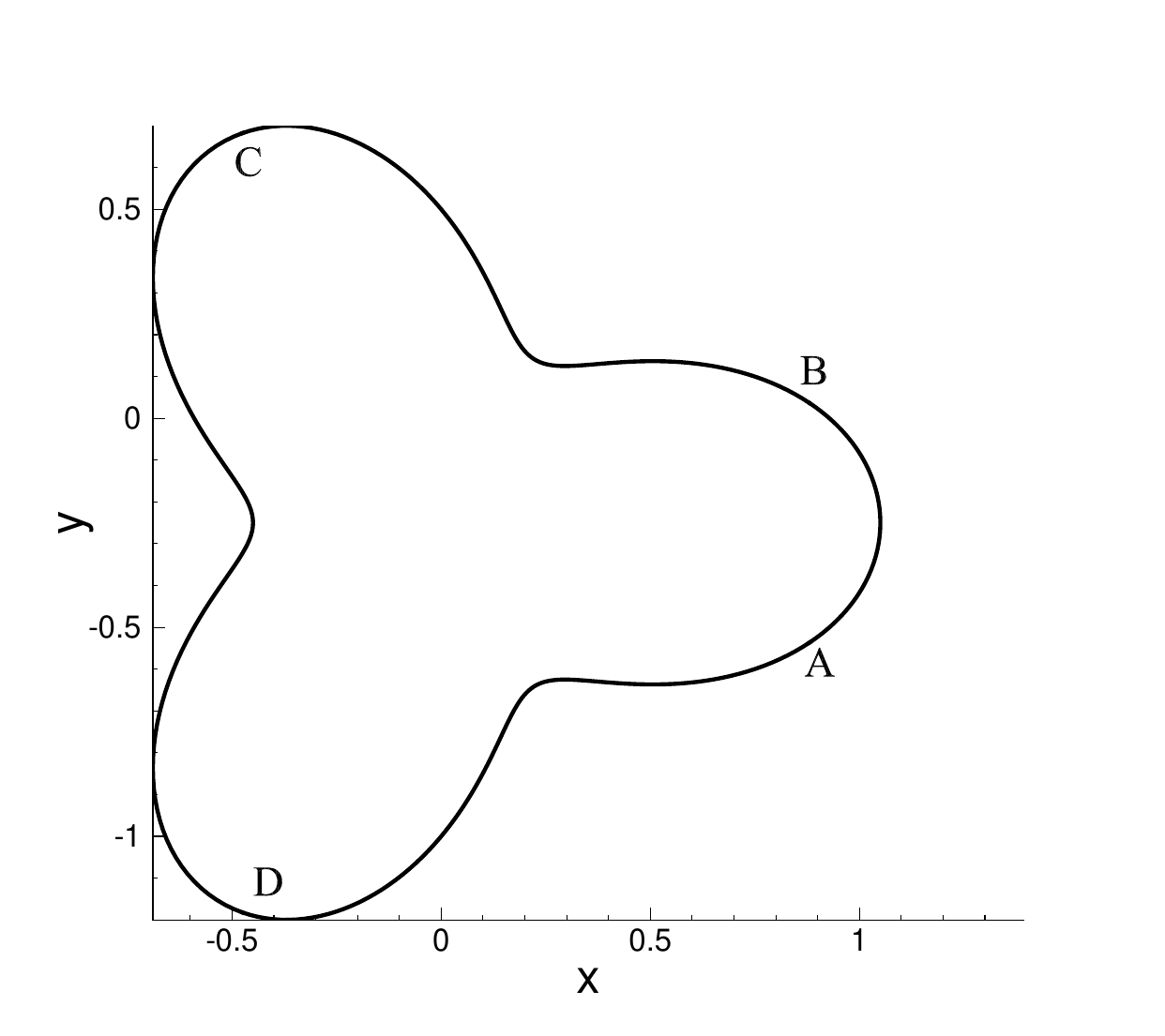}(d)
  }
  \centerline{
    \includegraphics[width=1.4in]{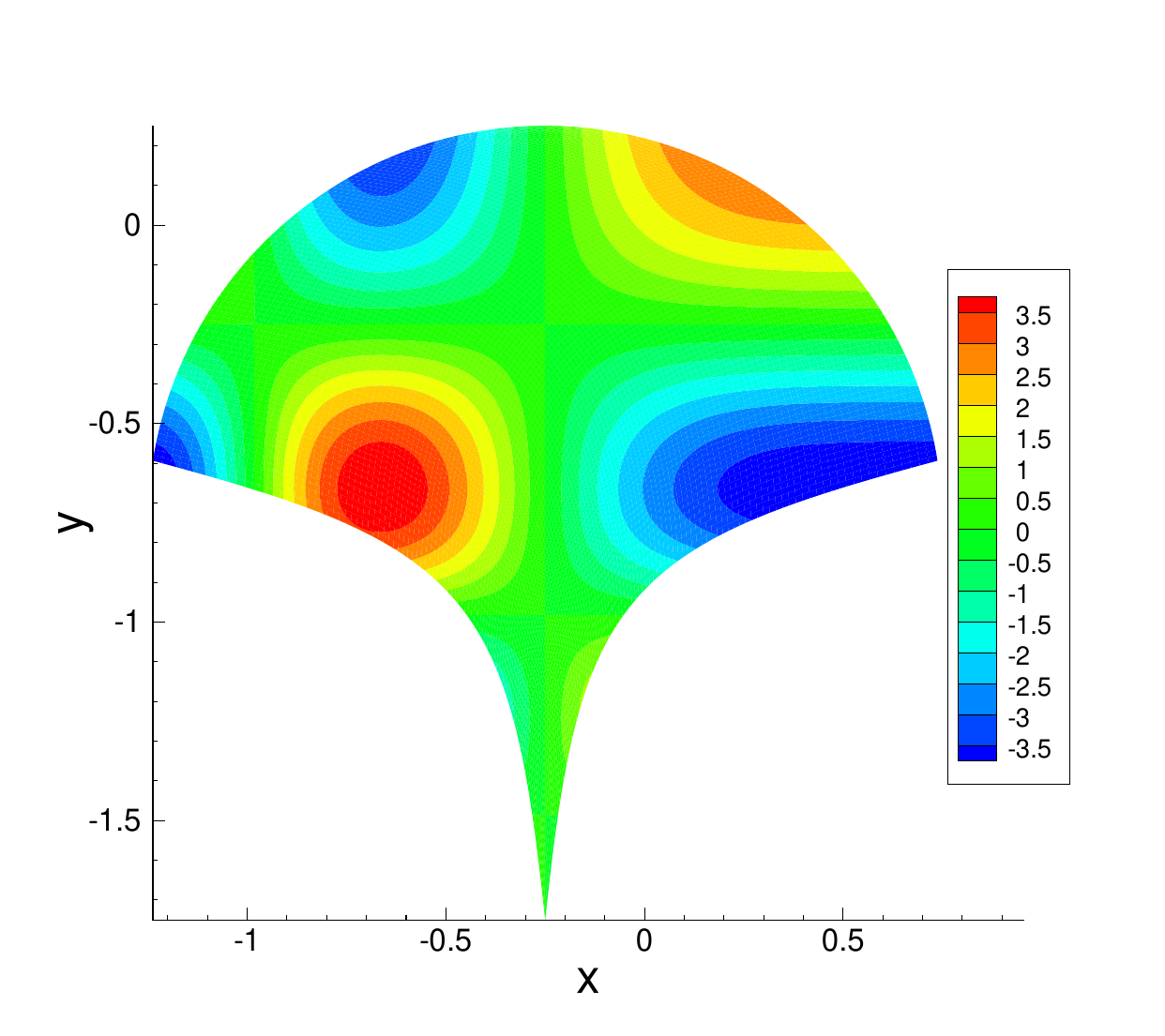}(e)
    \includegraphics[width=1.4in]{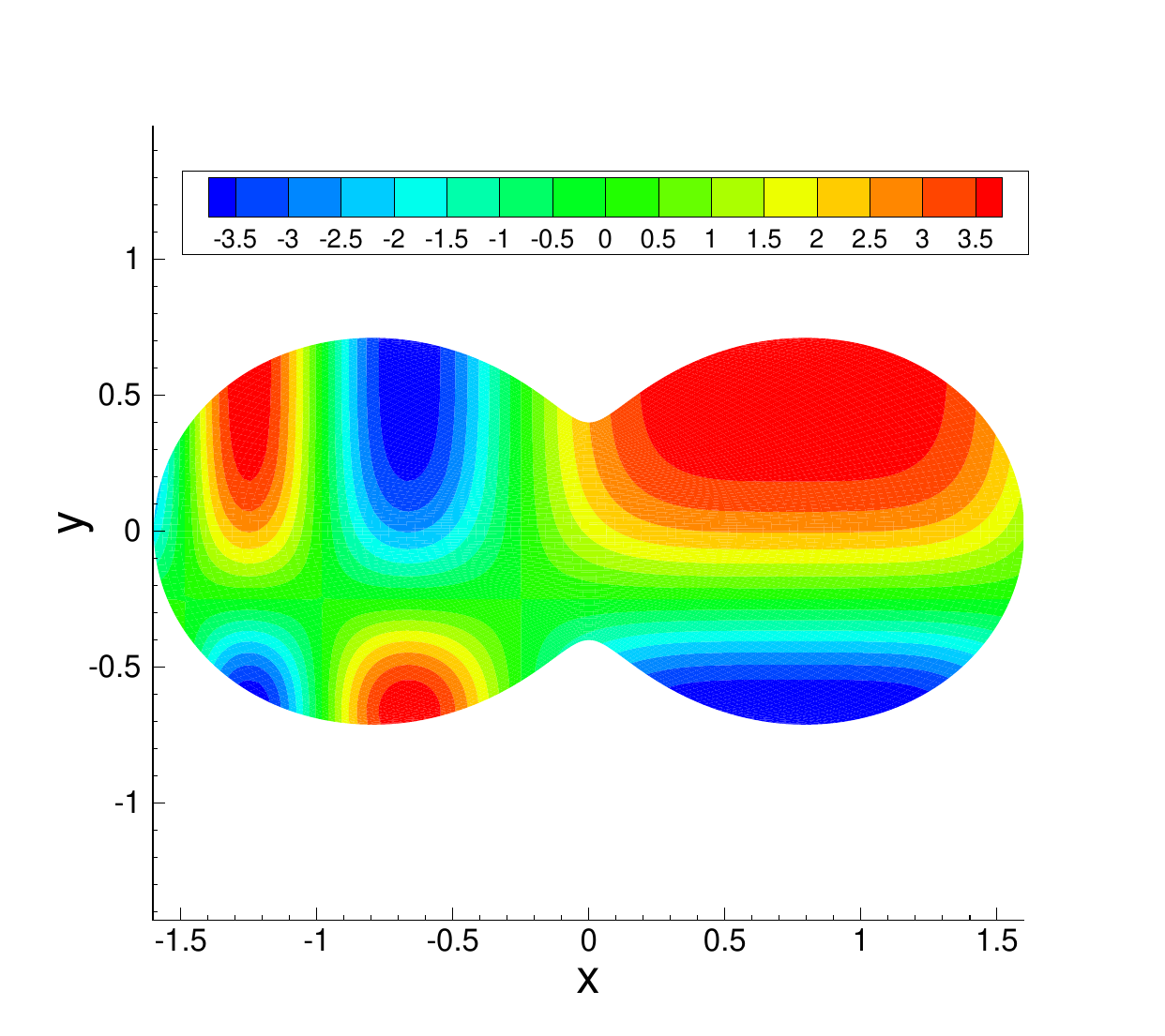}(f)
    \includegraphics[width=1.4in]{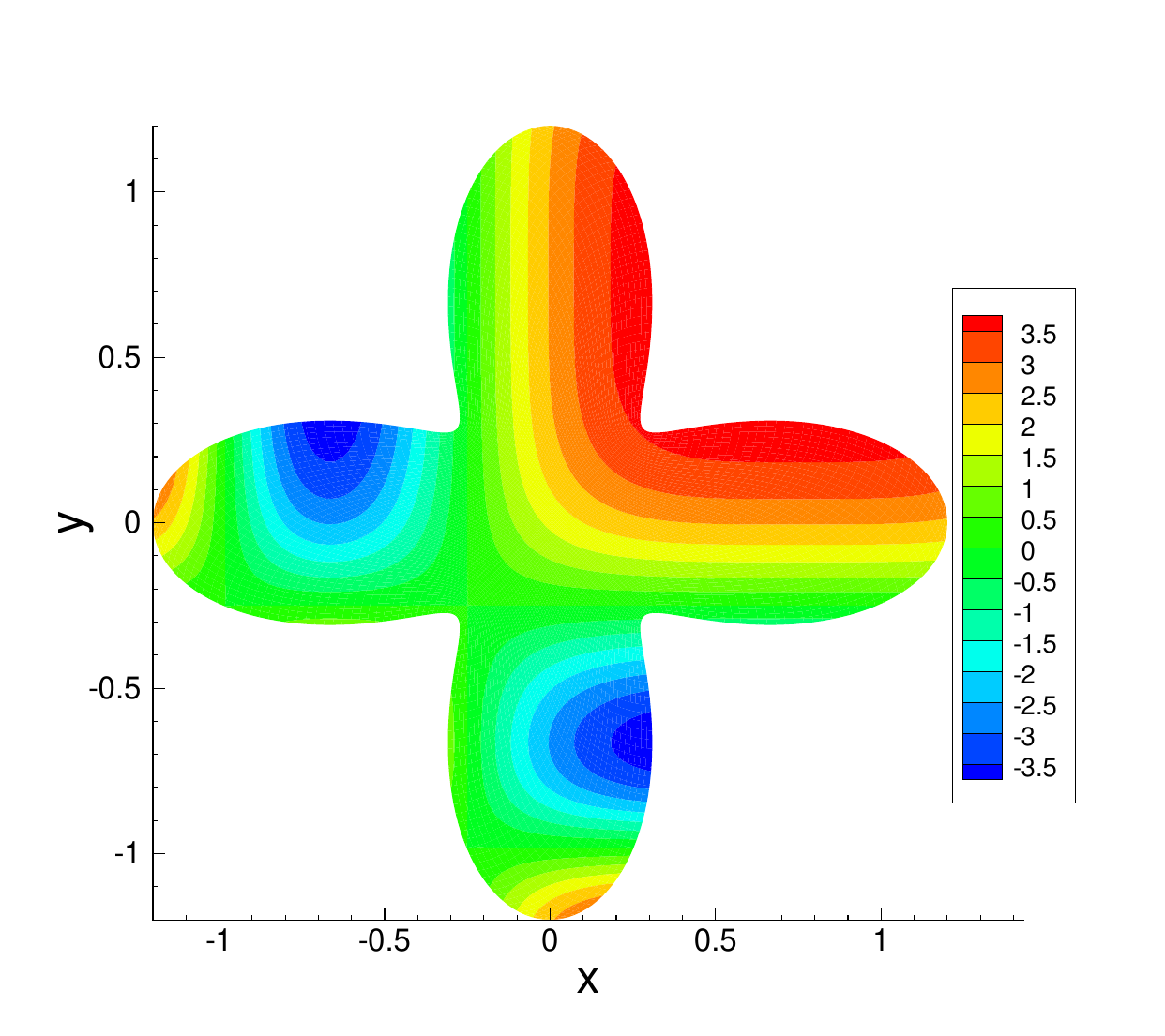}(g)
    \includegraphics[width=1.4in]{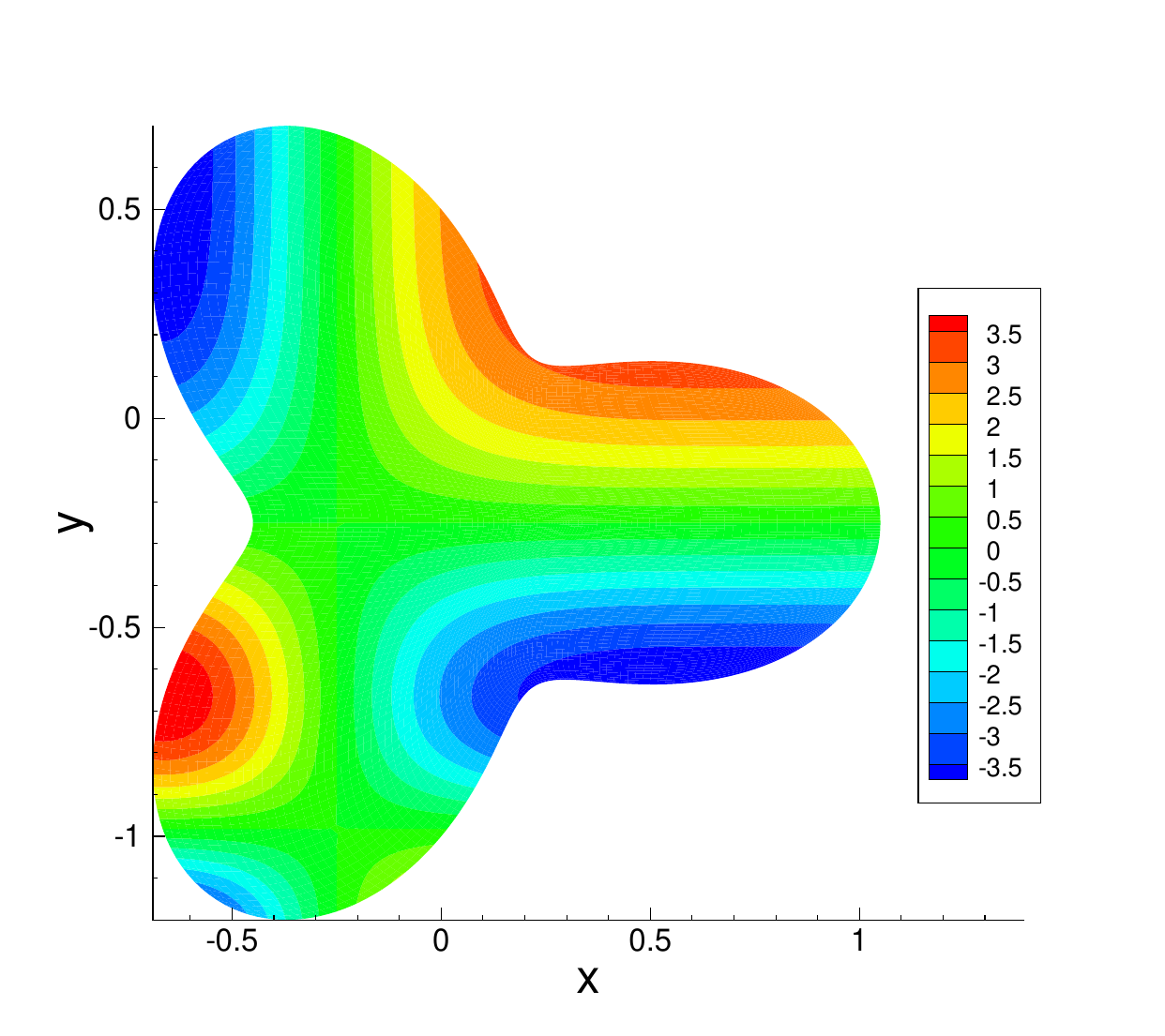}(h)
  }
  \caption{Nonlinear Helmholtz equation: Domain geometries (top row)  and
    the exact solutions (bottom row), (a,e) domain \#1, (b,f) domain \#2,
    (c,g) domain \#3, and (d,h) domain \#4.
  }
  \label{fg_5}
\end{figure}

We next investigate the boundary value problem with
the 2D nonlinear Helmholtz equation on the four domains $\Omega$
as depicted in Figure~\ref{fg_5} (top row),
\begin{subequations}
  \begin{align}
    & \frac{\partial^2u}{\partial x^2} + \frac{\partial^2u}{\partial y^2}
    - 20 u + 10\cos(2u) = f(x,y), \\
    & \left.\mathcal B u(x,y)\right|_{(x,y)\in\partial\Omega}=f_b(x,y),
  \end{align}
\end{subequations}
where $u(x,y)$ is the field function to be computed, $f$ and $f_b$
are source terms for the PDE and the boundary conditions, and
$\mathcal B$ again denotes the Dirichlet, Neumann or Robin
boundary conditions.
The geometric parameters for these domains are specified in
the appendix (Section~\ref{sec_geom}).
We set the source terms appropriately for different boundary conditions such
that the problem has the following exact solution,
\begin{align}
  & u(x,y) = 4\cos\left[\frac{\pi}{2}\left(x-\frac34 \right)^2 \right]
  \cos\left[\frac{\pi}{2}\left(y-\frac34 \right)^2 \right].
\end{align}
Figure~\ref{fg_5} (bottom row) shows distributions of the
exact solution over these domains.

For domains \#1 and \#4 (Figures~\ref{fg_5}a and~\ref{fg_5}d),
we employ the function in~\eqref{eq_9}
to map the problem domain $\Omega$ to the standard domain $\Omega_{st}$.
This function, however, fails to produce
a univalent map for domains \#2 and \#3.
In the following simulations we employ the following
mapping function for domains \#2 and \#3,
\begin{align}
  \mbs x(\xi,\eta) =&\ \mbs x_{AD}(\eta)\varpi_0(\xi) + \mbs x_{BC}(\eta)\varpi_2(\xi)
  + \mbs x_{AB}(\xi)\varpi_0(\eta) + \mbs x_{CD}(\xi)\varpi_2(\eta)
  + \mbs x_{I}\varpi_1(\xi)\varpi_1(\eta) \notag \\
  &\ - \left[\mbs x_A\varpi_0(\eta) + \mbs x_B\varpi_2(\xi) \right]\varpi_0(\eta)
  -\left[\mbs x_D\varpi_0(\xi) + \mbs x_C\varpi_2(\xi) \right]\varpi_2(\eta).
\end{align}
Here $\mbs x_I=(0,0)$ denotes the geometric center of the domains \#2 and \#3,
and $\varpi_i(\xi)$ ($i=0,1,2$) denote the Lagrange polynomials defined on
the points $\{-1,\ 0,\ 1 \}$, as given by
\begin{align}
  \varpi_0(\xi) = \frac12\xi(\xi-1), \quad
  \varpi_1(\xi) = (1+\xi)(1-\xi), \quad
  \varpi_2(\xi) = \frac12\xi(\xi+1), \quad \xi\in[-1,1].
\end{align}

\begin{figure}
  \centerline{
    \includegraphics[width=1.4in]{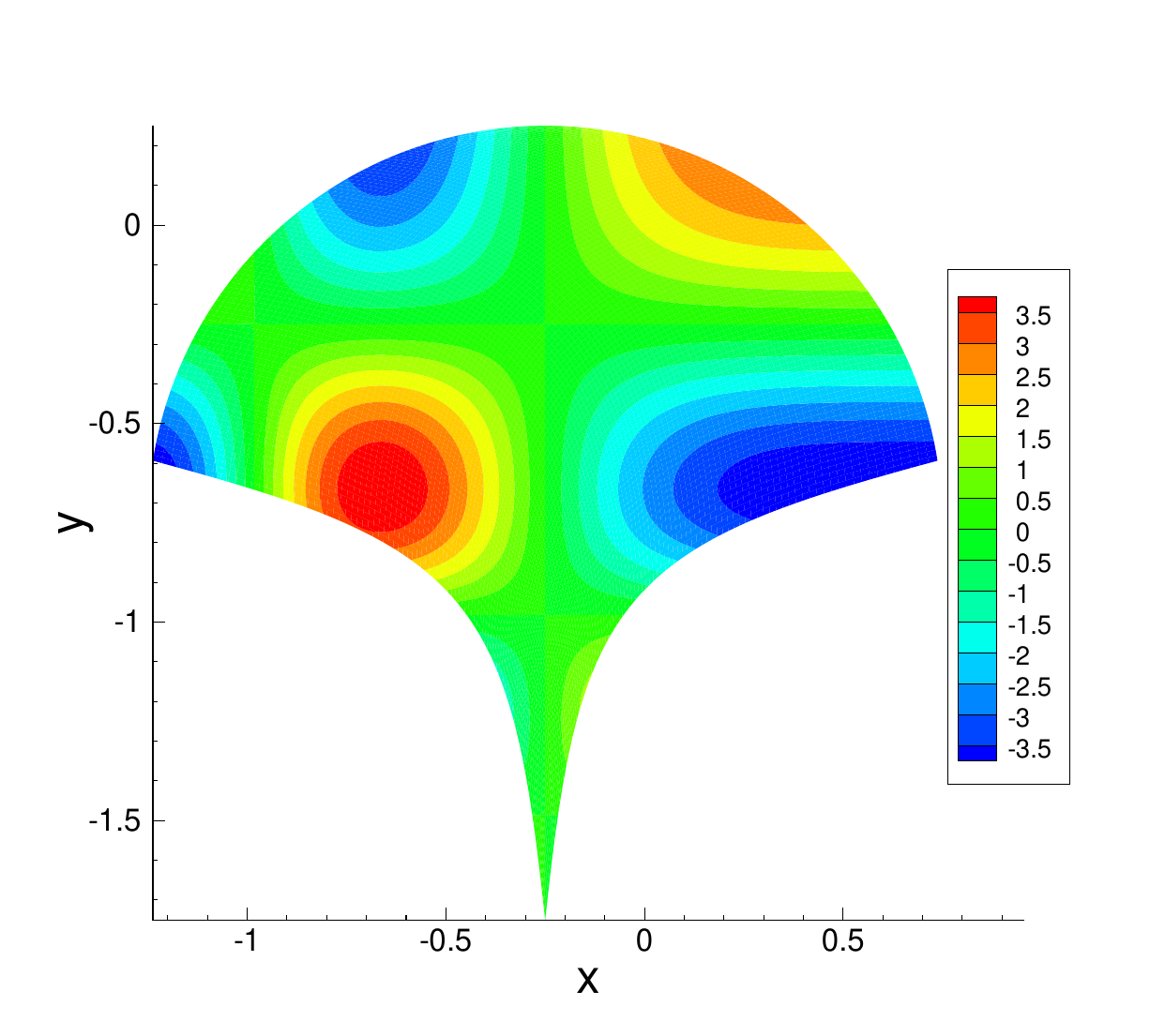}(a)
    \includegraphics[width=1.4in]{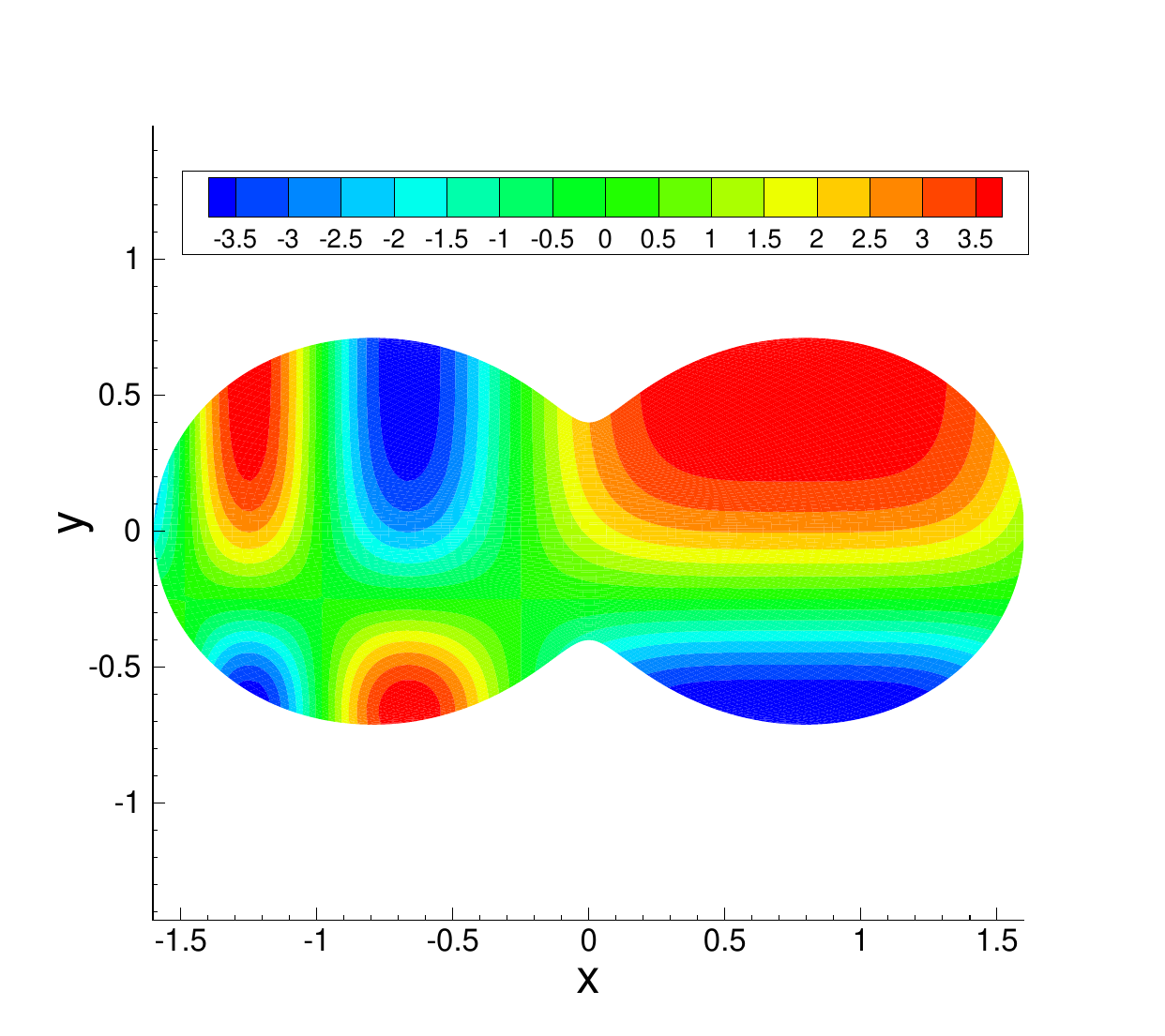}(b)
    \includegraphics[width=1.4in]{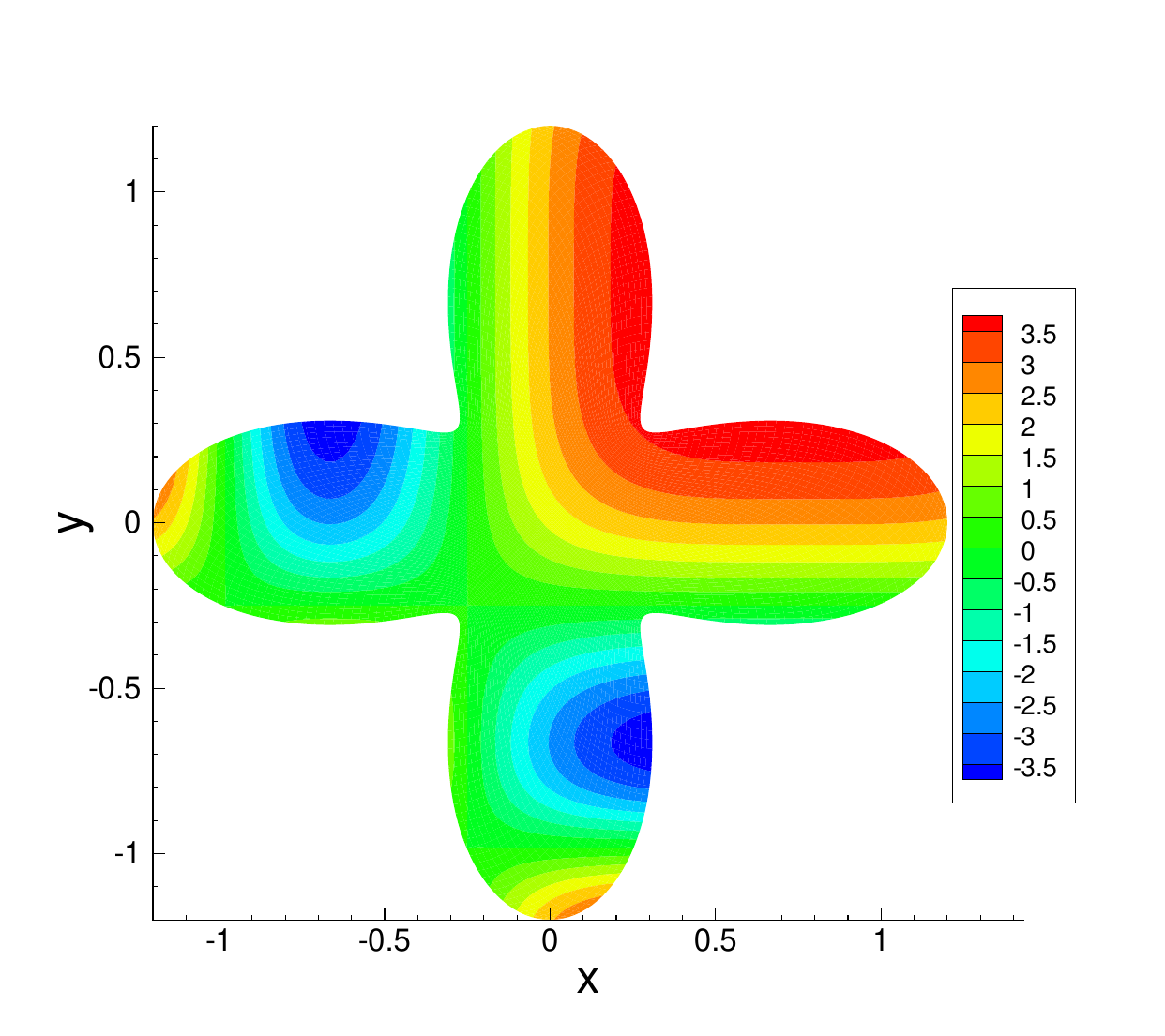}(c)
    \includegraphics[width=1.4in]{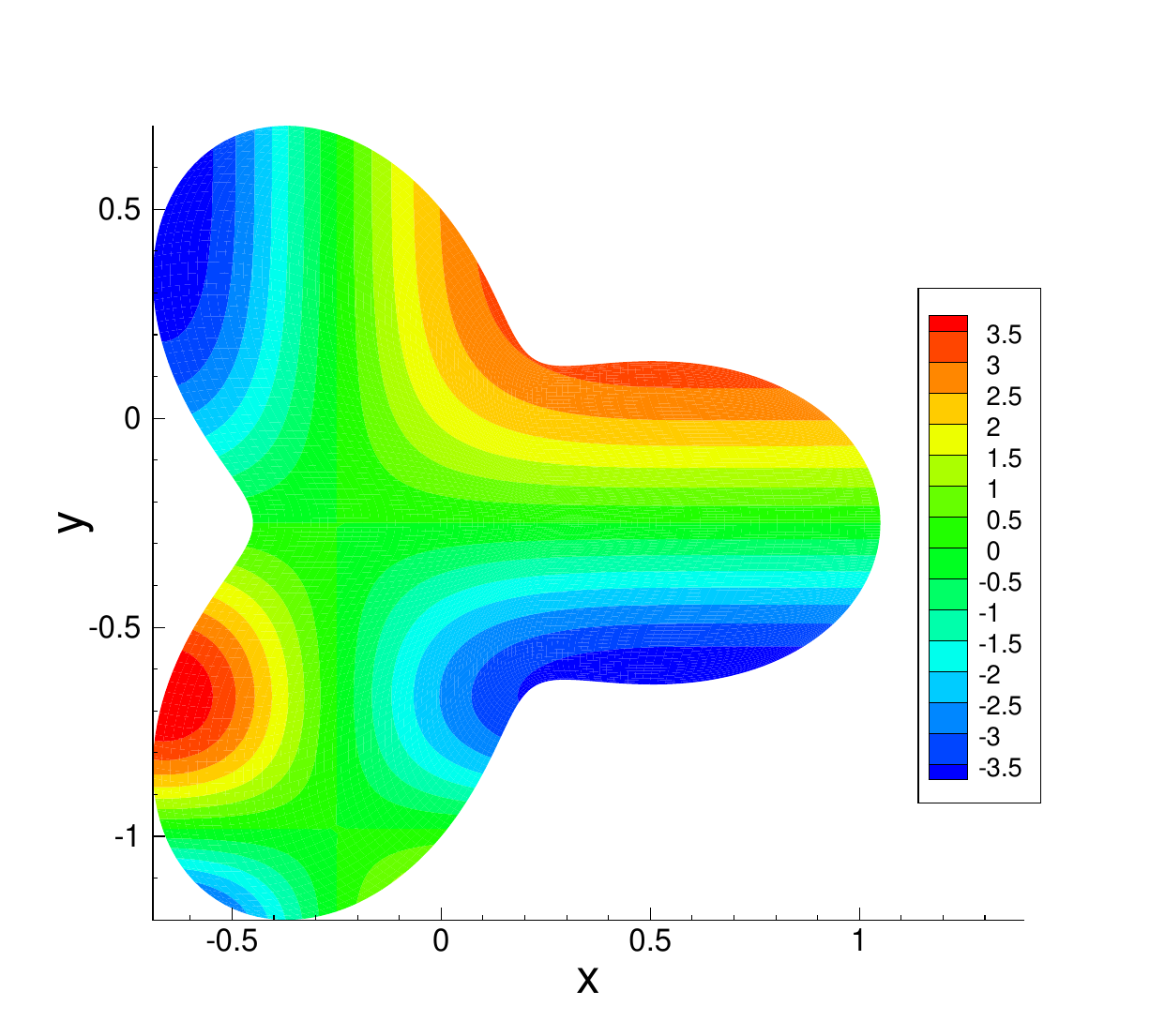}(d)
  }
  \centerline{
    \includegraphics[width=1.4in]{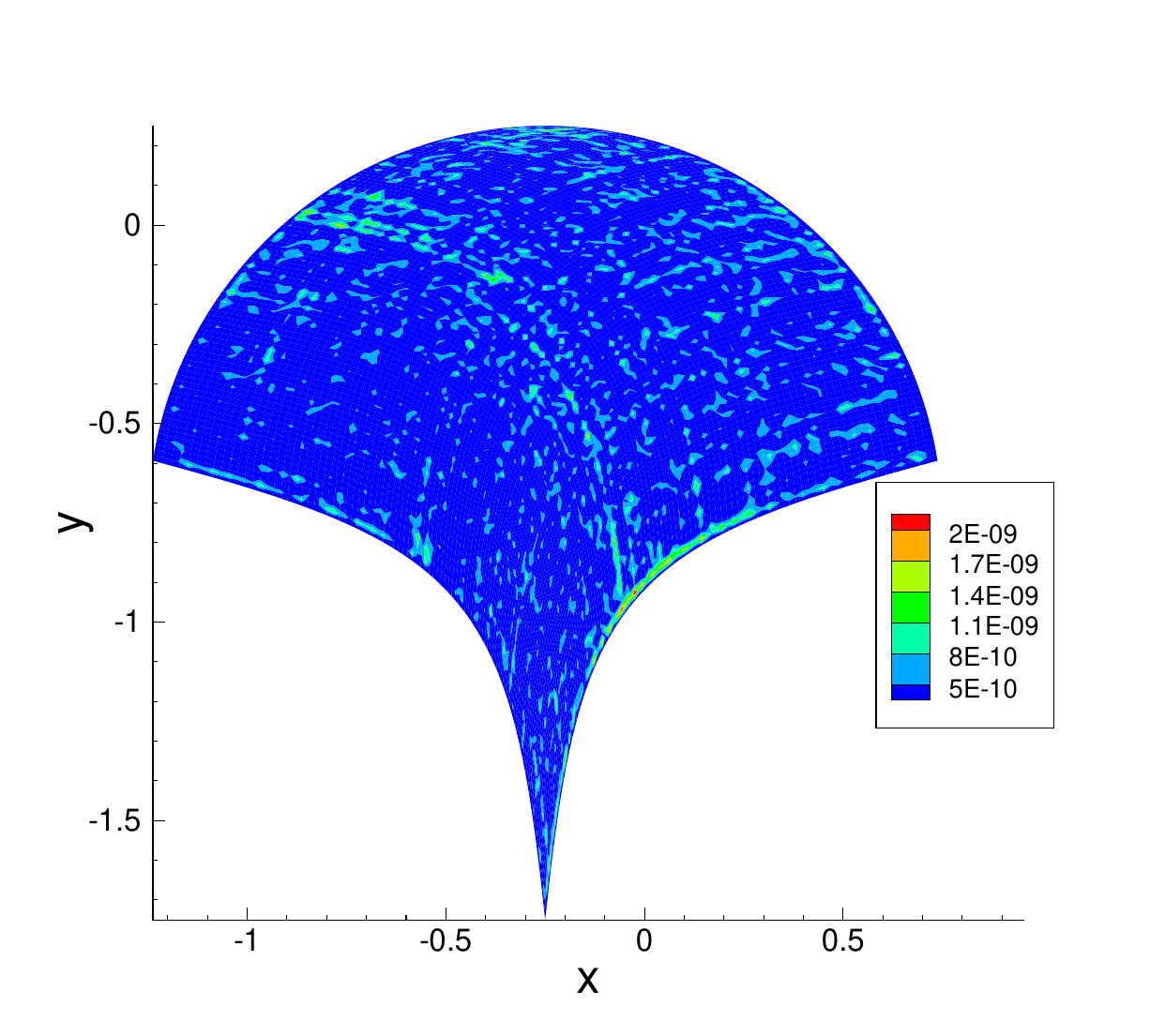}(e)
    \includegraphics[width=1.4in]{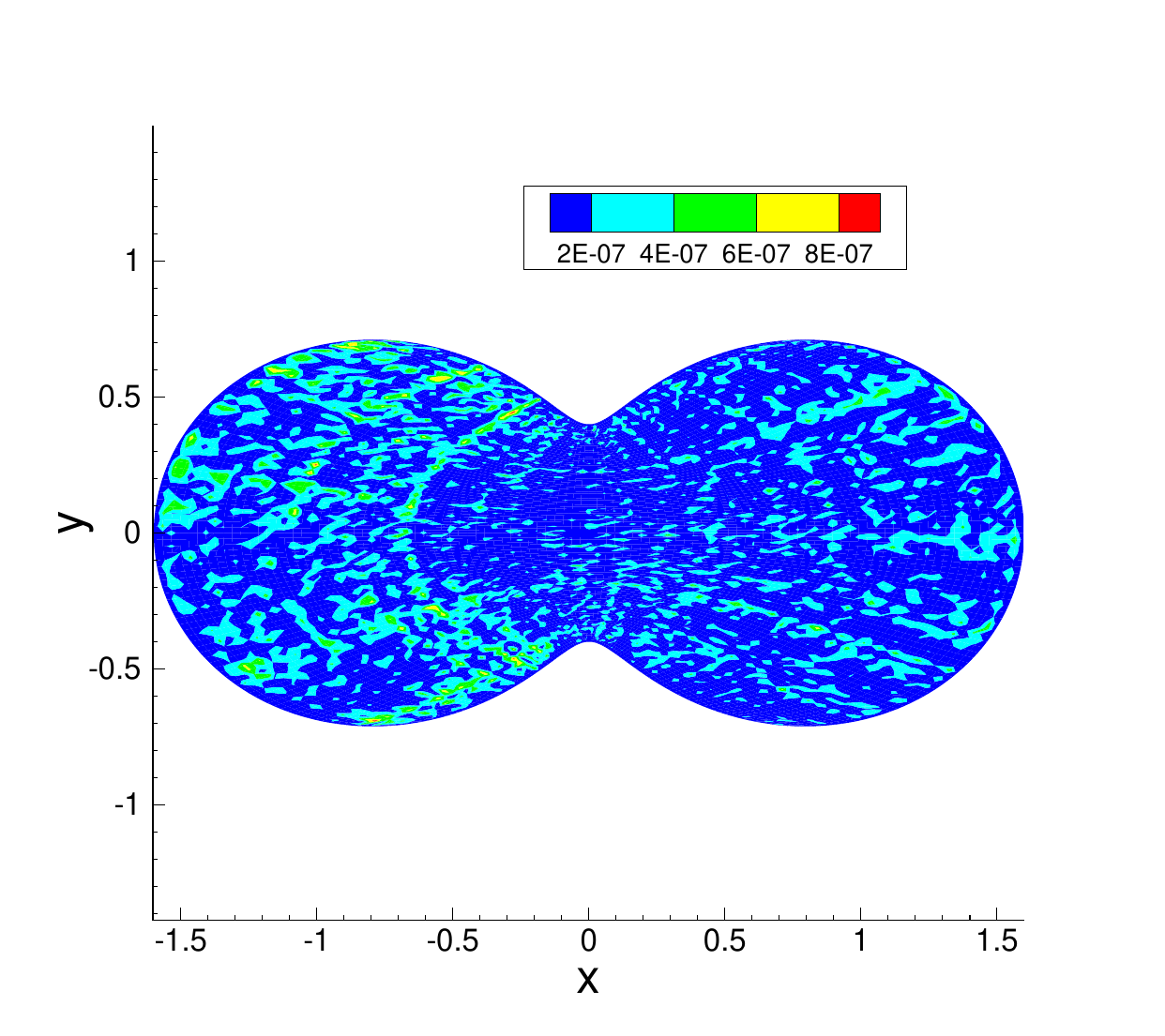}(f)
    \includegraphics[width=1.4in]{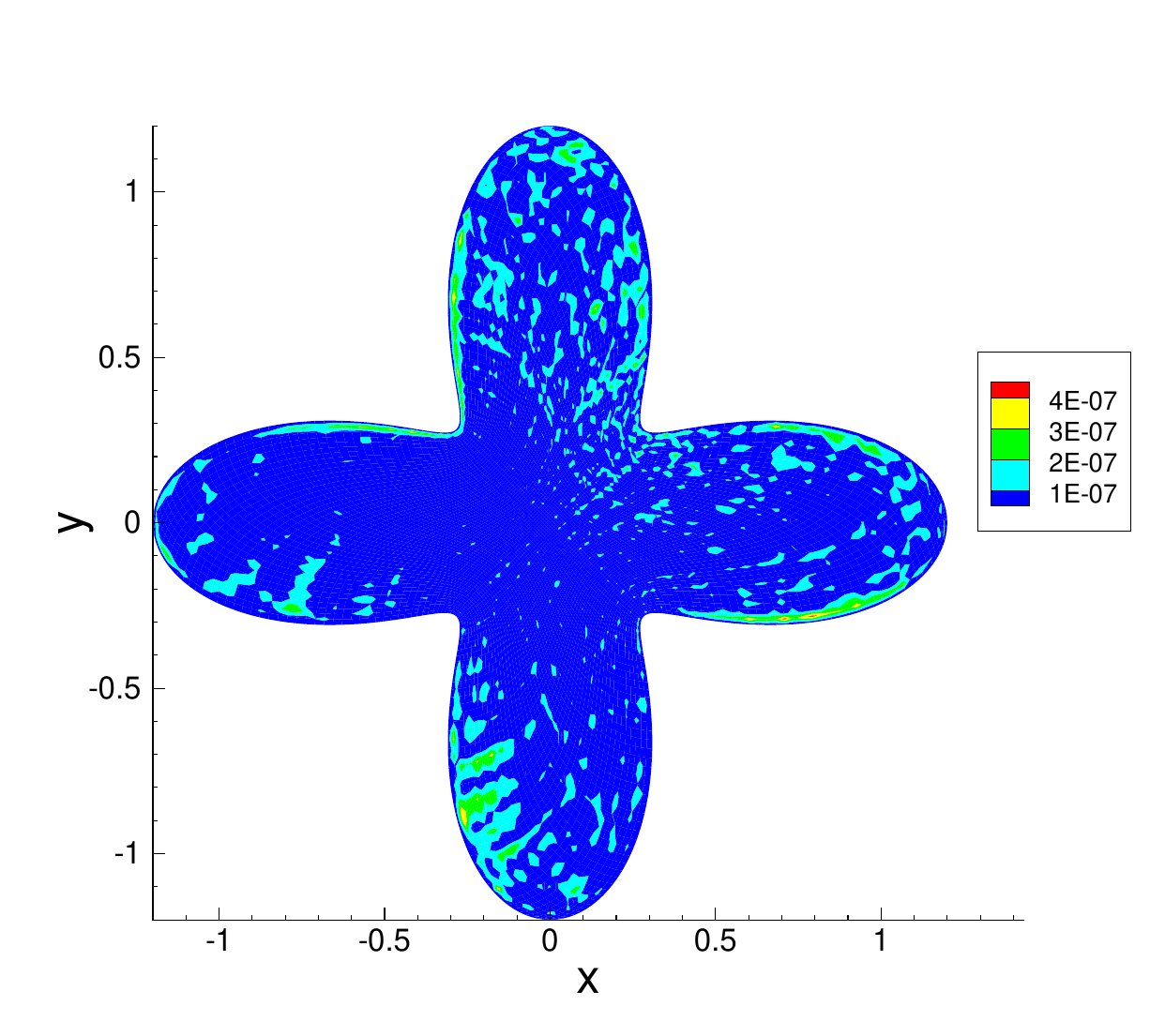}(g)
    \includegraphics[width=1.4in]{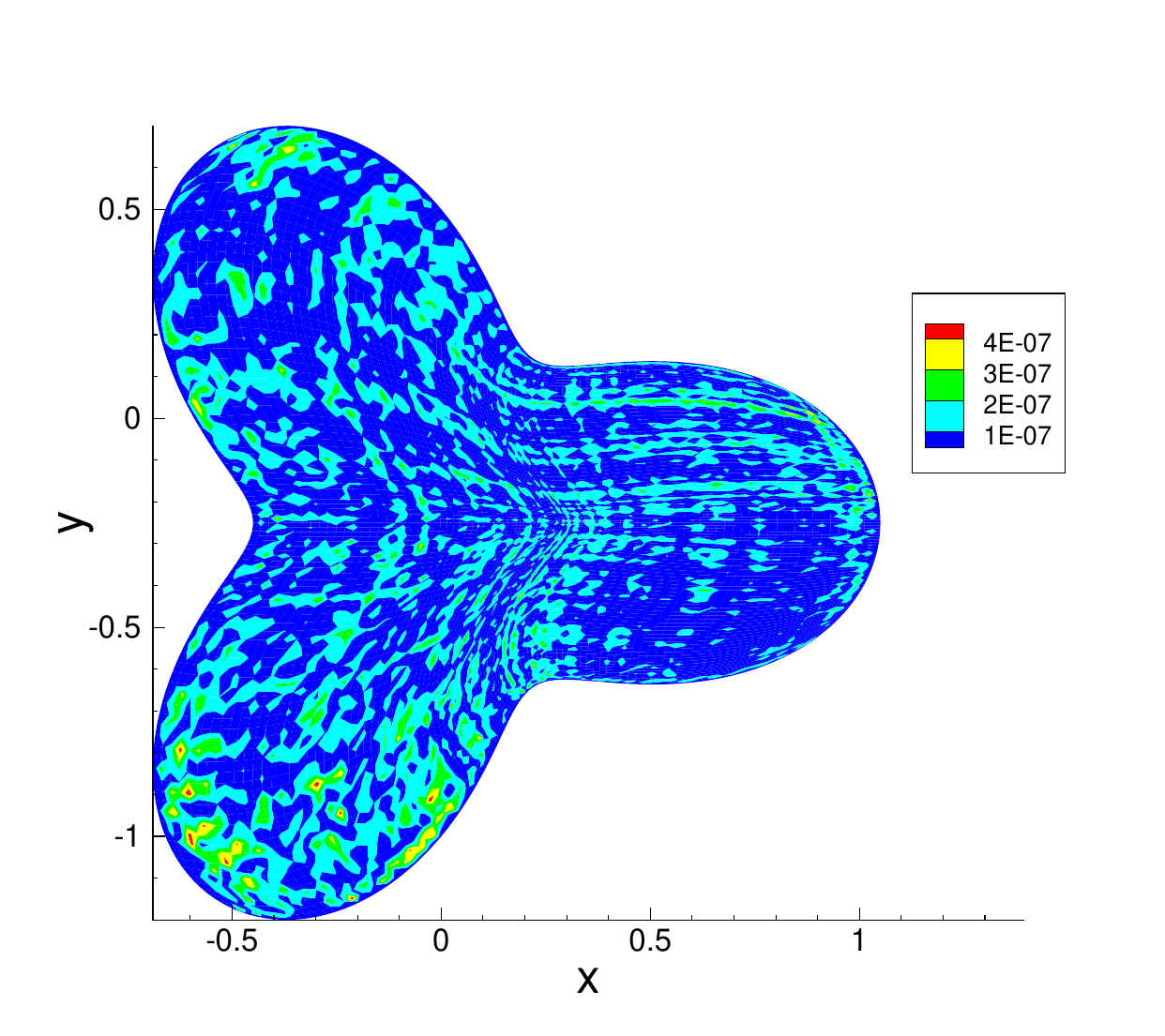}(h)
  }
  \caption{Nonlinear Helmholtz equation (Dirichlet BC on all boundaries):
    Distributions of the NN solutions (top row)
    and their point-wise absolute errors (bottom row) on the four domains.
    Simulation parameters: (a,e) $R_m=4.0$, $Q=50$, $M=1000$.
    (b,f) $R_m=4.5$, $Q=55$, $M=1000$. (c,g,d,h) $R_m=5.0$, $Q=60$, $M=1000$.
    $Q_{db}=8$ (see Remark~\ref{rem_210}) for all domains.
  }
  \label{fg_o3}
\end{figure}

\begin{table}
  \centering
  \begin{tabular}{l|l|l|l|l}
    \hline
     & domain \#1 & domain \#2 & domain \#3 & domain \#4 \\ \hline
    max-error (domain) & $2.544E-9$ & $8.948E-7$ & $4.475E-7$ & $4.808E-7$  \\[3pt]
    rms-error (domain) & $4.359E-10$ & $2.099E-7$ & $7.637E-8$ & $1.118E-7$ \\[3pt] \hline
    max DBC-error ($\overline{AB}$) & $0.0$ & $0.0$ & $0.0$ & $0.0$ \\[3pt]
    rms DBC-error ($\overline{AB}$) & $0.0$ & $0.0$ & $0.0$ & $0.0$ \\[3pt] \hline
    max DBC-error ($\overline{BC}$) & $8.882E-16$ & $4.441E-16$ & $5.551E-16$ & $4.441E-16$ \\[3pt]
    rms DBC-error ($\overline{BC}$) & $2.303E-16$ & $1.662E-16$ & $2.075E-16$ & $1.615E-16$ \\[3pt] \hline
    max DBC-error ($\overline{CD}$) & $0.0$ & $0.0$ & $0.0$ & $0.0$ \\[3pt]
    rms DBC-error ($\overline{CD}$) & $0.0$ & $0.0$ & $0.0$ & $0.0$ \\[3pt] \hline
    max DBC-error ($\overline{AD}$) & $4.441E-16$ & $4.441E-16$ & $4.996E-16$ & $4.441E-16$ \\[3pt]
    rms DBC-error ($\overline{AD}$) & $2.169E-16$ & $1.171E-16$ & $1.995E-16$ & $2.014E-16$ \\[3pt]
    \hline
  \end{tabular}
  \caption{Nonlinear Helmholtz equation (Dirichlet BC on all boundaries):
    maximum and rms NN-solution errors over
    the domain, and the maximum and rms DBC errors on
    the four boundaries. Simulation parameters follow those of Figure~\ref{fg_o3}.
  }
  \label{tab_5}
\end{table}

The ELM simulation results with Dirichlet conditions on all domain
boundaries are illustrated in Figure~\ref{fg_o3} and Table~\ref{tab_5} for different domains.
Figure~\ref{fg_o3} shows distributions of the NN solutions (top row) and
their point-wise absolute errors (bottom row) for these four domains.
These results are obtained using the simulation parameter values as given
in the figure caption. The results indicate that the current method has captured the solution
accurately on these domain geometries, with the maximum ELM
error on the order of $10^{-7}$ or $10^{-9}$ for different domains.

Table \ref{tab_5} is an assessment of the boundary-condition errors on different
boundaries, as well as the solution errors, for these five domains.
The DBC error is exactly zero on some boundaries ($\overline{AB}$ and $\overline{CD}$),
and has a maximum on the order of $10^{-16}$ on the other boundaries
($\overline{BC}$ and $\overline{AD}$).
Our method has evidently enforced the boundary condition to the machine
accuracy on these complex boundaries.


\begin{figure}
  \centerline{
    \includegraphics[width=1.4in]{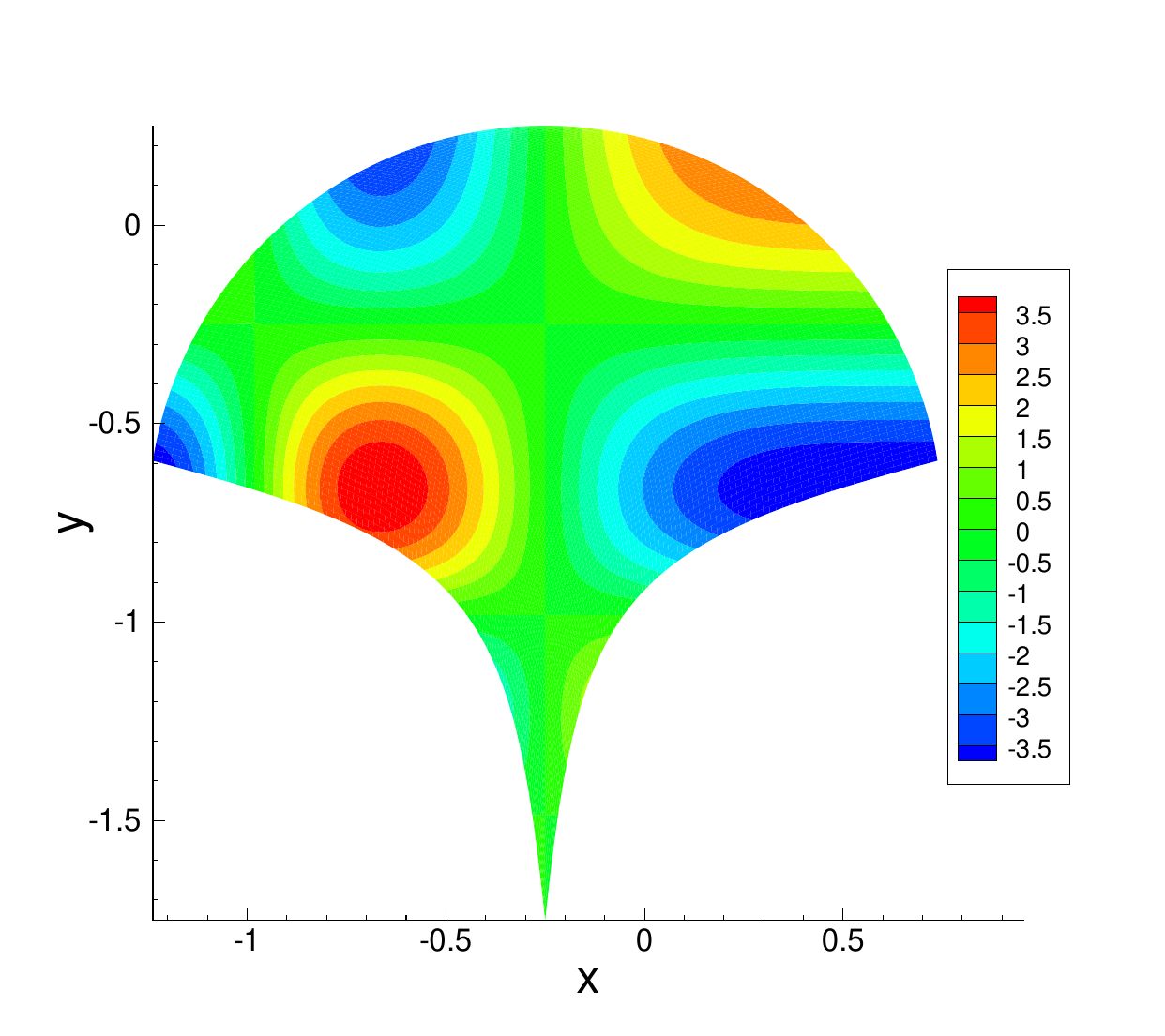}(a)
    \includegraphics[width=1.4in]{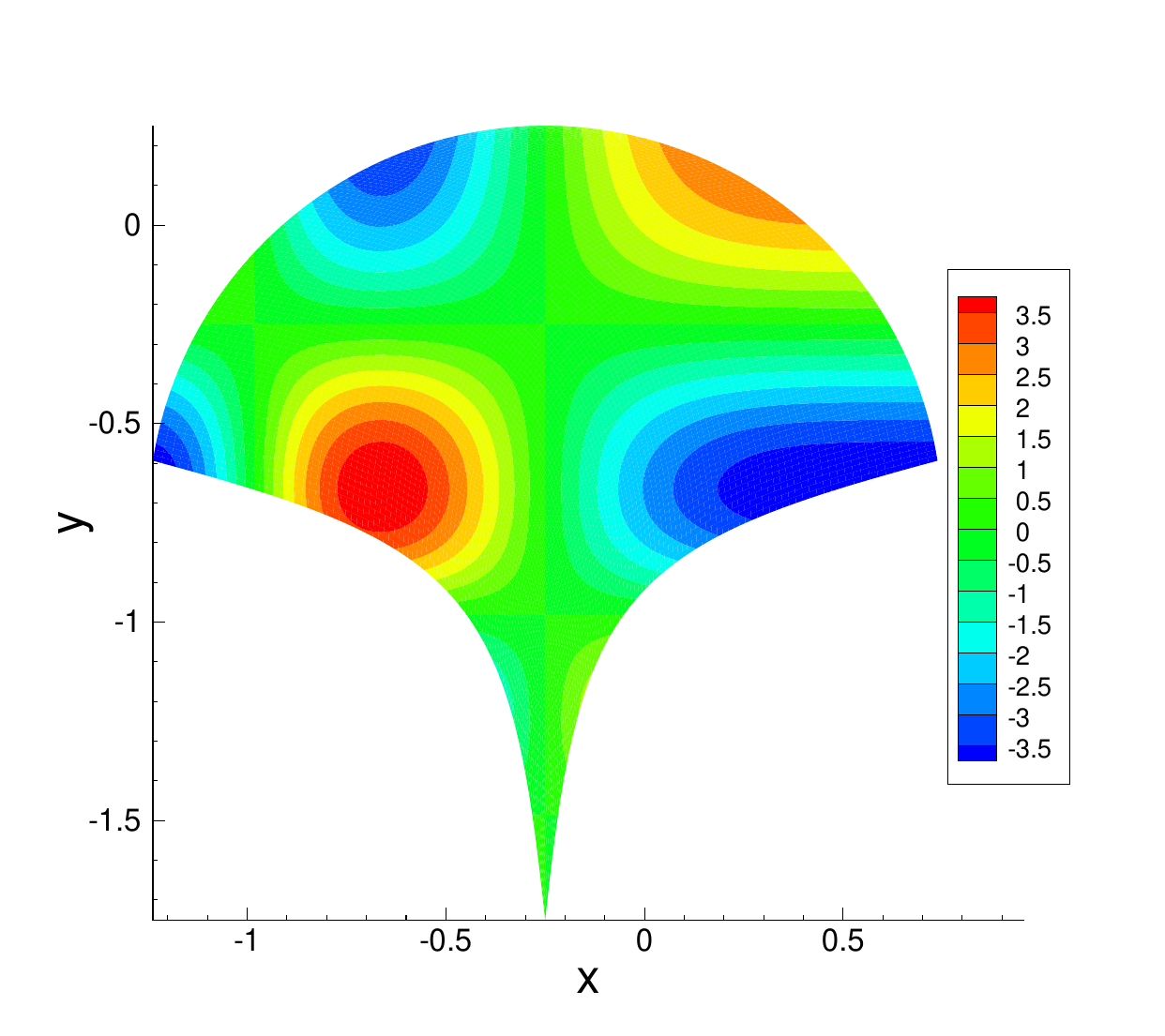}(b)
    \includegraphics[width=1.4in]{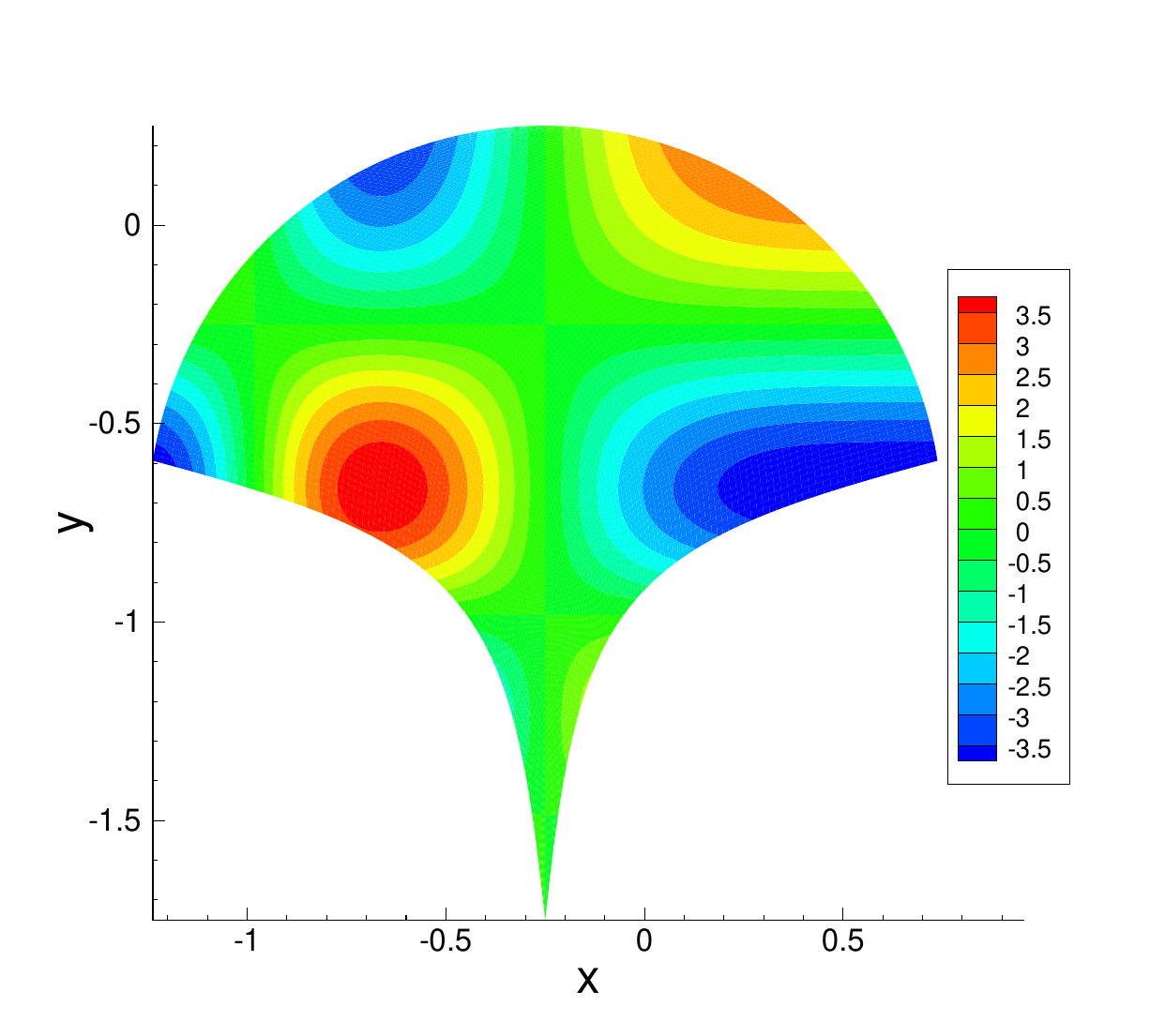}(c)
  }
  \centerline{
    \includegraphics[width=1.4in]{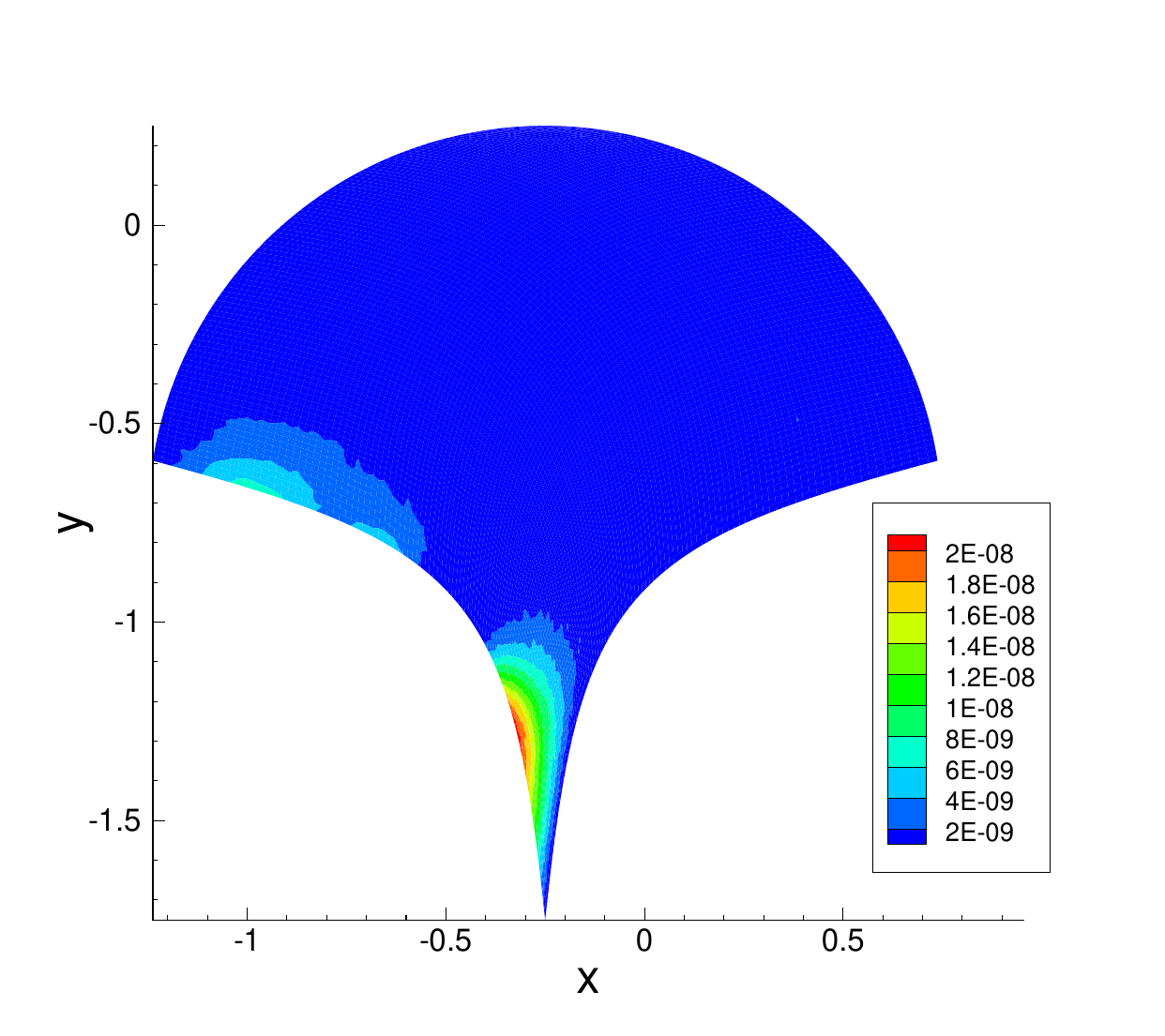}(d)
    \includegraphics[width=1.4in]{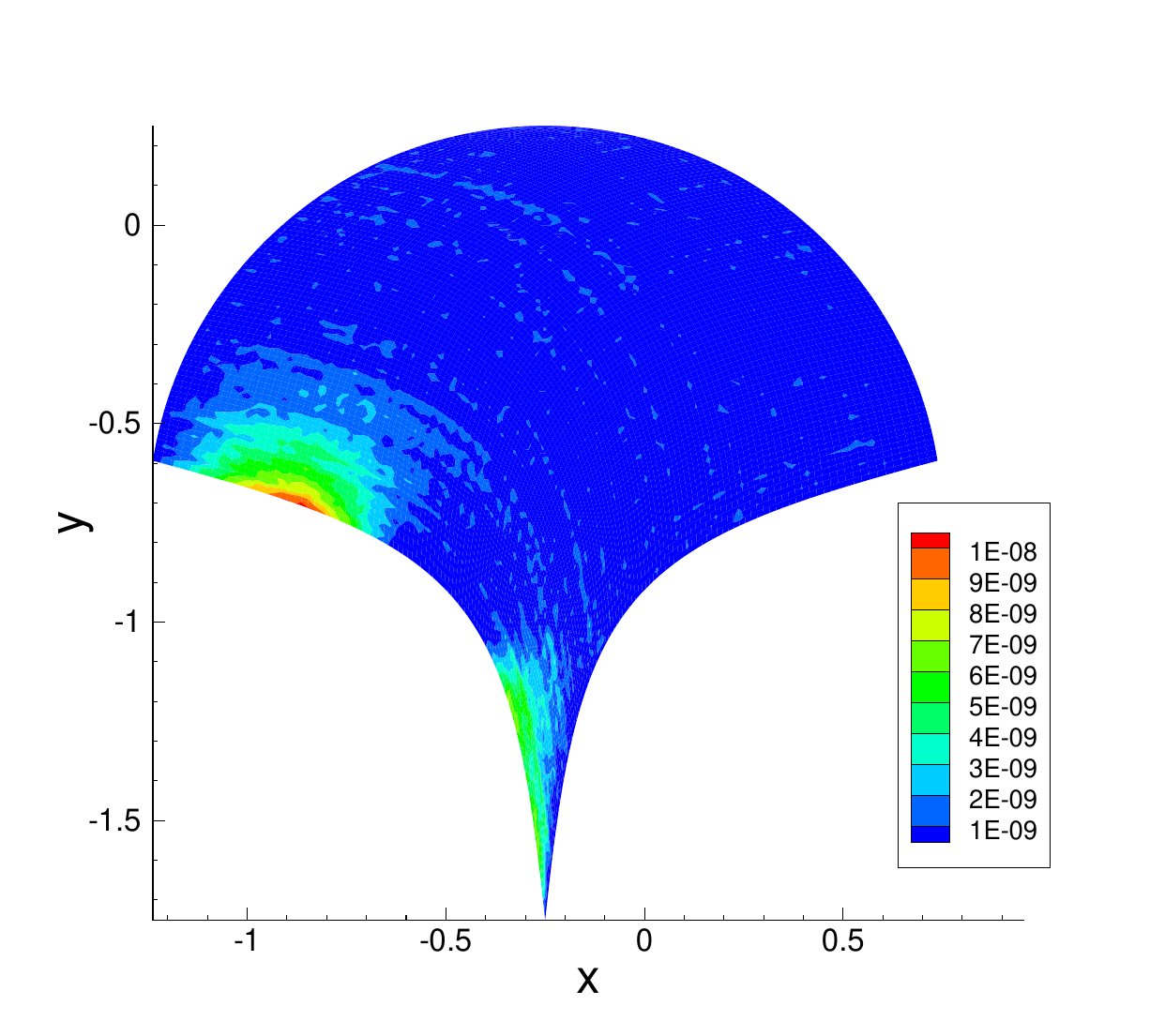}(e)
    \includegraphics[width=1.4in]{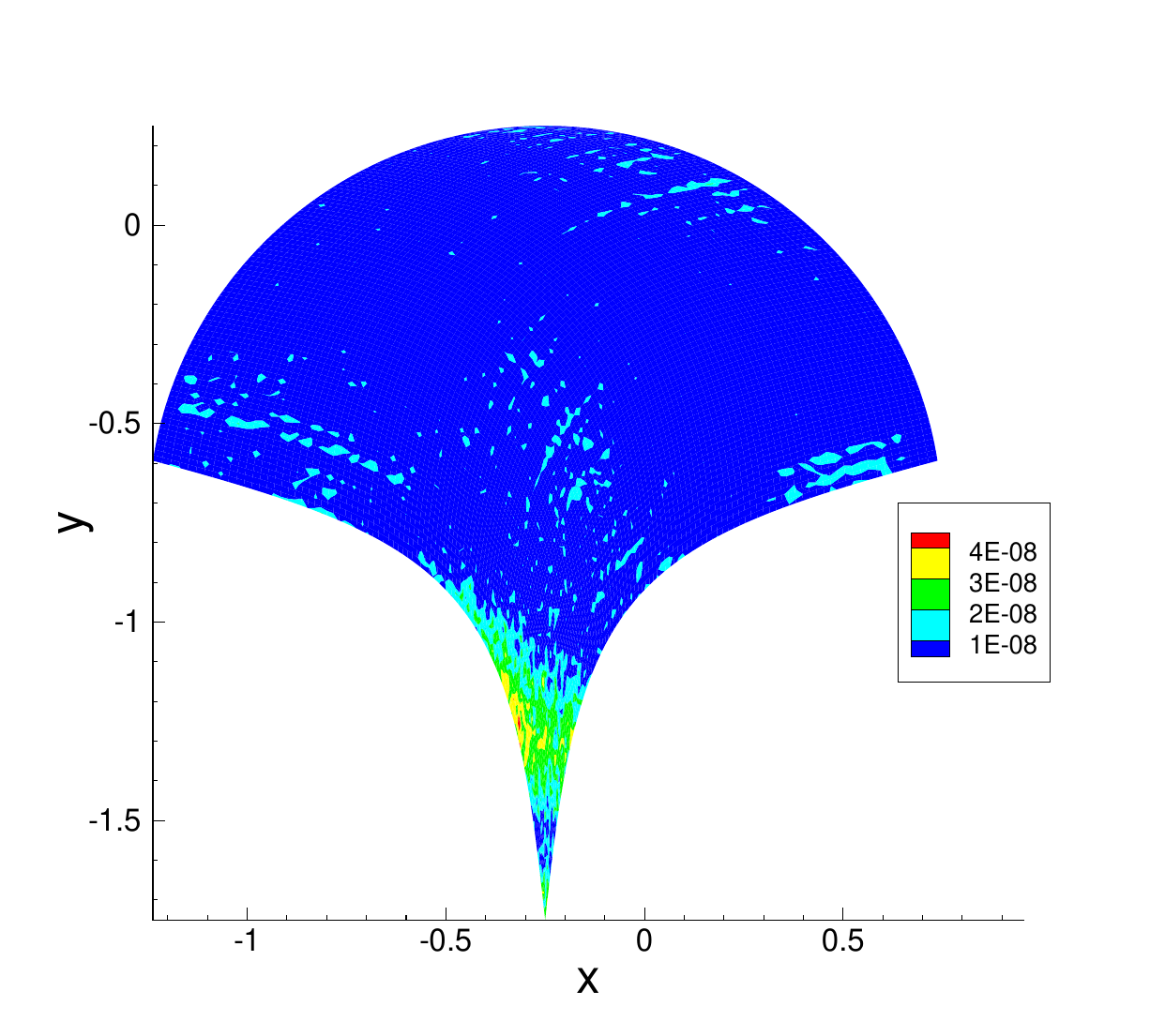}(f)
  }
  \caption{Nonlinear Helmholtz equation (Neumann or Robin BCs):
    Distributions of the NN solutions (top row) and their point-wise errors (bottom row)
    on domain \#1.
    (a,d) Case \#1: Neumann condition on $\overline{BC}$ and Dirichlet condition
    on the other boundaries.
    (b,e) Case \#2: Robin condition on $\overline{BC}$ and Dirichlet condition on the
    other boundaries.
    (c,f) Case \#3: Neumann condition on $\overline{BC}$ and $\overline{CD}$, and Dirichlet condition
    on the other boundaries.
    Simulation parameters: (a,d) $R_m=4.0$, $Q=60$, $M=1000$. (b,e) $R_m=3.5$, $Q=60$, $M=1000$.
    (c,f) $R_m=2.5$, $Q=60$, $M=1000$.
  }
  \label{fg_o4}
\end{figure}

\begin{table}
  \centering
  \begin{tabular}{l|c|c|c}
    \hline
     & Case \#1 & Case \#2 &  Case \#3 \\ \hline
    max solution-error (domain) & $2.133E-8$ & $1.045E-8$ & $4.709E-8$  \\[3pt]
    rms solution-error (domain) & $2.498E-9$ & $1.367E-9$ & $7.885E-9$  \\[3pt] \hline
    max DBC-error ($\overline{AB}$) & $0.0$ & $1.388E-16$ & $0.0$  \\[3pt]
    rms DBC-error ($\overline{AB}$) & $0.0$ & $6.418E-17$ & $0.0$  \\[3pt] \hline
    max NBC- or RBC-error ($\overline{BC}$) & $5.329E-15$ & $6.217E-15$ & $7.105E-15$  \\[3pt]
    rms NBC- or RBC-error ($\overline{BC}$) & $1.540E-15$ & $2.198E-15$ & $1.994E-15$ \\[3pt] \hline
    max DBC- or NBC-error ($\overline{CD}$) & $0.0$ & $1.110E-16$ & $6.217E-15$ \\[3pt]
    rms DBC- or NBC-error ($\overline{CD}$) & $0.0$ & $4.893E-17$ & $1.702E-15$ \\[3pt] \hline
    max DBC-error ($\overline{AD}$) & $6.661E-16$ & $6.661E-16$ & $4.441E-16$  \\[3pt]
    rms DBC-error ($\overline{AD}$) & $2.534E-16$ & $2.534E-16$ & $1.524E-16$ \\[3pt]
    \hline
  \end{tabular}
  \caption{Nonlinear Helmholtz equation on domain \#1 (Neumann or Robin BCs):
    maximum and rms NN-solution errors over
    the domain, and the maximum and rms boundary-condition (DBC, NBC, RBC) errors on
    different boundaries.
    The three cases correspond to those in Figure~\ref{fg_o4} for different
    types of boundary conditions.
  }
  \label{tab_6}
\end{table}

In Figure~\ref{fg_o4} and Table~\ref{tab_6} we demonstrate the ELM simulation results
obtained with Neumann or Robin boundary conditions.
Figure~\ref{fg_o4} shows the ELM solution and its point-wise absolute error
on domain \#1 for three cases: (i) Neumann condition imposed on $\overline{BC}$
and Dirichlet conditions imposed on the rest of the boundaries (plots (a,d)), (ii)
Robin condition with $\alpha_{BC}=1$ imposed on $\overline{BC}$
and Dirichlet conditions imposed on the rest of the boundaries (plots (b,e)),
and (iii) Neumann conditions imposed on $\overline{BC}$ and $\overline{CD}$
and Dirichlet conditions imposed on the rest of the boundaries.
The simulation parameter values for each case are provided in the figure caption.
It is observed that the ELM solution is highly accurate,
with the maximum errors on the order of $10^{-8}$ for all cases.

Table~\ref{tab_6} demonstrates the accuracy of the current method for
enforcing different types of boundary conditions.
Here we list the boundary-condition errors ($\varepsilon_{max}$, $\varepsilon_{rms}$)
on different boundaries for the three cases in Figure~\ref{fg_o4}.
The maximum boundary-condition error is on the order of $10^{-15}$
or $10^{-16}$, and on some boundaries it is exactly zero.
These results demonstrate that our method has enforced the Dirichlet, Neumann,
and Robin conditions to the machine accuracy for this nonlinear problem.

\subsection{Heat Conduction on Moving/Deforming Domains}
\label{sec_33}

\begin{figure}
  \centerline{
    \includegraphics[width=1.in]{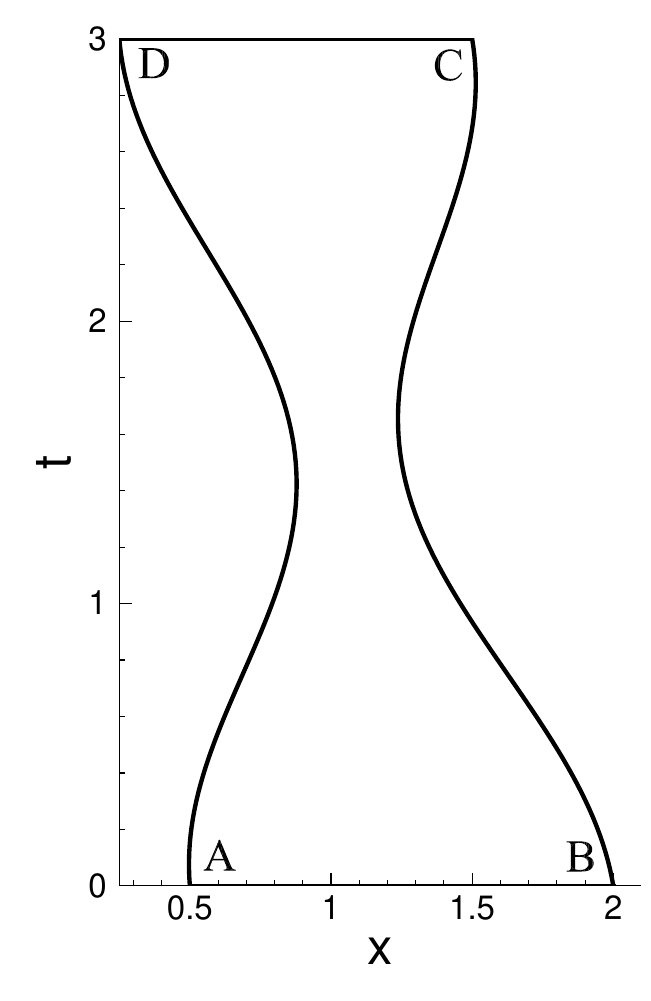}(a)\quad
    \includegraphics[width=1.in]{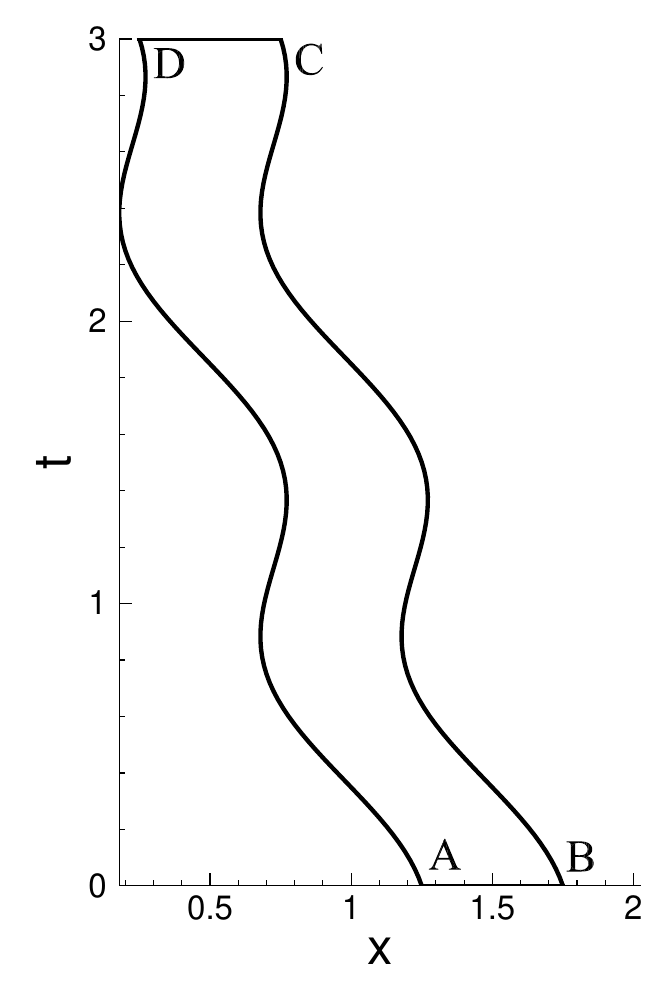}(b)\quad
    \includegraphics[width=1.in]{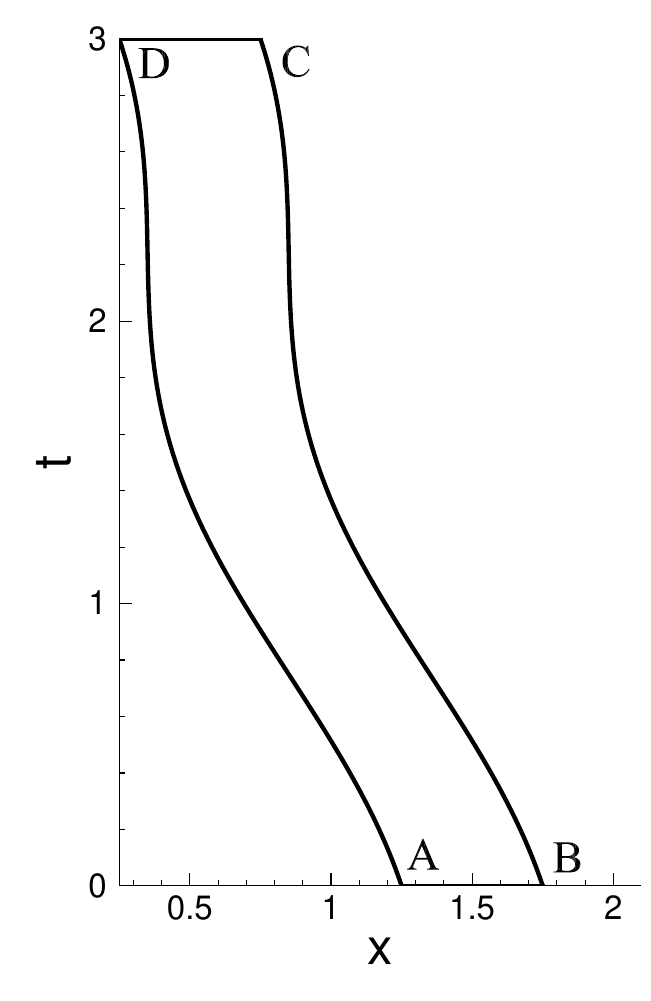}(c)
  }
  \centerline{
    \includegraphics[width=1.1in]{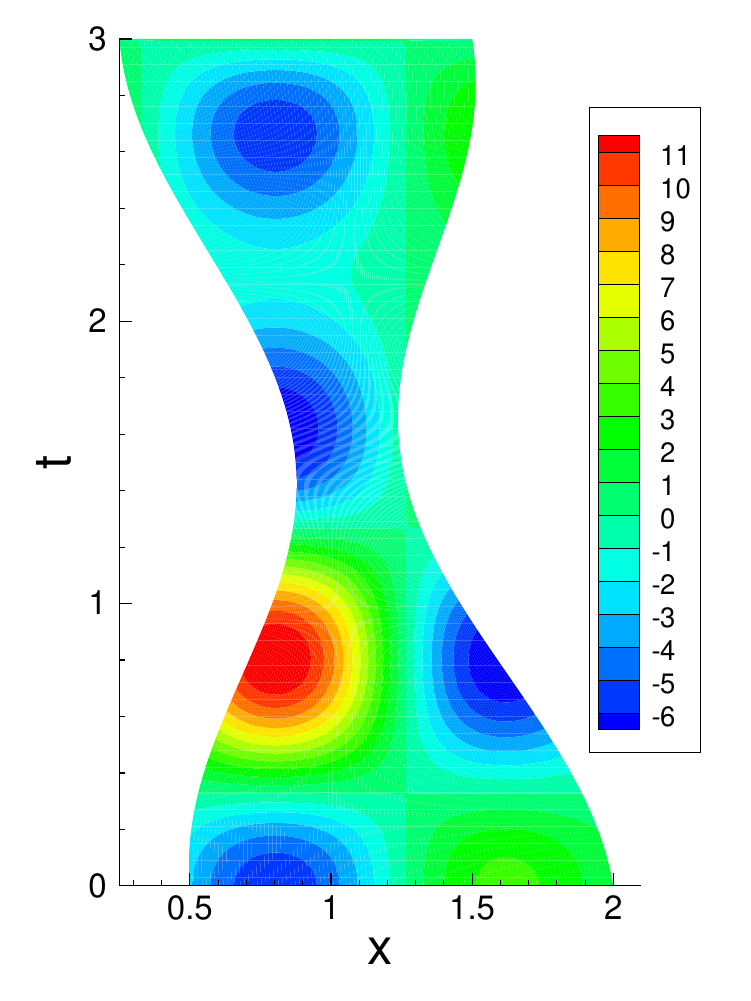}(d)
    \includegraphics[width=1.1in]{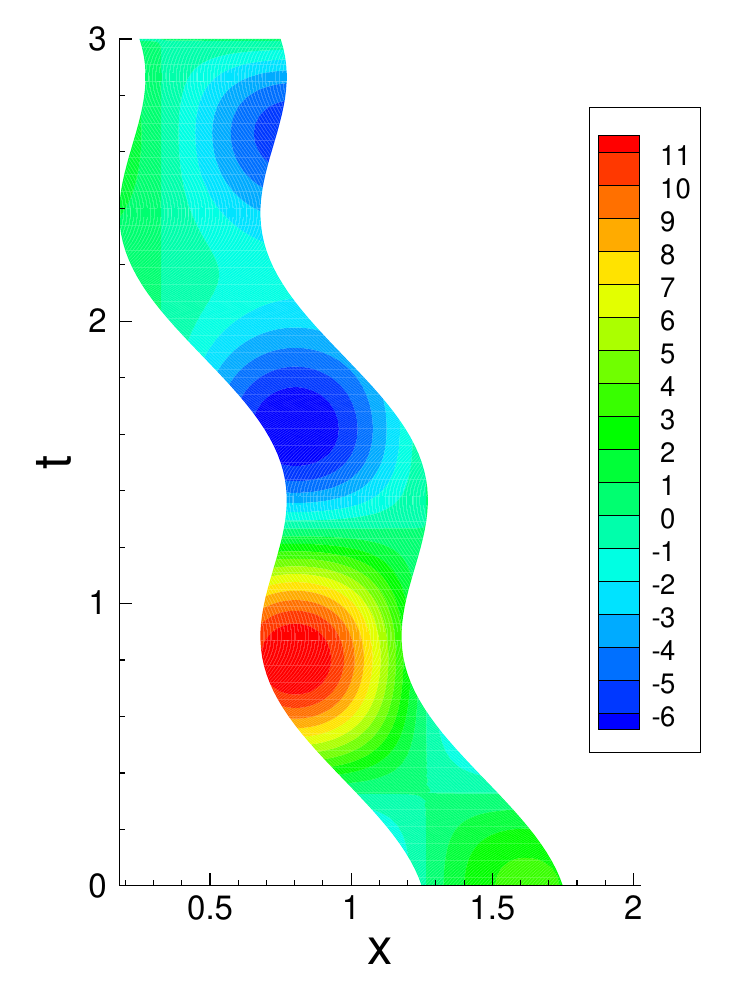}(e)
    \includegraphics[width=1.1in]{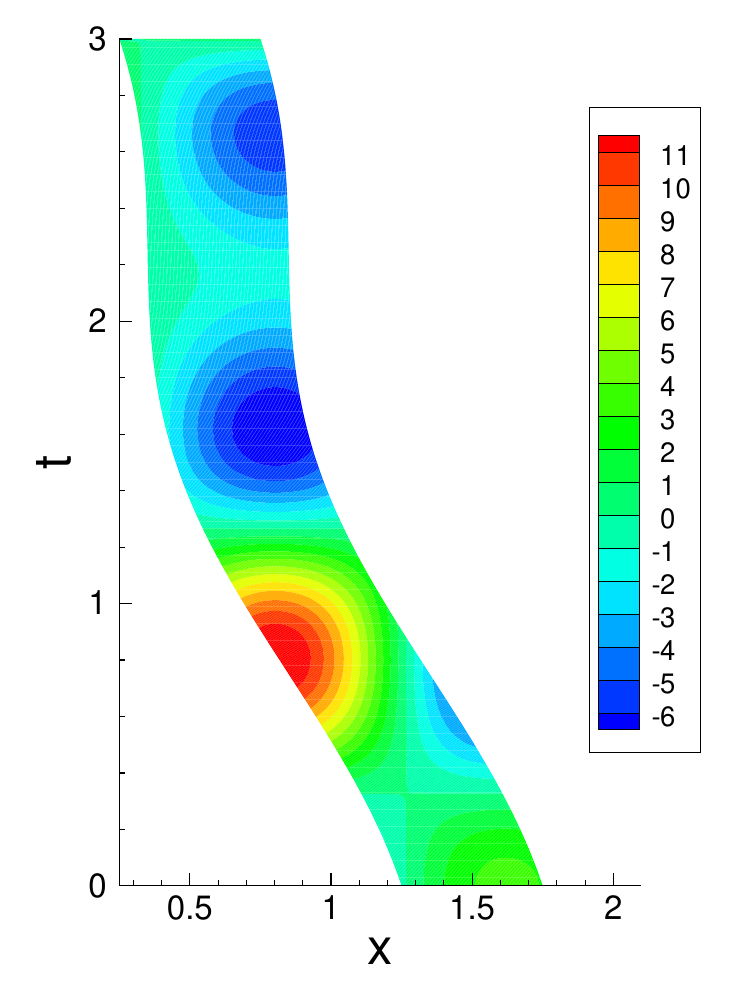}(f)
  }
  \caption{Heat equation on deforming/moving domains:
    space-time domain geometries (top row) and the exact
    solutions (bottom row). (a,d) Domain \#1 (deforming spatial domain);
    (b,e) Domain \#2 (moving spatial domain); (c,f) Domain \# (
    another moving spatial domain).
  }
  \label{fg_8}
\end{figure}

In the next example we investigate the heat conduction problem on
a spatial domain that deforms or moves over time.
Specifically, we consider a time-dependent domain in 1D,
$\Omega(t)=[a(t), b(t)]$, and the heat conduction equation on $\Omega(t)$,
\begin{subequations}\label{eq_114}
  \begin{align}
    & \frac{\partial u}{\partial t} - \nu \frac{\partial^2 u}{\partial x^2}
    = f(x,t), \quad x\in\Omega(t)=[a(t), b(t)], \quad t\in[0,t_f], \label{eq_114a} \\
    & u(a(t),t) = u_a(t), \quad t\in[0,t_f], \label{eq_114b} \\
    & u(b(t),t) = u_b(t), \quad t\in[0,t_f], \label{eq_114c} \\
    & u(x,0) = u_{in}(x), \quad x\in[a(0), b(0)] = [x_a, x_b]. \label{eq_114d}
  \end{align}
\end{subequations}
In the above equations, $u(x,t)$ is the field function to be computed,
$\nu=0.005$ is the diffusion coefficient (thermal diffusivity), $t_f$
denotes the time horizon of the problem, and
$f(x,t)$ is a source term. $u_a(t)$ and $u_b(t)$ are
prescribed boundary conditions, and $u_{in}(x)$ denotes the initial condition.
$[x_a,x_b]$ denotes the initial domain (at $t=0$).
It is assumed that the prescribed boundary and initial conditions
are compatible, namely, $u_{in}(x_a)=u_a(0)$ and $u_{in}(x_b)=u_b(0)$.
We choose the source term $f(x,t)$ and the boundary/initial conditions
appropriately such that this problem has the following exact solution,
\begin{align}\label{eq_115}
  u(x,t) =&\ \left[
    2\cos(0.75\pi x + 0.42\pi) + 1.5\cos(1.5\pi x - 0.22\pi)
    \right] \cdot \left[
    2\cos(0.75\pi y + 0.42\pi) \right. \notag \\
    &\ \left. + 1.5\cos(1.5\pi y - 0.22\pi)
    \right].
\end{align}

We consider the following three specific domains for this problem:
\begin{itemize}
\item Domain \#1 is defined by
  \begin{subequations}
    \begin{align}
      & a(t) = x_a(1-t/t_f) + x_d(t/t_f)
      +0.25\left[1-\cos(2\pi t/t_f) \right], \\
      & b(t) = x_b(1-t/t_f) + x_c(t/t_f)
      -0.25\left[1-\cos(2\pi t/t_f) \right],
    \end{align}
  \end{subequations}
  with the parameter values
  \begin{align}
    & x_a=0.5, \quad x_b=2.0, \quad x_c=1.5, \quad x_d=0.25, \quad t_f=3.0.
  \end{align}
  This is a deforming spatial domain, and is shown in Figure~\ref{fg_8}(a)
  as a space-time domain.

\item Domain \#2 is defined by
  \begin{subequations}
    \begin{align}
      & a(t) = x_a(1-t/t_f) + x_d(t/t_f)
      +0.15\left[ \cos(4\pi t/t_f) -1 \right], \\
      & b(t) = a(t) + (x_b-x_a),
    \end{align}
  \end{subequations}
  with the parameter values
  \begin{align}\label{eq_119}
    & x_a=1.25, \quad x_b=1.75, \quad x_d=0.25, \quad t_f=3.0.
  \end{align}
  This is a moving spatial domain, and is shown in Figure~\ref{fg_8}(b)
  as a space-time domain.

\item Domain \#3 is defined by
  \begin{subequations}
    \begin{align}
      & a(t) = x_a(1-t/t_f) + x_d(t/t_f)
      - 0.15\left[1 - \cos(2\pi t/t_f) \right], \\
      & b(t) = a(t) + (x_b-x_a),
    \end{align}
  \end{subequations}
  with the same parameter values as given in~\eqref{eq_119}.
  This is another moving spatial domain and is shown in Figure~\ref{fg_8}(c).
  
\end{itemize}
Distributions of the exact solution~\eqref{eq_115} over these domains
are included in Figure~\ref{fg_8} (bottom row).

\begin{figure}
  \centerline{
    \includegraphics[height=1.5in]{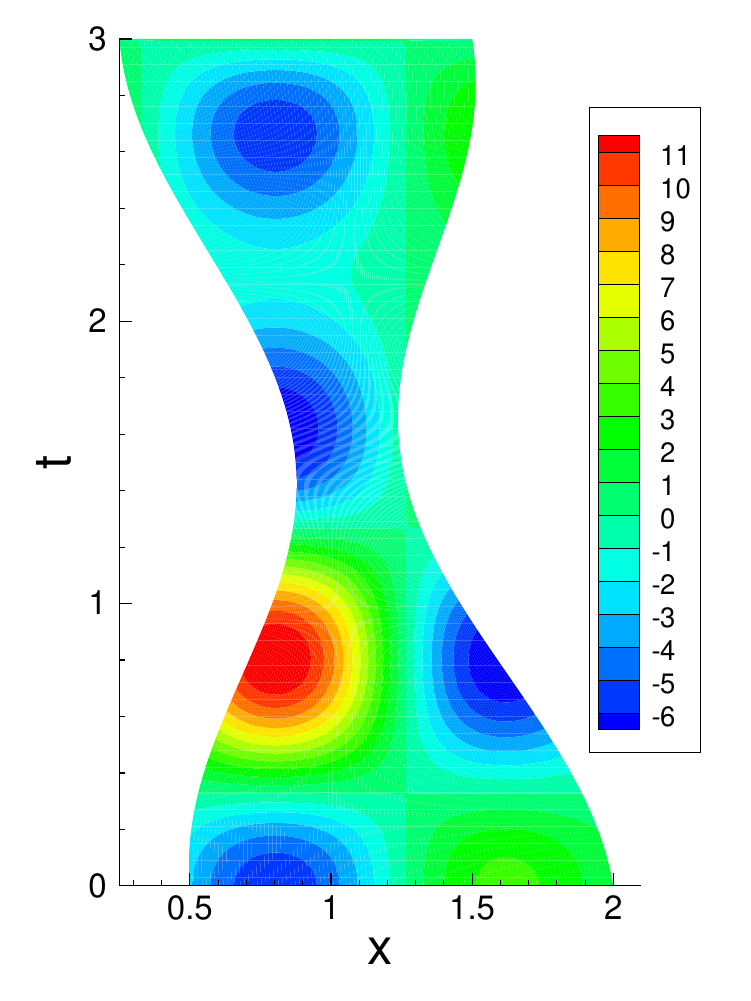}(a)
    \includegraphics[height=1.5in]{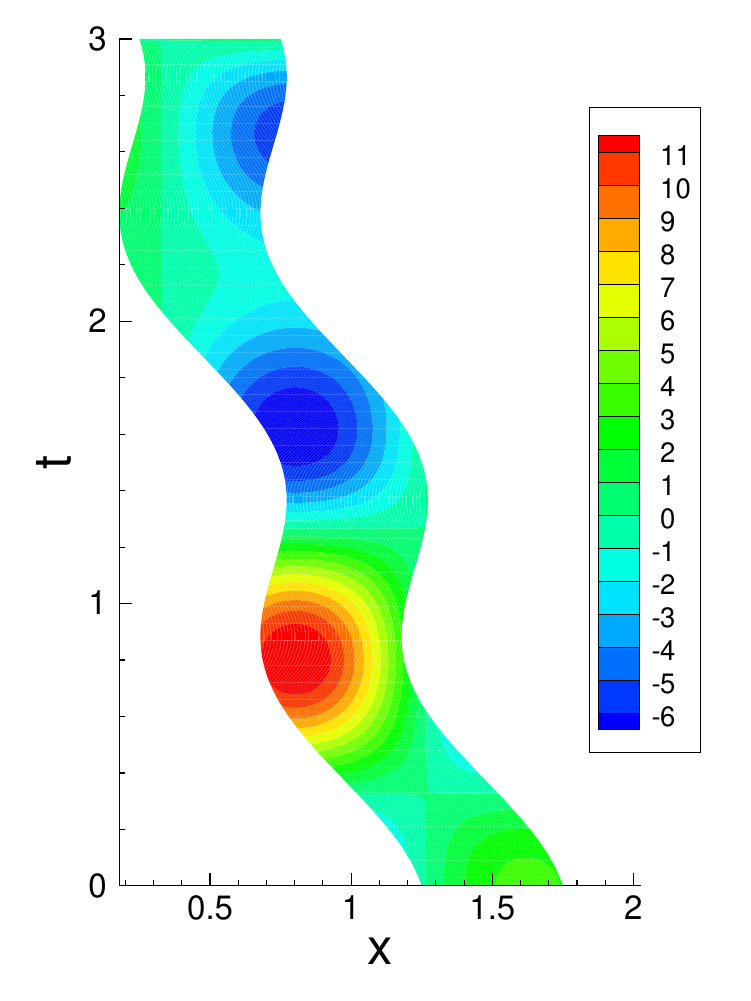}(b)
    \includegraphics[height=1.5in]{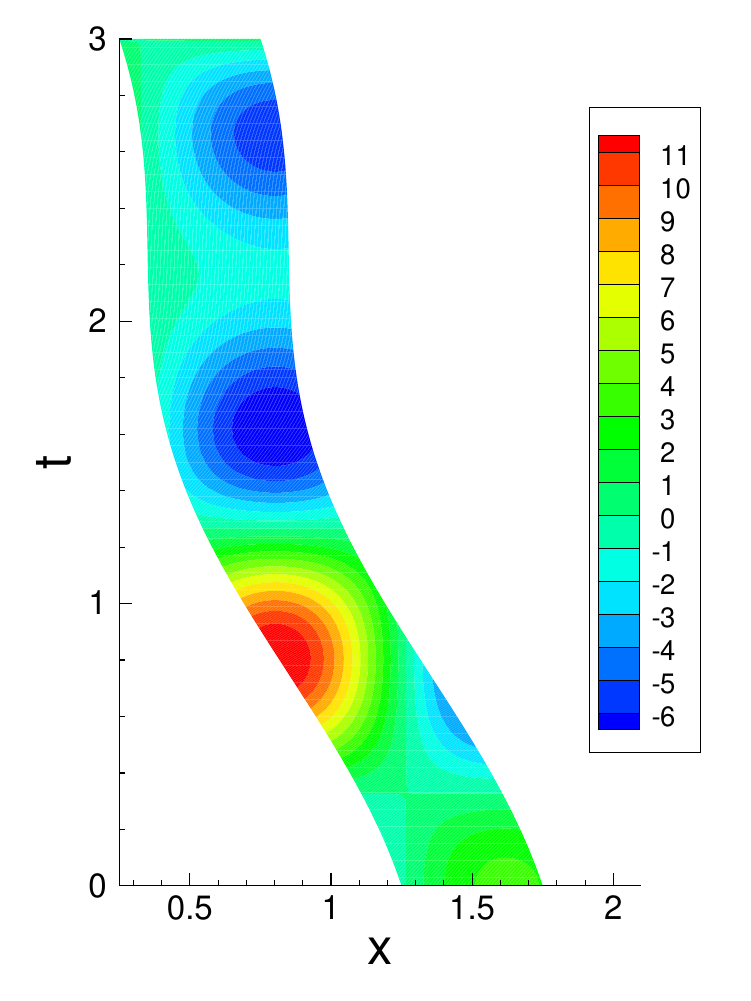}(c)
  }
  \centerline{
    \includegraphics[height=1.5in]{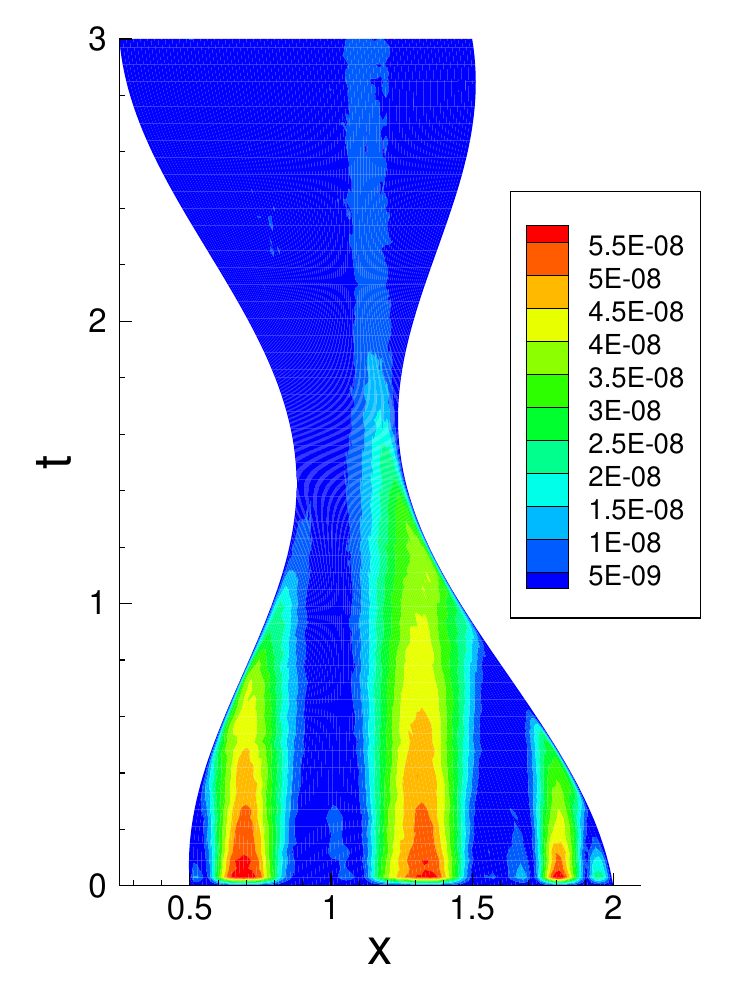}(d)
    \includegraphics[height=1.5in]{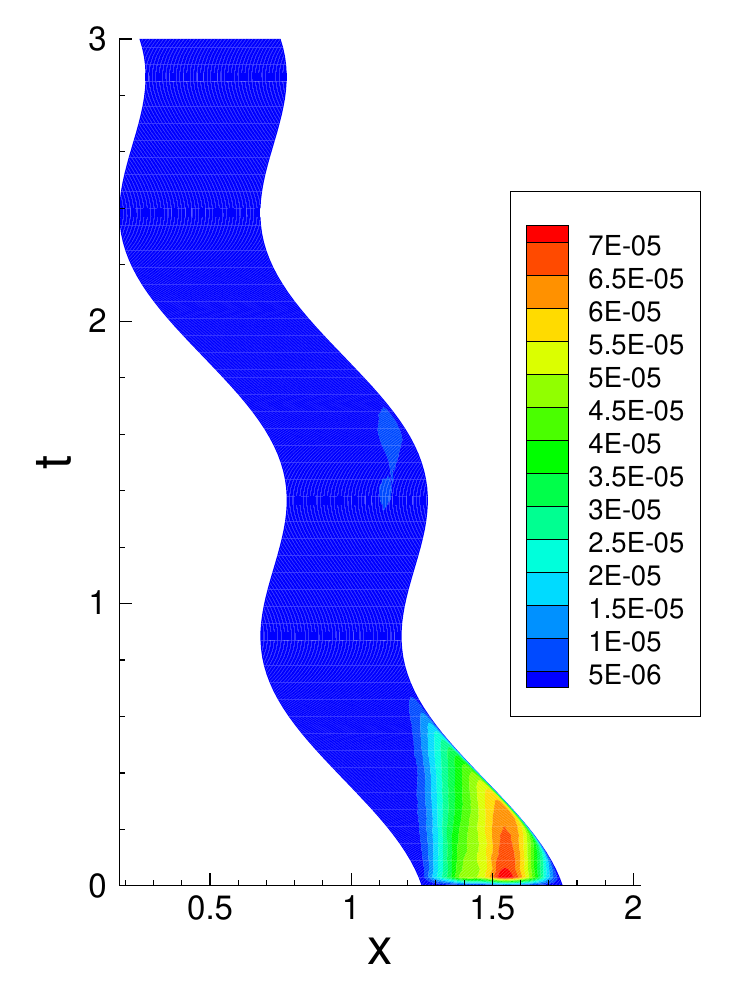}(e)
    \includegraphics[height=1.5in]{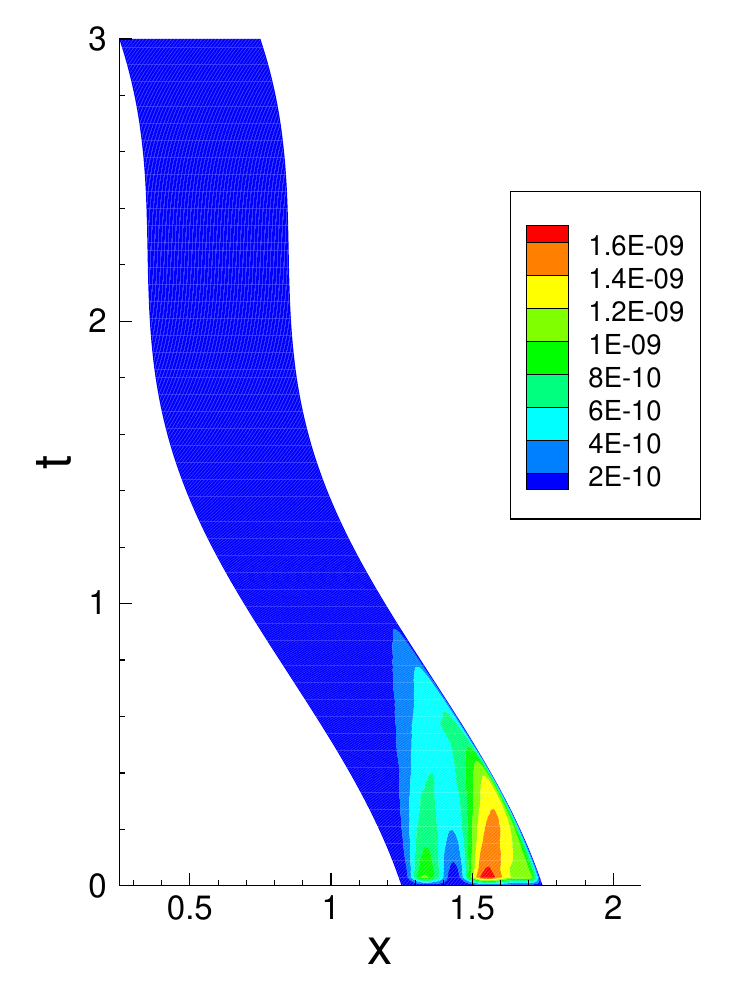}(f)
  }
  \caption{Heat equation on deforming/moving domains:
    Distributions of the NN solutions (top row),
    and their point-wise absolute errors (bottom row) for
    domains \#1 (a,d), \#2 (b,e) and \#3 (c,f).
    Simulation parameters: Domains \#1 and \#3,
    $R_m=4.0$, $Q=100$, $Q_{db}=7$, $M=1000$;
    Domains \#2,  $R_m=5.5$, $Q=100$, $Q_{db}=7$, $M=1000$.
  }
  \label{fg_9}
\end{figure}

\begin{table}
  \centering
  \begin{tabular}{l|l|l|l}
    \hline
     & domain \#1 & domain \#2 & domain \#3 \\ \hline
    max solution-error (domain) & $6.083E-8$ & $7.503E-5$ & $1.736E-9$  \\[3pt]
    rms solution-error (domain) & $1.712E-8$ & $1.401E-6$ & $3.208E-10$ \\[3pt] \hline
    max BC-error ($\overline{AB}$) & $0.0$ & $0.0$ & $0.0$ \\[3pt]
    rms BC-error ($\overline{AB}$) & $0.0$ & $0.0$ & $0.0$ \\[3pt] \hline
    max BC-error ($\overline{BC}$) & $8.882E-16$ & $8.882E-16$ & $8.882E-16$ \\[3pt]
    rms BC-error ($\overline{BC}$) & $1.983E-16$ & $3.436E-16$ & $3.743E-16$ \\[3pt] \hline
    max solution-error ($\overline{CD}$) & $6.537E-9$ & $3.875E-6$ & $3.662E-11$ \\[3pt]
    rms solution-error ($\overline{CD}$) & $2.765E-9$ & $1.960E-6$ & $1.711E-11$ \\[3pt] \hline
    max BC-error ($\overline{AD}$) & $1.776E-15$ & $2.220E-16$ & $4.441E-16$ \\[3pt]
    rms BC-error ($\overline{AD}$) & $4.842E-16$ & $7.842E-17$ & $5.925E-17$ \\[3pt]
    \hline
  \end{tabular}
  \caption{Heat equation on deforming/moving domains:
    the maximum and rms NN solution errors over
    the space-time domain and on the boundary $\overline{CD}$,
    and the maximum and rms boundary/initial
    condition errors on $\overline{AB}$, $\overline{BC}$, and $\overline{AD}$.
    Note that Dirichlet conditions are imposed on $\overline{AB}$, $\overline{BC}$
    and $\overline{AD}$ of the space-time domain, and that no boundary
    condition is imposed on $\overline{CD}$.
    Simulation parameters follow those of Figure~\ref{fg_9}.
  }
  \label{tab_7}
\end{table}

We solve this problem by a space-time approach and treat the time
variable $t$ on the same footing as the space variable $x$.
The boundary conditions~\eqref{eq_114b}--\eqref{eq_114c} and the initial
condition~\eqref{eq_114d} all become Dirichlet type conditions imposed
on the boundaries $\overline{AB}$, $\overline{BC}$
and $\overline{AD}$ of the space-time domain
$D=\overline{ABCD}=[a(t),b(t)]\times [0,t_f]$ (see Figure~\ref{fg_8}).
No condition is imposed on the boundary $\overline{CD}$.
These boundary conditions are enforced exactly using the method
from Section~\ref{sec_22}, with a modification by removing the Dirichlet
condition on $\overline{CD}$ from the formulation therein.

Figure~\ref{fg_9} shows distributions of the ELM solutions (top row)
and their point-wise absolute errors (bottom row) in the space-time plane
for these three domains. The simulation parameter values are
provided in the figure caption. The NN solutions on domains \#1 and \#3
are more accurate, with the maximum errors on the order of $10^{-8}$
and $10^{-9}$, respectively, compared with that on domain \#2,
with the maximum error around $10^{-5}$.

Table~\ref{tab_7} demonstrates the accuracy of the current method for
enforcing the boundary conditions with this problem. Here we list
the maximum/rms boundary-condition errors on $\overline{AB}$, $\overline{BC}$
and $\overline{AD}$ of the space-time domain, together with the
NN solution errors on $\overline{CD}$ and over the entire domain.
It is evident that our method has enforced the boundary conditions
to the machine accuracy for this problem with moving/deforming boundaries.



\section{Concluding Remarks}
\label{sec_summary}


We have developed a systematic method for
enforcing exactly the Dirichlet, Neumann, and Robin type boundary
conditions on general quadrilateral domains with arbitrary
curved boundaries. The method consists of two components,
an exact mapping of general quadrilateral domains
to the standard domain and a four-step procedure to systematically formulate the
general forms of trial functions that exactly satisfy the
imposed Dirichlet/Neumann/Robin  conditions,
both utilizing transfinite interpolations and
TFC constrained expressions in their construction. 

The constructed general forms of trial functions satisfying
the imposed boundary conditions are in parametric forms, expressed
with respect to the standard domain. When only
Dirichlet boundaries are involved, the formulation is conceptually straightforward
by leveraging the domain mapping.
When Neumann or Robin type boundaries are present, the
formulation becomes significantly more challenging.
We formulate the general forms
of trial functions for exact BC enforcement
through a procedure consisting of four steps:
(i) Identify the set of variables that
the transformed boundary conditions 
and the compatibility constraints induced by these conditions are imposed on;
(ii) Construct the transfinite interpolation for the types of identified variables;
(iii) Formulate the preliminary TFC constrained expression
based on this transfinite interpolation;
(iv) Update terms of the transfinite interpolation by corresponding
terms involving the free function from the preliminary TFC expression, thus
giving rise to the final TFC form as the constructed trial function.

When Neumann (or Robin) boundaries are present, employing the four-step procedure,
we have analyzed and presented in detail the formulation for two types of
situations: (i) when a Neumann (or Robin) boundary only intersects
with Dirichlet boundaries, and (ii) when two Neumann (or Robin) boundaries
intersect with each other. 
When the quadrilateral domain involves a combination of Dirichlet, Neumann,
and Robin boundaries and if multiple Neumann or Robin boundaries
are present, the formulation either falls into
or can be constructed based on the two aforementioned situations. 
In this case, at every vertex where two Neumann (or Robin) boundaries meet or
where a Neumann (or Robin) boundary and a Dirichlet boundary meet,
the compatibility constraints  and the corresponding constructions 
for handling such constraints analyzed herein will apply. The overall formulation can be
constructed analogously based on the four-step procedure.


The presented method for exact BC enforcement
has been implemented together with the extreme learning technique for physics-informed
 machine learning.
Extensive numerical experiments are conducted for several linear or nonlinear,
stationary or dynamic, boundary/initial value problems on a variety of
domains with complex geometries. 
The numerical results demonstrate that the current method has enforced
the Dirichlet, Neumann, and Robin conditions to machine accuracy
on curved domain boundaries. 


How to enforce Dirichlet/Neumann/Robin type conditions exactly on complex
domain geometries is a crucial issue to scientific machine learning.
The method developed in this work is but a preliminary
step toward achieving this goal.
It has been noted that the current formulations
are based on quadrilateral domains, and as such they inevitably
inherit many associated limitations.
Overcoming these limitations to advance the technique further
defines the goal for future research endeavors.




\section{Appendix: Geometric Domain  Parameters}
\label{sec_geom}

\noindent\underline{\bf Section~\ref{sec_31}: Helmholtz Equation}

\vspace{5pt}
\noindent\underline{Domain \#1:}\\
Vertices: $\mbs x_A=(0.25, 0.25)$, $\mbs x_B=(2.5, 0.0)$,
$\mbs x_C=(2.0, 2.5)$, $\mbs x_D=(0.0, 1.5)$.

\noindent Edges:
\begin{subequations}
\begin{align}
  & \mbs x_{AB}(\xi) = \mbs x_A\phi_0(\xi) + \mbs x_B\phi_1(\xi) - (0,h(\xi)),
  \quad \xi\in[-1,1], \\ 
  & \mbs x_{BC}(\eta) = \mbs x_B\phi_0(\eta) + \mbs x_C(\eta) + (h(\eta),0),
  \quad \eta\in[-1,1], \\
  & \mbs x_{CD}(\xi) = \mbs x_D\phi_0(\xi) + \mbs x_C\phi_1(\xi) - (0,h(\xi)),
  \quad \xi\in[-1,1], \\
  & \mbs x_{AD}(\eta) = \mbs x_A\phi_0(\eta) + \mbs x_D\phi_1(\eta) - (h(\eta),0),
  \quad \eta\in[-1,1],
\end{align}
\end{subequations}
where $\phi_0(\xi)$ and $\phi_1(\xi)$ are defined in~\eqref{eq_7}, and
$h(\xi)=-0.15\left[1+\cos(\pi \xi) \right]$ for $\xi\in[-1,1]$.

\vspace{5pt}
\noindent\underline{Domain \#2:}\\
This domain is formed by a unit circle centered at $(1,0)$, subtracting
a second unit circle centered at $(2,0)$.

\vspace{3pt}
\noindent Vertices: $\mbs x_A=\left(\frac12, -\frac{\sqrt{3}}{2}\right)$,
$\mbs x_B=(1, 0)$,
$\mbs x_C=\left(\frac12, \frac{\sqrt{3}}{2}\right)$, $\mbs x_D=(0,0)$.

\noindent Edges:
\begin{subequations}
\begin{align}
  & \mbs x_{AB}(\xi) = (2,0) + (\cos\theta_{AB}(\xi), \sin\theta_{AB}(\xi)),
  \quad \theta_{AB}(\xi) = -\frac{2\pi}{3}\phi_0(\xi) -\pi\phi_1(\xi), 
  \quad \xi\in[-1,1]; \\ 
  & \mbs x_{BC}(\eta) = (2,0) + (\cos\theta_{BC}(\eta), \sin\theta_{BC}(\eta)),
  \quad \theta_{BC}(\eta) = \pi\phi_0(\eta) + \frac{2\pi}{3}\phi_1(\eta),
  \quad \eta\in[-1,1]; \\
  & \mbs x_{CD}(\xi) = (1,0) + (\cos\theta_{CD}(\xi), \sin\theta_{CD}(\xi)),
  \quad \theta_{CD}(\xi) = \pi \phi_0(\xi) + \frac{\pi}{3}\phi_1(\xi),
  \quad \xi\in[-1,1]; \\
  & \mbs x_{AD}(\eta) = (1,0) + (\cos\theta_{AD}(\eta), \sin\theta_{AD}(\eta)),
  \quad \theta_{AD}(\eta) = -\frac{\pi}{3}\phi_0(\eta) - \pi \phi_1(\eta),
  \quad \eta\in[-1,1].
\end{align}
\end{subequations}

\vspace{5pt}
\noindent\underline{Domain \#3:}

\vspace{3pt}
\noindent Vertices: $\mbs x_A=\left(-\frac{\sqrt{2}}{2}, -\frac{\sqrt{2}}{2}\right)$,
$\mbs x_B=(\frac{\sqrt{2}}{2}, -\frac{\sqrt{2}}{2})$,
$\mbs x_C=\left(\frac{\sqrt{2}}{2}, \frac{\sqrt{2}}{2}\right)$,
$\mbs x_D=(-\frac{\sqrt{2}}{2}, \frac{\sqrt{2}}{2})$.

\noindent Edges:
\begin{subequations}
\begin{align}
  & \mbs x_{AB}(\xi) = \left[1-\frac{a_{AB}}{2}\left(1+\cos(\pi \xi)\right) \right]
  (\cos\theta_{AB}(\xi), \sin\theta_{AB}(\xi)), \notag \\
  &\qquad 
  \quad \theta_{AB}(\xi) = -\frac{3\pi}{4}\phi_0(\xi) -\frac{\pi}{4}\phi_1(\xi),
  \quad a_{AB} = 0.25,
  \quad \xi\in[-1,1]; \\ 
  & \mbs x_{BC}(\eta) = \left[1-\frac{a_{BC}}{2}\left(1+\cos(\pi \eta)\right) \right]
  (\cos\theta_{BC}(\eta), \sin\theta_{BC}(\eta)), \notag \\
  &\qquad 
  \quad \theta_{BC}(\eta) = -\frac{\pi}{4}\phi_0(\eta) +\frac{\pi}{4}\phi_1(\eta),
  \quad a_{BC} = 0.4,
  \quad \eta\in[-1,1]; \\
  & \mbs x_{CD}(\xi) = \left[1-\frac{a_{CD}}{2}\left(1+\cos(\pi \xi)\right) \right]
  (\cos\theta_{CD}(\xi), \sin\theta_{CD}(\xi)), \notag \\
  &\qquad 
  \quad \theta_{CD}(\xi) = \frac{3\pi}{4}\phi_0(\xi) + \frac{\pi}{4}\phi_1(\xi),
  \quad a_{CD} = 0.3,
  \quad \xi\in[-1,1]; \\
  & \mbs x_{AD}(\eta) = (\cos\theta_{AD}(\eta), \sin\theta_{AD}(\eta)),
  \quad \theta_{AD}(\eta) = \frac{5\pi}{4}\phi_0(\eta) + \frac{3\pi}{4} \phi_1(\eta),
  \quad \eta\in[-1,1].
\end{align}
\end{subequations}

\vspace{5pt}
\noindent\underline{Domain \#4:}\\
This domain is formed by two straight sides ($\overline{AB}$ and $\overline{AD}$)
and an elliptic arc ($\overline{BCD}$).

\vspace{3pt}
\noindent Vertices:
\begin{equation} \left\{
  \begin{split}
&\mbs x_A=\left(0,-1.35\right), \quad
\mbs x_B=\left(0.95\cos\left(\frac{5\pi}{36} \right), 0.55+0.6\sin\left(\frac{5\pi}{36} \right) \right), \\
& \mbs x_C=\left(0,1.15\right), \quad
\mbs x_D=\left(-0.95\cos\left(\frac{5\pi}{36} \right), 0.55+0.6\sin\left(\frac{5\pi}{36} \right) \right).
\end{split}
\right.
\end{equation}

\noindent Edges:
\begin{subequations}
\begin{align}
  & \mbs x_{AB}(\xi) = \mbs x_A\phi_0(\xi) + \mbs x_B\phi_1(\xi),
  \quad \xi\in[-1,1]; \\ 
  & \mbs x_{BC}(\eta) = (0.95\cos\theta_{BC}(\eta), 0.55+0.6\sin\theta_{BC}(\eta)), \
  \theta_{BC}(\eta) = \frac{5\pi}{36}\phi_0(\eta) +\frac{\pi}{2}\phi_1(\eta),
  \ \eta\in[-1,1]; \\
  & \mbs x_{CD}(\xi) = (0.95\cos\theta_{CD}(\xi), 0.55+0.6\sin\theta_{CD}(\xi)), \
  \theta_{CD}(\xi) = \frac{31\pi}{36}\phi_0(\xi) + \frac{\pi}{2}\phi_1(\xi),
  \ \xi\in[-1,1]; \\
  & \mbs x_{AD}(\eta) = \mbs x_A\phi_0(\eta) + \mbs x_D\phi_1(\eta),
  \quad \eta\in[-1,1].
\end{align}
\end{subequations}

\vspace{5pt}
\noindent\underline{Domain \#5:}\\
This is a triangle with vertices at $A$, $B$ and $D$. $C$ is the mid-point
of $\overline{BD}$.

\vspace{3pt}
\noindent Vertices: $\mbs x_A=(0,0)$,
$\mbs x_B=(2,0.2)$,
$\mbs x_C=(1.3, 1)$,
$\mbs x_D=(0.6,1.8)$.

\noindent Edges:
\begin{subequations}
\begin{align}
  & \mbs x_{AB}(\xi) = \mbs x_A\phi_0(\xi) + \mbs x_B\phi_1(\xi),
  \quad \xi\in[-1,1]; \\ 
  & \mbs x_{BC}(\eta) = \mbs x_B\phi_0(\eta) + \mbs x_C\phi_1(\eta),
  \quad \eta\in[-1,1]; \\
  & \mbs x_{CD}(\xi) = \mbs x_D\phi_0(\xi) + \mbs x_C\phi_1(\xi),
  \quad \xi\in[-1,1]; \\
  & \mbs x_{AD}(\eta) = \mbs x_A\phi_0(\eta) + \mbs x_D\phi_1(\eta),
  \quad \eta\in[-1,1].
\end{align}
\end{subequations}

\vspace{10pt}
\noindent\underline{\bf Section~\ref{sec_32}: Nonlinear Helmholtz Equation}

\vspace{8pt}
\noindent\underline{Domain \#1:}

\noindent Vertices:
\begin{equation}\left\{
  \begin{split}
&\mbs x_A=(-0.25,0.25), \quad
\mbs x_B=\left(\cos\left(\frac{19\pi}{20} \right)-0.25,\sin\left(\frac{19\pi}{20} \right)-0.75 \right), \\
& \mbs x_C=(-0.25,-1.75), \quad
\mbs x_D=\left(\cos\left(\frac{\pi}{20} \right)-0.25,\sin\left(\frac{\pi}{20} \right)-0.75 \right).
\end{split}
\right.
\end{equation}

\noindent Edges:
\begin{subequations}
\begin{align}
  & \mbs x_{AB}(\xi) = (-0.25,-0.75) +  \left(\cos\theta_{AB}(\xi), \sin\theta_{AB}(\xi) \right), \notag \\
  &\qquad\quad \theta_{AB}(\xi) = \frac{\pi}{2}\phi_0(\xi) + \frac{19\pi}{20}\phi_1(\xi),
  \quad \xi\in[-1,1]; \\ 
  & \mbs x_{BC}(\eta) = 2\left[
    \left(\cos\left(\frac{19\pi}{20} \right),\sin\left(\frac{19\pi}{20} \right) \right)
    \phi_0(\eta) + (0,-1)\phi_1(\eta)
    \right] - \left(\cos\theta_{BC}(\eta),\sin\theta_{BC}(\eta) \right) \notag \\
  &\qquad\qquad + (-0.25,-0.75),
  \quad \theta_{BC}(\eta) = \frac{19\pi}{20}\phi_0(\eta) + \frac{3\pi}{2}\phi_1(\eta),
  \quad \eta\in[-1,1]; \\
  & \mbs x_{CD}(\xi) = 2\left[
    \left(\cos\left(\frac{\pi}{20} \right),\sin\left(\frac{\pi}{20} \right) \right)
    \phi_0(\xi) + (0,-1)\phi_1(\xi)
    \right] - \left(\cos\theta_{CD}(\xi),\sin\theta_{CD}(\xi) \right) \notag \\
  &\qquad\qquad + (-0.25,-0.75),
  \quad \theta_{CD}(\xi) = \frac{\pi}{20}\phi_0(\xi) - \frac{\pi}{2}\phi_1(\xi),
  \quad \xi\in[-1,1]; \\
  & \mbs x_{AD}(\eta) = (-0.25,-0.75) +  \left(\cos\theta_{AD}(\eta), \sin\theta_{AD}(\eta) \right), \notag \\
  &\qquad\quad \theta_{AD}(\eta) = \frac{\pi}{2}\phi_0(\eta) + \frac{\pi}{20}\phi_1(\eta),
  \quad \eta\in[-1,1].
\end{align}
\end{subequations}

\vspace{5pt}
\noindent\underline{Domain \#2:}

\noindent Vertices:
$\mbs x_A=\left(\frac{\sqrt{2}}{2}, -\frac{\sqrt{2}}{2}\right)$,
$\mbs x_B=\left(\frac{\sqrt{2}}{2}, \frac{\sqrt{2}}{2}\right)$,
$\mbs x_C=\left(-\frac{\sqrt{2}}{2}, \frac{\sqrt{2}}{2}\right)$,
$\mbs x_D=\left(-\frac{\sqrt{2}}{2}, -\frac{\sqrt{2}}{2}\right)$.

\noindent Edges:
\begin{subequations}
\begin{align}
  & \mbs x_{AB}(\xi) = \left[1+0.6\cos(2\theta_{AB}(\xi))\right]\left(\cos\theta_{AB}(\xi),\sin\theta_{AB}(\xi)\right), \notag \\
  &\qquad\quad \theta_{AB}(\xi) = -\frac{\pi}{4}\phi_0(\xi) + \frac{\pi}{4}\phi_1(\xi),
  \quad \xi\in[-1,1]; \\ 
  & \mbs x_{BC}(\eta) = \left[1+0.6\cos(2\theta_{BC}(\eta))\right]\left(\cos\theta_{BC}(\eta),\sin\theta_{BC}(\eta)\right), \notag \\
  &\qquad\quad \theta_{BC}(\eta) = \frac{\pi}{4}\phi_0(\eta) + \frac{3\pi}{4}\phi_1(\eta),
  \quad \eta\in[-1,1]; \\
  & \mbs x_{CD}(\xi) = \left[1+0.6\cos(2\theta_{CD}(\xi))\right]\left(\cos\theta_{CD}(\xi),\sin\theta_{CD}(\xi)\right), \notag \\
  &\qquad\quad \theta_{CD}(\xi) = \frac{5\pi}{4}\phi_0(\xi) + \frac{3\pi}{4}\phi_1(\xi),
  \quad \xi\in[-1,1]; \\
  & \mbs x_{AD}(\eta) = \left[1+0.6\cos(2\theta_{AD}(\eta))\right]\left(\cos\theta_{AD}(\eta),\sin\theta_{AD}(\eta)\right), \notag \\
  &\qquad\quad \theta_{AD}(\eta) = -\frac{3\pi}{4}\phi_0(\eta) - \frac{\pi}{4}\phi_1(\eta),
  \quad \eta\in[-1,1].
\end{align}
\end{subequations}

\vspace{5pt}
\noindent\underline{Domain \#3:}

\noindent Vertices:
$\mbs x_A=\left(1.2,0\right)$,
$\mbs x_B=(0,1.2)$,
$\mbs x_C=\left(-1.2,0\right)$,
$\mbs x_D=(0,-1.2)$.

\noindent Edges:
\begin{subequations}
\begin{align}
  & \mbs x_{AB}(\xi) = \left[0.8+0.4\cos(4\theta_{AB}(\xi))\right]\left(\cos\theta_{AB}(\xi),\sin\theta_{AB}(\xi)\right), \notag \\
  &\qquad\quad \theta_{AB}(\xi) =  \frac{\pi}{2}\phi_1(\xi),
  \quad \xi\in[-1,1]; \\ 
  & \mbs x_{BC}(\eta) = \left[0.8+0.4\cos(4\theta_{BC}(\eta))\right]\left(\cos\theta_{BC}(\eta),\sin\theta_{BC}(\eta)\right), \notag \\
  &\qquad\quad \theta_{BC}(\eta) = \frac{\pi}{2}\phi_0(\eta) + \pi\phi_1(\eta),
  \quad \eta\in[-1,1]; \\
  & \mbs x_{CD}(\xi) = \left[0.8+0.4\cos(4\theta_{CD}(\xi))\right]\left(\cos\theta_{CD}(\xi),\sin\theta_{CD}(\xi)\right), \notag \\
  &\qquad\quad \theta_{CD}(\xi) = \frac{3\pi}{2}\phi_0(\xi) + \pi\phi_1(\xi),
  \quad \xi\in[-1,1]; \\
  & \mbs x_{AD}(\eta) = \left[0.8+0.4\cos(4\theta_{AD}(\eta))\right]\left(\cos\theta_{AD}(\eta),\sin\theta_{AD}(\eta)\right), \notag \\
  &\qquad\quad \theta_{AD}(\eta) = - \frac{\pi}{2}\phi_1(\eta),
  \quad \eta\in[-1,1].
\end{align}
\end{subequations}

\vspace{5pt}
\noindent\underline{Domain \#4:}

\noindent Vertices:
\begin{equation}\left\{
  \begin{split}
&\mbs x_A=\left(0.75+0.3\cos(3\theta_A) \right)\left(\cos\theta_A,\sin\theta_A \right), \quad
\mbs x_B= \left(0.75+0.3\cos(3\theta_B) \right)\left(\cos\theta_B,\sin\theta_B \right), \\
& \mbs x_C=\left(0.75+0.3\cos(3\theta_C) \right)\left(\cos\theta_C,\sin\theta_C \right),  \quad
\mbs x_D= \left(0.75+0.3\cos(3\theta_D) \right)\left(\cos\theta_D,\sin\theta_D \right),\\
& \theta_A=-\frac{\pi}{10}, \quad \theta_B=\frac{\pi}{10}, \quad
\theta_C = \frac{13\pi}{20}, \quad \theta_D=\frac{27\pi}{20}.
\end{split}
\right.
\end{equation}

\noindent Edges:
\begin{subequations}
\begin{align}
  & \mbs x_{AB}(\xi) = \left[0.75+0.3\cos(3\theta_{AB}(\xi))\right]\left(\cos\theta_{AB}(\xi),\sin\theta_{AB}(\xi)\right), \notag \\
  &\qquad\quad \theta_{AB}(\xi) = \theta_A\phi_0(\xi) + \theta_B\phi_1(\xi),
  \quad \xi\in[-1,1]; \\ 
  & \mbs x_{BC}(\eta) = \left[0.75+0.3\cos(3\theta_{BC}(\eta))\right]\left(\cos\theta_{BC}(\eta),\sin\theta_{BC}(\eta)\right), \notag \\
  &\qquad\quad \theta_{BC}(\eta) = \theta_B\phi_0(\eta) + \theta_C\phi_1(\eta),
  \quad \eta\in[-1,1]; \\
  & \mbs x_{CD}(\xi) = \left[0.75+0.3\cos(3\theta_{CD}(\xi))\right]\left(\cos\theta_{CD}(\xi),\sin\theta_{CD}(\xi)\right), \notag \\
  &\qquad\quad \theta_{CD}(\xi) = \theta_D\phi_0(\xi) + \theta_C\phi_1(\xi),
  \quad \xi\in[-1,1]; \\
  & \mbs x_{AD}(\eta) = \left[0.75+0.3\cos(3\theta_{AD}(\eta))\right]\left(\cos\theta_{AD}(\eta),\sin\theta_{AD}(\eta)\right), \notag \\
  &\qquad\quad \theta_{AD}(\eta) = \theta_A\phi_0(\eta) + (\theta_D-2\pi)\phi_1(\eta),
  \quad \eta\in[-1,1].
\end{align}
\end{subequations}

\bibliographystyle{plain}
\bibliography{reference,bc,pinn_bc,elm,elm1,elm2,dnn,dnn1,dnn2,dnn3,dnn4,ref1,ml}

\end{document}